\documentclass[11pt, leqno]{amsart}
\usepackage{amsfonts,amssymb, amscd}
\usepackage{amsmath}
\usepackage{lmodern}
\usepackage{mathrsfs} 
\usepackage{amsthm}
\usepackage{qsymbols}
\usepackage{graphicx}
\usepackage{latexsym}
\usepackage{chngcntr}
\usepackage[noadjust]{cite}
\usepackage{paralist}
\usepackage[parfill]{parskip}
\def\theenumi{\roman{enumi}}
\usepackage{enumitem}
\newtheorem{theorem}{Theorem}[section]
\newtheorem{lemma}[theorem]{Lemma}
\newtheorem{proposition}[theorem]{Proposition}
\newtheorem{corollary}[theorem]{Corollary}
\newtheorem{taggedtheoremx}{Assumption}
\newenvironment{assumption}[1]
 {\renewcommand\thetaggedtheoremx{#1}\taggedtheoremx}
 {\endtaggedtheoremx}
\theoremstyle{definition}
\newtheorem{definition}[theorem]{Definition}
\newtheorem{remark}[theorem]{Remark}
\numberwithin{equation}{section}
\usepackage[plainpages=false,pdfpagelabels,citecolor=red]{hyperref}
\usepackage{xcolor}
\hypersetup{
 colorlinks,
 linkcolor={cyan!90!black},
 citecolor={magenta},
 urlcolor={green!40!black}
} %important to load after amsthm and before the rest
\newcommand{\R}{\mathbb{R}}
\newcommand{\IC}{\mathbb{C}}
\newcommand{\IN}{\mathbb{N}}
\newcommand{\IZ}{\mathbb{Z}}

\newcommand{\cH}{\mathcal{H}}
\newcommand{\cQ}{\mathcal{Q}}
\newcommand{\cJ}{\mathcal{J}}

\newcommand{\cS}{\mathcal{S}}
\newcommand{\cT}{\mathcal{T}}
\newcommand{\cA}{\mathcal{A}}

\newcommand{\cN}{\mathcal{N}}
\renewcommand{\L}{\mathrm{L}}
\newcommand{\C}{\mathrm{C}}
\newcommand{\B}{\mathrm{B}}
\renewcommand{\H}{\mathrm{H}}
\renewcommand{\S}{\mathrm{S}}
\newcommand{\W}{\mathrm{W}}

\newcommand{\IW}{\mathbb{W}_D}
\newcommand{\fa}{a}
\newcommand{\Jnorm}[1]{\llbracket #1 \rrbracket}

\newcommand{\fabs}[1]{\left|#1\right|}
\newcommand{\scal}[2]{(#1 \mathrel \vert \mathrel{} #2)}%Scalar products, angle
\newcommand{\dscal}[2]{\langle #1 \mathrel \vert \mathrel{} #2 \rangle}

\newcommand{\ind}{{\mathbf{1}}}
\newcommand{\e}{\mathrm{e}}
\let\ii\i
\renewcommand{\i}{\mathrm{i}}
\renewcommand{\d}{\mathrm{d}}
\newcommand{\eps}{\varepsilon}

\renewcommand\Re{\operatorname{Re}}

\newcommand{\tildu}{\widetilde{u}}
\newcommand{\tildb}{\widetilde{b}}
\newcommand{\tildg}{\widetilde{g}}
\newcommand{\pdx}[2]{\frac{\partial #1}{\partial x_{#2}}}
\newcommand{\Lop}{\mathcal{L}}
\newcommand{\cl}[1]{\overline{#1}}
\renewcommand\div{\operatorname{div}}
\DeclareMathOperator{\bd}{\partial}
\DeclareMathOperator{\supp}{supp}
\DeclareMathOperator{\essup}{essup}
\DeclareMathOperator{\dist}{d}
\DeclareMathOperator{\diam}{diam}

\DeclareMathOperator{\Rg}{\mathcal{R}}

\DeclareMathOperator{\dom}{\mathcal{D}}

\newcommand{\Lloc}{\L_{\mathrm{loc}}}
\newcommand{\pe}{\perp}
\newcommand{\pa}{\parallel}
\newcommand{\note}[1]{{\color{black} #1}}
\def\Xint#1{\mathchoice
{\XXint\displaystyle\textstyle{#1}}%
{\XXint\textstyle\scriptstyle{#1}}%
{\XXint\scriptstyle\scriptscriptstyle{#1}}%
{\XXint\scriptscriptstyle%
\scriptscriptstyle{#1}}%
\!\int}
\def\XXint#1#2#3{{\setbox0=\hbox{$#1{#2#3}{%
\int}$ }
\vcenter{\hbox{$#2#3$ }}\kern-.6\wd0}}
\def\barint{\,\Xint-} %\, corrects the \! used in the definition
\title[Square root of elliptic systems with mixed boundary conditions] {$\L^p$-estimates for the square root of elliptic systems with mixed boundary conditions}
\author{Moritz Egert}
\address{Laboratoire de Math\'{e}matiques d'Orsay, Univ.\ Paris-Sud, CNRS, Universit\'{e} Paris-Saclay, 91405 Orsay, France}
\email{moritz.egert@math.u-psud.fr}
\date{\today}
\dedicatory{}
\thanks{}
\begin{document}
\begin{abstract}
This article focuses on $\L^p$-estimates for the square root of elliptic systems of second order in divergence form on a bounded domain. We treat complex bounded measurable coefficients and allow for mixed Dirichlet/Neumann boundary conditions on domains beyond the Lipschitz class. If there is an associated bounded semigroup on $\L^{p_0}$, then we prove that the square root extends for all $p \in (p_0,2)$ to an isomorphism between a closed subspace of $\W^{1,p}$ carrying the boundary conditions and $\L^p$. This result is sharp and extrapolates to exponents slightly above $2$. As a byproduct, we obtain an optimal $p$-interval for the bounded $\H^\infty$-calculus on $\L^p$. Estimates depend holomorphically on the coefficients, thereby making them applicable to questions of (non-autonomous) maximal regularity and optimal control. For completeness we also provide a short summary on the Kato square root problem in $\L^2$ for systems with lower order terms in our setting.
\end{abstract}
\maketitle
%%%%%%%%%%%%%%%%%%%%%%%%%%%%%%%%%%%%%%%%%%%%%%%%%%%%%%%%%%%%%%%%%%%%%%%%%%%%%%%%%%%%%%%%%%%%%%%%%%%%%%%%%%%%%%%%%%%%%%%%%%%%%%%%%%%%%%%%%%%%%%%%%%%%%%%%%%%%%%%%%%%%
\section{Introduction and main results}
\label{Sec: Introduction}

Elliptic divergence form operators are amongst the most carefully studied differential operators with variable coefficients. In this paper we contribute to the functional calculus of such operators with complex, bounded and measurable coefficients, formally given by
\begin{align}
\label{L}
 Lu = - \sum_{i,j=1}^d \pdx{}{i}\Big(a_{ij} \pdx{u}{j} \Big) - \sum_{i=1}^d \pdx{}{i}(a_{i0} u) + \sum_{j=1}^d a_{0j} \pdx{u}{j} + a_{00}u,
\end{align}
on a bounded domain $\Omega \subseteq \R^d$, $d \geq 2$. We allow for mixed boundary conditions. Namely, $u$ satisfies homogeneous Dirichlet boundary conditions on a closed part $D$ of the boundary and natural boundary conditions on the complementary part $N = \bd \Omega \setminus D$. The geometric constellation can be `rough' in that we require Lipschitz coordinate charts for $\bd \Omega$ only around the closure of $N$, whereas around $D$ the domain $\Omega$ can merely be $d$-Ahlfors regular, and $D$ itself has to be $(d-1)$-Ahlfors regular. These notions, henceforth called Assumptions~\ref{N}, \ref{Omega}, and \ref{D}, will be recalled in Section~\ref{Subsec: Geometry}. We include $(m \times m)$-systems in our considerations, that is to say, $u$ takes its values in $\IC^m$ and each $a_{ij}$ is valued in the space of matrices $\Lop(\IC^m)$. As in \cite{HJKR}, we may have different Dirichlet boundary parts for each coordinate of $u$. These assumptions are amongst the most general ones that allow for a proper functional analytic framework for $L$~\cite{HJKR, ABHR, RobertJDE, Brewster}.

As usual, we interpret $L$ in the weak sense via the sesquilinear form 
\begin{align}
\label{a}
\fa(u,v) = \int_\Omega \sum_{i,j=1}^d \bigg(a_{ij} \pdx{u}{j} \cdot \cl{\pdx{v}{i}} + a_{i0} u \cdot \cl{\pdx{v}{i}} + a_{0j} \pdx{u}{j} \cdot \cl{v} + a_{00}u \cdot \cl{v} \bigg) \; \d x
\end{align}
defined on $\dom(\fa) = \IW^{1,2}(\Omega)$, where the subscripted $D$ is reminiscent of the boundary conditions. Ellipticity is in the sense of a G\aa{}rding inequality, turning $L$ into a maximal accretive operator on $\L^2(\Omega)^m$. This way of understanding $L$ is called `Kato's form method'. Definitions are provided in Section~\ref{Subsec: L} and we refer to \cite{Kato, Ouhabaz} for the general background. Let us stress that our setup incorporates, for example, the Lam\'e system. We shall come back to that.

The focus in this paper lies on establishing $\L^p$-estimates for the (unique) maximal accretive square root $L^{1/2}$ of $L$. More precisely, we study for which $p \in (1,\infty)$ it extends or restricts to a topological isomorphism
\begin{align}
\label{Eq: General Isom}
 L^{1/2} : \IW^{1,p}(\Omega) \xrightarrow{\; \cong \;} \L^p(\Omega)^m.
\end{align}
Our results are the first of this kind for `rough' divergence form systems on domains and provide optimal ranges of exponents.

Recent years have witnessed a vast number of applications of property \eqref{Eq: General Isom}. It is key to the approach of Rehberg and collaborators to quasilinear parabolic equations on distribution spaces via maximal regularity techniques originating from~\cite{RobertJDE} and its extensions for example to optimal control problems~\cite{Bonifacius-Neitzel} and quasilinear stochastic evolution equations~\cite{Hornung-Weis}, as well as recent progress on maximal regularity for the non-autonomous Cauchy problem on Lebesgue spaces \cite{Fackler-LpFrac, Fackler-LpRough} and distribution spaces \cite{Disser-terElst-Rehberg}, see also \cite{ADLO13, Dier-Zacher, Achache} for the case $p=2$. Aiming in a slightly different direction, \cite{Disser-terElst-Rehberg2} uses property \eqref{Eq: General Isom} to prove H\"older continuity of solutions to quasilinear parabolic equations in rough domains.

The common idea in all of these applications is that \eqref{Eq: General Isom} allows to switch between Lebesgue spaces and Sobolev spaces as well as their duals by means of an isomorphism that is build from $L$ itself and hence commutes with the latter. This allows one to transfer knowledge between any two of these spaces. Let us give two illustrating examples. If $L$ corresponds to an equation ($m=1$) with real coefficients, then $L$ has a bounded $\H^\infty$-calculus and hence maximal regularity on $\L^p$ for any $p \in (1,\infty)$ due to Gaussian estimates \cite{Duong-Robinson, Arendt-terElst}, and one asks for the same on $\W^{-1,p'}$ to treat distributional right-hand sides in quasilinear equations. In the realm of non-autonomous Cauchy problems, an old result of Lions guarantees non-autonomous maximal regularity on $\W^{-1,2}$ but one wants to transfer this knowledge to $\L^2$ for example to make sense of boundary conditions \cite{ADLO13, Dier-Zacher}. We stress that $L$ itself cannot play the role of this transference operator because in general $\dom(L)$ is not a Sobolev space of order two~\cite{Shamir-Counterexample}.

In the Hilbert space case $p=2$, having \eqref{Eq: General Isom} means having $\dom(L^{1/2}) = \dom(\fa)$ with equivalent norms. If $\fa$ is symmetric -- which here amounts to $a_{ij} = a_{ji}^*$ for all $i,j$ -- this is a property of closed densely defined sectorial forms that has nothing to do with differential operators~\cite[Ch.~VI]{Kato}. The case of non-symmetric forms has a long history and became known as \emph{Kato square root problem}, see for example \cite{Kato-Square-Root-Proof, Asterisque, AT2}. Within our setup it has been settled in \cite{Darmstadt-KatoMixedBoundary, Laplace-Extrapolation} by a non-trivial refinement of the first-order method of Axelsson--Keith--McIntosh proposed in their seminal paper \cite{AKM-QuadraticEstimates} and their pioneering application to mixed boundary value problems in \cite{AKM}. The author has to concede that it is somewhat unfortunate that \cite{Darmstadt-KatoMixedBoundary} and  \cite{Laplace-Extrapolation} treat systems with lower-order terms only implicitly. He sees this paper as the right moment to close this gap and shortly review the underlying methods in Section~\ref{Sec: L2} to prove the following

\begin{theorem}
\label{Thm: Kato}
Suppose $\Omega \subseteq \R^d$ is a bounded domain that satisfies Assumptions~\ref{D}, \ref{N}, and \ref{Omega}. Then $\dom(L^{1/2}) = \dom(a) = \IW^{1,2}(\Omega)$ and 
\begin{align*}
\|L^{1/2} u\|_2^2 \simeq a(u,u) \simeq \|u\|_2^2  + \|\nabla u\|_2^2 \qquad (u \in \IW^{1,2}(\Omega))
\end{align*}
with implicit constants depending on ellipticity, dimensions, and geometry.
\end{theorem}

Another possibility would have been to adapt the perturbation argument of \cite[Ch.~0.7]{Asterisque}. Here, $d$ and the number of equations $m$ in \eqref{L} are referred to as \emph{dimensions}. The constants $\lambda$ and $\Lambda$ are the lower and upper bounds for the sesquilinear form in \eqref{a} and will be defined in Section~\ref{Subsec: L}. They are referred to as \emph{ellipticity}. \emph{Geometry} refers to all constants implicit in those assumptions amongst \ref{D}, \ref{N}, and \ref{Omega} that are used in the particular situation.

The literature with regard to $p \neq 2$ is much sparser. Pure Dirichlet and pure Neumann boundary conditions on a Lipschitz domain were first treated in \cite{ATp} under size and regularity assumptions on the kernel of the semigroup generated by $-L$. This applies in particular to equations with real coefficients. Auscher--Badr--Haller-Dintelmann--Rehberg~\cite{ABHR} have more recently established \eqref{Eq: General Isom} in the range $p \in (1,2)$ also for mixed boundary conditions on making the same geometric assumptions \ref{D}, \ref{N}, and \ref{Omega}. There are, however, no competing results available -- even on more regular domains -- if one is to consider mixed boundary conditions for operators with complex coefficients, let alone systems. The only exception is the case $d=1=m$, where \eqref{Eq: General Isom} is known to hold for all $p \in (1,\infty)$ and all common two-point boundary conditions on an open interval~\cite{Auscher-McIntosh-Nahmod}. 

A coherent treatment of $\L^p$-estimates for the square root of elliptic systems with complex coefficients on $\R^d$ when $d \geq 2$ is found in the monograph~\cite{Riesz-Transforms}. On a large scale our approach is to incorporate the machinery of off-diagonal estimate from \cite{Riesz-Transforms} into the fine geometric setup of \cite{ABHR} to compensate for the lack of Gaussian upper bounds for the kernel of the semigroup. (In fact, there might not be a kernel in any but the distributional sense.) This was left as an open problem on p.66 in \cite{Riesz-Transforms}. 

We shall work in an $L$-adapted range of exponents
\begin{align}
\label{JL}
 \cJ(L) := \Big \{p \in (1,\infty): \sup_{t > 0} \|\e^{-tL}\|_{\L^p \to \L^p} < \infty \Big \},
\end{align}
that is to say, those $p \in (1,\infty)$ for which there is a bounded (strongly continuous) semigroup associated with $-L$ on $\L^p(\Omega)^m$. We postpone a detailed discussion of the size of $\cJ(L)$ and mention for the moment only that $\cJ(L)$ is an interval that contains at least $2_*:= 2d/(d+2)$ and $2^*:=2d/(d-2)$. An obvious advantage of working with $\cJ(L)$ is that any improvement on $\L^p$ boundedness of the semigroup entails an improvement in our results for free. 

As our main result we obtain \eqref{Eq: General Isom} in the best possible range below $2$ (except for maybe the endpoints). For the sake of clarity let us introduce the array
\begin{align}
\label{Jnorm}
 \Jnorm{p} := \Big[d,m,\lambda, \Lambda, \sup_{t > 0} \|\e^{-tL}\|_{\L^p \to \L^p} \Big]
\end{align}
to store all the constants that usually show up in our estimates.

\begin{theorem}
\label{Thm: Square root solution Lp}
Suppose $\Omega \subseteq \R^d$ is a bounded domain satisfying \ref{D}, \ref{N}, and \ref{Omega}.

\begin{enumerate}
 \item \label{i Square root solution Lp} If $p_0 \in \cJ(L) \cap [1,2)$, then $L^{1/2}$ extends to an isomorphism $\IW^{1,p}(\Omega) \to \L^p(\Omega)^m$ for every $p \in (p_0,2)$. Upper and lower bounds of the extension depend on $\Jnorm{p_0}$, $p$, and geometry.
 \item \label{ii Square root solution Lp} The result in \eqref{i Square root solution Lp} is optimal in that if $L^{1/2}$ extends to an isomorphism $\IW^{1,p}(\Omega) \to \L^p(\Omega)^m$ for some $p \in [1,2)$, then already $(p,2) \subseteq \cJ(L)$.
\end{enumerate}
\end{theorem}

We note that Theorem~\ref{Thm: Square root solution Lp}.\eqref{i Square root solution Lp} in principle consists of two estimates. One is the $\L^p$ boundedness of the Riesz transform $\nabla L^{-1/2}$ expressed in the domination
\begin{align*}
 \|\nabla u\|_p \lesssim \|L^{1/2} u\|_p \qquad (u \in \dom(L^{1/2})).
\end{align*}
In Section~\ref{Sec: Riesz} we present details of the (short) proof relying on Blunck--Kunstmann's weak type criterion from \cite{BK, BK2}, see also \cite{Riesz-Transforms}, and $\L^p \to \L^2$ off-diagonal estimates for the gradient of the semigroup, to be established in Section~\ref{Sec: Properties of the semigroup}. This is only for the reader's convenience and to avoid overloading this article with indications on how to modify and extract additional information from \cite{BK2, Riesz-Transforms}. Indeed, these references both treat the case $\Omega = \R^d$ only and do not state an explicit dependence of implicit constants. But we do not claim much originality here. An interesting observation to keep in mind, however, is that for this part we shall not require Assumption~\ref{Omega}, that is, $\Omega$ with the restricted Lebesgue measure need not be a space of homogeneous type.

The other ingredient is the \emph{a priori} inequality
\begin{align*}
 \|L^{1/2} u\|_p \lesssim \|u\|_p + \|\nabla u \|_p \qquad (u \in \IW^{1,2}(\Omega)).
\end{align*}
This will be handled in Section~\ref{Sec: Reverse estimates for the Riesz transform} by a careful modification of the main argument in \cite{ABHR}. It goes by a weak-type criterion and requires a new Calder\'on-Zygmund decomposition that we shall establish beforehand in Section~\ref{Sec: CZ decomposition}.

We remind the reader that most aforementioned applications of \eqref{Eq: General Isom} would invest this isomorphism along with maximal regularity on $\L^p$. The latter is not known \emph{a priori} in our context but there is a way out: By Dore--Venni's theorem~\cite{Dore-Venni} maximal regularity on $\L^p$ follows from the bounded $\H^\infty$-calculus on $\L^p$ and this in turn is an easy consequence of methods used to get a grip on the Riesz transform, namely $\L^p \to \L^2$ off-diagonal estimates. Again, this connection is not new in general but has not been exploited in our context. It goes back to the seminal contribution \cite{BK} of Blunck and Kunstmann.

\begin{theorem}
\label{Thm: Hinfty in Lp}
Suppose $\Omega \subseteq \R^d$ is a bounded domain that satisfies Assumption~\ref{N} and let $p_0 \in \cJ(L)$. If $f$ is bounded and holomorphic on the sector $\S^+_\psi$, where $\psi \in (\omega, \pi)$, then for all $p \in (p_0,2) \cup (2, p_0)$,
\begin{align}
\label{Eq: Hinfty in Lp}
 \|f(L)u\|_p \leq C \|f\|_\infty \|u\|_p \qquad (u \in \L^2(\Omega)^m \cap \L^p(\Omega)^m).
\end{align}
Moreover, $C$ depends on $\psi$, $\Jnorm{p_0}$, $p$, and geometry. 

The above range for $p$ is optimal in that \eqref{Eq: Hinfty in Lp} for some $p \in (1,\infty)$, some $\psi \in (\omega, \pi)$, and all $f$ as above implies $p \in \cJ(L)$.
\end{theorem}

Above, $\omega \in [0, \pi/2)$ is the angle of accretivity for $L$ and $\S^+_\psi$ is the open sector in the complex plane of opening angle $2\psi$ symmetric around the positive real axis, see Section~\ref{Subsec: L}.

As far as $L^{1/2}$ is concerned, we have only dealt with exponents $p<2$ but a duality argument allows us to extrapolate \eqref{Eq: General Isom} to some exponents above $2$. In Section~\ref{Sec: Extended square root solution Lp} we present a proof of

\begin{theorem}
\label{Thm: Extended square root solution Lp}
Suppose $\Omega \subseteq \R^d$ is a bounded domain satisfying Assumptions~\ref{D}, \ref{N}, and \ref{Omega}. There exists $\eps(L) > 0$ such that $L^{1/2}$ restricts to an isomorphism $\IW^{1,p}(\Omega) \to \L^p(\Omega)^m$ for every $p \in [2,2+\eps(L))$. Moreover, $\eps(L)$ depends on ellipticity, dimensions, and geometry. Upper and lower bounds of the restriction import at most an additional dependence on $p$.
\end{theorem}

We did not try to find a characterisation of the admissible exponents $p>2$ in \eqref{Eq: General Isom} in terms of the semigroup (or rather its gradient $\sqrt{t} \nabla \e^{-tL}$) as in \cite{Riesz-Transforms}. This is left as an independent open problem. However, already for the Riesz transform no exponent $p>2$ works simultaneously for all real symmetric $L$ subject to Dirichlet boundary conditions on a Lipschitz domain. This follows from combining the example on p.120 in \cite{Asterisque} with \cite[Thm.~B]{Shen-Riesz}.

The reader will have already noticed that we take special care of implicit constants. This is \emph{not} because we are trying to be particularly pedantic or even annoying. Rather than that, we need to prepare our results for the aforementioned applications to non-autonomous parabolic problems, where a family of operators $L_t$ with uniform ellipticity parameters in $t$ acts in the spatial variables, and one needs the same uniformity on estimates for $L_t^{1/2}$. This was asked for in \cite{Fackler-LpFrac, Fackler-LpFrac, Hornung-Weis, Bonifacius-Neitzel}. It is only implicit in \cite{ABHR} and sometimes all but impossible to track. We shall comment on that issue in Section~\ref{Sec: Reverse estimates for the Riesz transform}. There is also an `inner' application for having uniform constants with respect to ellipticity, as it allows us to obtain holomorphic dependence on the coefficients in all of our estimates by a simple application of Vitali's theorem from complex analysis. To this end, let us denote by $\cA(\Omega)$ the set of coefficient functions 
\begin{align*}
 A=(a_{i,j})_{0 \leq i,j \leq d}: \Omega \to \Lop(\IC^m)^{(d+1) \times (d+1)}
\end{align*}
that satisfy our ellipticity assumptions, which can canonically be identified with an open subset of $\L^\infty(\Omega)^{m^2(d+1)^2}$. We shall say that a function $h$ from some open subset of $\cA(\Omega)$ into a complex Banach space $X$ is holomorphic if for all holomorphic functions $\varphi$ from an open subset of $\IC$ into $\cA(\Omega)$ the composition $h \circ \varphi$ is holomorphic. Let us also write $L_A$ instead of $L$ and so on to stress the dependence on $A$.

\begin{theorem}
\label{Thm: Holomorphy}
Suppose $\Omega \subseteq \R^d$ is a bounded domain and let $O \subseteq \cA(\Omega)$ be open.
\begin{enumerate}
 \item\label{Holo1} If $\sup_{ A \in O} \sup_{t>0} \|\e^{-tL_A}\|_{\L^{p_0} \to \L^{p_0}} < \infty$ and $\psi> \sup_{A \in O} \omega(L_A)$, then in the setting and with the notation of Theorem~\ref{Thm: Hinfty in Lp} for every $f \in \H^\infty(\S^+_\psi)$ the function
 \begin{align*}
  O \to \Lop(\L^p(\Omega)^m), \quad A \mapsto f(L_A)
 \end{align*}
 is holomorphic on $O$.
 \item\label{Holo2} If $\sup_{ A \in O} \sup_{t>0} \|\e^{-tL_A}\|_{\L^{p_0} \to \L^{p_0}} < \infty$, then in the setting and with the notation of Theorem~\ref{Thm: Square root solution Lp} the function
 \begin{align*}
  O \to \Lop(\IW^{1,p}(\Omega), \L^p(\Omega)^m), \quad A \mapsto L_A^{1/2}
 \end{align*}
 is holomorphic on $O$. The same holds for $p \in \bigcap_{A \in O} [2,2+\eps(L_A))$ in the setting of Theorem~\ref{Thm: Extended square root solution Lp}.
\end{enumerate}
\end{theorem}

For divergence form operators on rough bounded domains this seems to be a new result. It will be proved in Section~\ref{Sec: Holomorphy Lp}. Let us remark that such kind of perturbation result usually necessitates the treatment of complex coefficients even if one aims at applying it in the realm of real equations only.

As for applications, all aforementioned results become more powerful the more is known about the set $\cJ(L)$. Riesz-Thorin interpolation reveals that $\cJ(L)$ is an interval and by maximal accretivity of $L$ on $\L^2(\Omega)^m$ it contains $2$. By means of Nash-type inequalities we shall prove in Section~\ref{Sec: Properties of the semigroup} the following

\begin{theorem}
\label{Thm: Range for Lp realization}
If the bounded domain $\Omega \subseteq \R^d$ satisfies Assumption~\ref{N}, then
\begin{align*}
 \cJ(L) \supseteq \begin{cases}
                   (1, \infty) &\text{if $d=2$} \\
                   (2_*,2^*) &\text{if $d \geq 3$}.
                  \end{cases}
\end{align*}
If in addition Assumption~\ref{D} is in power, then $(2_*, 2^*)$ can be replace with $(2_*-\eps, 2^* + \eps)$ if $d \geq 3$ and $\eps > 0$ depends on ellipticity, dimensions, and geometry. For $p$ in these sub-intervals the $\L^p$ bound for the semingroup depends on $p$, ellipticity, dimensions, and geometry.
\end{theorem}

We remind the reader that $2_*:= 2d/(d+2)$ and $2^*:=2d/(d-2)$ are the lower and upper Sobolev conjugates of $2$, respectively. The above result for $d \geq 3$ is sharp even for $m=1$ in the sense that no interval larger than $[2_*,2^*]$ is contained in $\cJ(L)$ for \emph{every} elliptic equation $L$ even with pure Dirichlet boundary conditions on a smooth and bounded domain. This follows from \cite[Prop.~2.20]{HMMc}, see also \cite[Thm.~5]{Davies-CE}: Indeed, for $p$ in the interior of $\cJ(L)$ the semigroup generated by $-L$ decays exponentially in $\L^p$-norm due to Lemma~\ref{Lem: Expontial stability} below and interpolation, hence $L^{-1}$ is bounded in $\L^p$-topology, but \cite{HMMc} constructs for every $p>2^*$ an example such  $L^{-1}$ does not even map $\C_0^\infty(\Omega)$ into $\L^p(\Omega)$. By duality we can argue likewise for $p<2_*$.

Let us also mention that many applications to physically motivated models require the adjoint isomorphism to \eqref{Eq: General Isom}, that is
\begin{align*}
 (L^*)^{1/2}: \L^{p'}(\Omega)^m \xrightarrow{\; \cong \;} \IW^{-1,p'}(\Omega),
\end{align*}
for some $p$ with H\"older conjugate $p' > d$. This is granted by Theorems~\ref{Thm: Square root solution Lp} and \ref{Thm: Range for Lp realization} in dimensions $d=2,3,4$ corresponding to $(2_*)' = 2^* = \infty, 6,4$. If $m=1$ and $L$ has real coefficients, then $\cJ(L) = (1,\infty)$ by Gaussian estimates \cite{Arendt-terElst} and we recover the result in~\cite{ABHR}. Improvements on $\cJ(L)$ for certain systems with real coefficients were obtained in \cite{Shen-Resolvent} and \cite{Wei-Zhang}. The Lam\'e system
\begin{align*}
 L_{D,0} u =  -\mu \Delta u - \mu' \nabla \div u
\end{align*}
fits into our framework provided $\mu > 0$ and $\mu + \mu' > 0$, see \cite{MM-Lame}. In this paper, M.~Mitrea and Monniaux consider $L_{D,0}$ with pure Dirichlet boundary conditions on a bounded domain $\Omega$ satisfying an interior ball conditions and obtain maximal regularity on $\L^q(\Omega)$ in the range $q \in (2_*,2^*)$. We remind the reader that Assumption~\ref{N} is void if one considers pure Dirichlet conditions. Hence, by putting together Theorems~\ref{Thm: Hinfty in Lp} and \ref{Thm: Range for Lp realization} we are able to drop this geometric assumption and obtain even a bounded $\H^\infty$-calculus for $L_{D,0}$ on any bounded domain, which in turn implies maximal regularity~\cite{Dore-Venni}. Some boundary regularity in the sense of Assumption~\ref{D} also allows us to increase the range for $q$. Let us stress, however, that the maximal regularity part has previously been obtained in a broader context by Tolksdorf~\cite{Patrick}, who uses a technique different to ours that does not pass through the bounded $\H^\infty$-calculus and yields the range $(2_*-\eps, 2^* + \eps)$ without further geometric assumptions.

In the following section we provide precise definitions of all assumptions and notation that have been dealt with rather intuitively up to now. The remaining sections are devoted to the proofs of our main results, Theorems~\ref{Thm: Kato} - \ref{Thm: Range for Lp realization}. The order of proofs will slightly differ from the presentation above.

\subsection*{Added after publication}

Proposition~\ref{Prop: Weak type estimate square root} of the published version (J.~Differential Equations 265 (2018), no.~4, 1279--1323) contains an unfortunate confusion concerning choice and role of the exponent $p$. In dimension $d \geq 3$, the given argument was correct but instead of $p \in \cJ(L)$ the weaker assumption $p^* \in \cJ(L)^\circ$ suffices. In dimension $d=2$, the premise was void but the conclusion holds for all $p \in (1,2]$. This update contains the necessary, very minor correction. All other results of the published version remain unchanged. I am grateful to Sebastian Bechtel for pointing out this mistake to me.
%%%%%%%%%%%%%%%%%%%%%%%%%%%%%%%%%%%%%%%%%%%%%%%%%%%%%%%%%%%%%%%%%%%%%%%%%%%%%%%%%%%%%%%%%%%%%%%%%%%%%%%%%%%%%%%%%%%%%%%%%%%%%%%%%%%%%%%%%%%%%%%%%%%%%%%%%%%%%%%%%%%%
\section{Notation and background}
\label{Sec: Notation}

Any Banach space $X$ under consideration is over the complex numbers and $X^*$ is the \emph{(anti)-dual space} of conjugate linear bounded functionals $X \to \IC$. We write $\dscal{\cdot}{\cdot}$ for duality pairings and $\scal{\cdot}{\cdot}$ for inner products on Hilbert spaces.

If $B$ is an open ball of radius $r$, then we denoted by $\alpha B$ the concentric ball of radius $\alpha r$ and let $C_1(B) := 4B$ as well as $C_j(B) := 2^{j+1}B \setminus 2^{j}B$ for $j=2,3,\ldots$. We use similar notation for cubes. We denote the semi-distance of sets $E, F \subseteq \R^d$ induced by the Euclidean distance on $\R^d$ by $\dist(E,F)$ and abbreviate $\dist_F(x):= \dist(\{x\},F)$. Lebesgue measure on $\R^d$ is denoted $|\cdot|$.

% We use the shorthand notations $\lesssim$, $\gtrsim$, and $\simeq$ for inequalities that hold up to multiplicative constants depending not on the parameters at stake (to be referred to as implied constants). 

\subsection{Geometry}
\label{Subsec: Geometry}

Henceforth we assume that $\Omega \subseteq \R^d$, $d \geq 2$, is a bounded, open, and connected set, that is to say, a bounded \emph{domain}. We remind the reader that the elliptic operator $L$ corresponds to a system of $m$ equations and hence acts on functions $u: \Omega \to \IC^m$. We can allow for different \emph{Dirichlet parts} for each coordinate function $u^{(k)}$, which we denote $D_k$, and which are assumed to be closed subsets of $\bd \Omega$. Hence, $L$ is subject to \emph{mixed boundary conditions}
\begin{align}
\label{Dirichlet boundary conditions}
 \text{$u^{(k)} = 0 \quad$ on $D_k \quad$ for $k=1, \ldots,m$},
\end{align}
where the form method forces \emph{natural boundary conditions} on the complementary parts $N_k:= \bd \Omega \setminus D_k$ through a formal integration by parts. See Section~\ref{Subsec: L} below. We put $D := \bigcap_{k=1}^m D_k$ and $N := \bd \Omega \setminus D = \bigcup_{k=1}^m N_k$.

Our results hold under the following three geometric assumptions. Sometimes not all of them shall be required.

\begin{assumption}{D}
 \label{D} For every $k = 1,\ldots,m$ the set $D_k$ is closed and either empty or such that $\cH_{d-1}(B \cap D) \simeq r^{d-1}$ holds for for all open balls $B$ of radius $r<1$ centered in $D$. Here, $\cH_{d-1}$ is the $(d-1)$-dimensional Hausdorff measure in $\R^d$.
\end{assumption}

\begin{assumption}{N}
 \label{N} Around every $x \in \cl{N}$ there is an open neighborhood $U_x$ and a bi-Lipschitz mapping $\Phi_x: U_x \to (-1,1)^d$ such that
 \begin{align*}
  \qquad \qquad \Phi_x(U_x \cap \Omega) = (-1,0) \times (-1,1)^{d-1}, \qquad \Phi_x(U_x \cap \bd \Omega) = \{0\} \times (-1,1)^{d-1}.
 \end{align*}
\end{assumption}

\begin{assumption}{$\Omega$}
 \label{Omega}  Comparability $|B \cap \Omega| \simeq r^d$ holds for all open balls $B$ of radius $r<1$ centered in $\Omega$.
\end{assumption}

Assumption~\ref{D} means that the $D_k$ are either empty or \emph{$(d-1)$-Ahlfors regular}. Likewise, Assumption~\ref{Omega} means that $\Omega$ is \emph{$d$-Ahlfors regular}. Assumption~\ref{Omega} is sometimes called \emph{weak Lipschitz condition}. It is strictly weaker than requiring that $\Omega$ has a Lipschitz boundary around $\cl{N}$, see \cite[Sec.~7.3]{RobertJDE} for a relevant example.

\subsection{Sobolev spaces}
\label{Subsec: Sobolev spaces}

We introduce Sobolev spaces on a domain $\Xi \subseteq \R^d$ with a vanishing trace condition on some closed subset $F \subseteq \cl{\Xi}$. This is understood in a very weak approximate sense but can be rephrased in a proper pointwise sense under minimal regularity assumptions \cite{Brewster, ET}. We shall not need such precision. Namely, for $1 < q < \infty$ we let $\W_F^{1,q}(\Xi)$ be the closure of
\begin{align*}
 \C_F^\infty(\Xi) := \Big\{\varphi |_\Xi : \varphi \in \C_0^\infty(\R^d),\, \supp(\varphi) \cap F = \emptyset \Big\}
\end{align*}
for the norm $\varphi \mapsto (\int_\Xi |\varphi|^q + |\nabla \varphi|^q \; \d x)^{1/q}$. The endpoint space $\W_F^{1,\infty}(\Xi)$ consists of all bounded Lipschitz continuous functions $u: \cl{\Xi} \to \IC$ that vanish everywhere on $F$. It carries the norm $u \mapsto \|u\|_\infty + \mathrm{Lip}(u)$, where $\mathrm{Lip}(u)$ is the smallest Lipschitz constant for $u$ on $\cl{\Xi}$. 

Usually we encounter the spaces $\W^{1,p}_{D_k}(\Omega)$. Under Assumption~\ref{N} there are bounded linear Sobolev extension operators $E_k$ that extend $\W^{1,p}_{D_k}(\Omega) \to \W^{1,p}_{D_k}(\R^d)$ and $\L^p(\Omega) \to \L^p(\R^d)$ for every $p \in (1,\infty)$, see \cite[Thm.~6.16]{Hardy-Poincare}. In particular, the usual embeddings of type $\W^{1,p}_{D_k}(\Omega) \subseteq \L^q(\Omega)$ hold. Usually, $q=p^*$ is the \emph{upper Sobolev conjugate} of defined by $1/p^* = 1/p - 1/d$. We also define a lower conjugate by $1/p_* = 1/p + 1/d$. In the particular situation they will be contained in $(1,\infty)$.

Sobolev spaces adapted to the boundary conditions \eqref{Dirichlet boundary conditions} are $\IW^{1,p}(\Omega):= \prod_{k=1}^m \W^{1,p}_{D_k}(\Omega)$. 
For $p \in (1,\infty)$ we define corresponding spaces of negative order $\IW^{-1,p}(\Omega) := (\IW^{1, q}(\Omega))^*$, where $1/p + 1/q = 1$.

\subsection{Holomorphic functional calculi}
\label{Subsec: Functional calculus}

General background and proofs of all relevant statements on the holomorphic functional calculus for sectorial operators can be found in~\cite{Haase}. Bisectorial operators can be treated almost identically but details have also been written down in \cite[Ch.~3]{EigeneDiss}. Throughout we shall assume that $X$ is a Hilbert space.

A linear operator $T$ in $X$ is \emph{sectorial} of angle $\phi \in [0, \pi)$ if its spectrum $\sigma(T)$ is contained in the closure of the sector $\S^+_\phi:= \{z \in \IC : \, |\arg z|<\phi \}$ and if 
\begin{align*}
 \IC \setminus \cl{\S^+_\psi} \to \Lop(X), \qquad z  \mapsto z(z-T)^{-1}
\end{align*}
is uniformly bounded for every $\psi \in (\phi, \pi)$. We agree on $\S^+_0 := (0,\infty)$. 

For $\psi \in (\phi, \pi)$ let $\H^\infty(\S^+_\psi)$ be the algebra of \emph{bounded holomorphic functions} on $\S^+_\psi$ and let $\H_0^\infty(\S^+_\psi)$ be the sub-algebra of functions $g$ satisfying $|g(z)| \leq C \min\{|z|^s, |z|^{-s}\}$ for some $C,s >0$ and all $z \in \S^+_\psi$. If $f(z) = a + b(1+z)^{-1} + g(z)$ for some $a,b \in \IC$ and $g \in \H_0^\infty(\S^+_\psi)$, then $f(T)$ is defined as a bounded operator on $X$ via
\begin{align}
\label{Eq: Def H0infty calculus}
f(T) = a + b(1+T)^{-1} + \frac{1}{2 \pi \i} \int_{\bd \S^+_\nu} g(z)(z-T)^{-1} \; \d z, 
\end{align}
where $\nu \in (\phi, \psi)$, the choice of which does not matter in view of Cauchy's theorem, and $\bd \S^+_\nu$ is oriented such that it surrounds $\sigma(T)$ counterclockwise in the extended complex plane. 

The definition of $f(T)$ is extended to larger classes of holomorphic functions by \emph{regularization}: One defines the closed operator $f(T) := e(T)^{-1}(ef)(T)$, if $e(T)$ and $(ef)(T)$ are already defined by the procedure above and $e(T)$ is one-to-one. This definition does not depend on the choice of $e$. The expected relations $f(T) + g(T) \subseteq (f + g)(T)$ and $f(T)g(T) \subseteq (f g)(T)$ hold and there is equality if $f(T)$ is bounded. An example are \emph{fractional powers} $T^\alpha$, $\alpha > 0$, which are defined on using $e(z)=(1+z)^{-\lceil\alpha\rceil-1}$.

If $T$ is one-to-one, then $e(z)=z(1+z)^{-2}$ regularizes any $f \in \H^\infty(\S^+_\psi)$. Its $\H^\infty(\S^+_\psi)$-calculus is called \emph{bounded} if for some constant $C_\psi > 0$ it holds
\begin{align*}
 \|f(T)\|_{X \to X} \leq C_\psi \|f\|_{\L^\infty(\S^+_\psi)} \qquad (f \in \H^\infty(\S^+_\psi)).
\end{align*}
It suffices to check this bound on $\H_0^\infty(\S^+_\psi)$. Indeed, for general $f \in \H^\infty(\S^+_\psi)$ the convergence lemma states that $f_n = e^{1/n} f \in \H^\infty_0(\S_\psi^+)$ satisfy $f_n \to f$ pointwise, $\|f_n\|_\infty \to \|f\|_\infty$, and $f_n(T) \to f(T)$ strongly on $X$. 

We will frequently use that if $T$ has a bounded $\H^\infty$-calculus of angle $\psi$, then so has the adjoint $T^*$. This is a consequence of the identity $f(T)^* = f^*(T^*)$ for every $f \in \H^\infty(\S^+_\psi)$, where $f^*(z) := \cl{f(\cl{z})}$. 

\emph{Bisectorial} operators are defined similarly upon replacing sectors by double sectors $\S_\phi = \S^+_\phi \cup -\S^+_\phi$, where $\phi \in [0, \pi/2)$. In their calculus $(\i + T)^{-1}$ replaces $(1+T)^{-1}$.

\subsection{The divergence form operator}
\label{Subsec: L}

We turn to the precise definition of the divergence form operator formally given by \eqref{L}. The coefficients $a_{ij}: \Omega \to \Lop(\IC^m)$ are measurable and essentially bounded and we put
\begin{align*}
 \Lambda := \sup_{0 \leq i,j \leq d} \essup_{x \in \Omega} \|a_{ij}(x)\|_{\IC^m \to \IC^m}.
\end{align*}
We remind the reader of the sesquilinear form
\begin{align*}
 \fa(u,v) = \int_\Omega \sum_{i,j = 1}^d \bigg( a_{ij} \pdx{u}{j} \cdot \cl{\pdx{v}{i}} + a_{i0} u \cdot \cl{\pdx{v}{i}} + a_{0j} \pdx{u}{j} \cdot \cl{v} + a_{00}u \cdot \cl{v} \bigg) \; \d x
\end{align*}
acting on $\IC^m$-valued functions. For all $u \in \IW^{1,2}(\Omega)$ we have $|a(u,u)| \leq \Lambda(d+1) (\|u\|_2^2 + \|\nabla u\|_2^2)$, where $\nabla u := (\pdx{u}{i})_i$ is considered as a vector in $(\IC^m)^d \cong \IC^{dm}$. Our ellipticity assumption is the following lower bound.

\begin{assumption}{L}
\label{Ass: Ellipticity}
There exists $\lambda > 0$ such that $\Re a(u,u)  \geq \lambda (\|u\|_2^2 + \|\nabla u\|_2^2)$ holds for all $u \in \IW^{1,2}(\Omega)$.
\end{assumption}

This implies that the \emph{numerical range} $\{a(u,u) : u \in \dom(a),\, \|u\|_2 = 1 \}$ is contained in the closed sector $\cl{S_\phi^+}$ of opening angle $\phi = \arctan((d+1)\Lambda/\lambda)$. We define $\omega \in [0, \pi/2)$ to be the \emph{smallest} such angle. 

The Lax-Milgram lemma associates with $a$ the bounded and invertible operator
\begin{align*}
 \Lop: \IW^{1,2}(\Omega) \to \IW^{1,2}(\Omega)^*, \quad \dscal{\Lop u}{v} = \fa(u,v).
\end{align*}
We define $L$ to be the maximal restriction of $\Lop$ to the Hilbert space $\L^2(\Omega)^m$. Our assumptions entail that $a$ is a closed densely defined sectorial form of angle $\omega$ in $\L^2(\Omega)^m$ and hence $L$ is \emph{maximal $\omega$-accretive}, see \cite[Ch.~VI]{Kato}. This is a stronger notion than sectoriality of angle $\omega$. It is known that such operators admit a bounded $\H^\infty$-calculus on any sector containing $\cl{\S^+_\omega}$. This is due to Crouzeix--Delyon \cite{Crouzeix-Delyon}, see also \cite[Cor.~7.1.17]{Haase}:

\begin{proposition}
\label{Prop: Hinfty on L2}
Let $\psi \in (\omega, \pi)$ and $f \in \H^\infty(\S^+_\psi)$. Then $\|f(L)\|_{\L^2 \to \L^2} \leq 4 \|f\|_\infty$.
\end{proposition}

\begin{remark}
\label{Rem: Hinfty on L2}
In fact, the constant $2+2/\sqrt{3}$ instead of $4$ works but this would become messy to carry through all subsequent calculations. The important information is that this bound is universal for the class of maximal accretive operators on Hilbert spaces.
\end{remark}

Since $L$ is maximal accretive on $\L^2(\Omega)^m$, the semigroup operators $\e^{-zL}$, $z \in \S_{\pi/2 - \omega}^+$, are \emph{contractions} on $\L^2(\Omega)^m$, see \cite[Sec.~XI.6]{Kato}. It will be useful to have the following \emph{exponential stability} which follows simply because $L-\lambda/2$ is still maximal accretive.

\begin{lemma}
\label{Lem: Expontial stability}
For every $t>0$ the bound $\|\e^{-tL}\|_{\L^2 \to \L^2} \leq \e^{-\lambda t/2}$. In particular, $L$ is invertible.
\end{lemma}

\subsection{The square root}
\label{Subsec: square root L}

Being maximal accretive, $L$ has a unique maximal accretive square root denoted $L^{1/2}$ and $\dom(L)$ is a core for $\dom(L^{1/2})$, see \cite[Thm.~V.3.35]{Kato}. This is the same operator as given by the functional calculus~\cite[Cor.~7.1.13]{Haase} and since $L$ is invertible, so is $L^{1/2}$ with inverse $L^{-1/2}$, see \cite[Prop.~3.1.1]{Haase}. We also have the formula $(L^{1/2})^* = (L^*)^{1/2}$ for the adjoints~\cite[Prop.~7.0.1e)]{Haase}. Since $L$ has a bounded $\H^\infty$-calculus of some angle $\omega \in [0,\pi/2)$, we have a \emph{resolution of the identity} in the sense of an improper Riemann integral
\begin{align}
\label{Eq: Resolution of the identity}
 u = \frac{2}{\sqrt{\pi}} \int_0^\infty L^{1/2} \e^{-t^2 L}u \; \d t \qquad (u \in \L^2(\Omega)^m),
\end{align}
see \cite[Thm.~5.2.6]{Haase}. We can apply $L^{1/2}$ or $L^{-1/2}$ on both sides of \eqref{Eq: Resolution of the identity} to obtain well-known integral formul\ae \, for either of them.

% %%%%%%%%%%%%%%%%%%%%%%%%%%%%%%%%%%%%%%%%%%%%%%%%%%%%%%%%%%%%%%%%%%%%%%%%%%%%%%%%%%%%%%%%%%%%%%%%%%%%%%%%%%%%%%%%%%%%%%%%%%%%%%%%%%%%%%%%%%%%%%%%%%%%%%%%%%%%%%%%%%%%
\section{A review on the \texorpdfstring{$\L^2$}{L2} results}
\label{Sec: L2}

We survey the first-order formalism and its practicability to the operator $L$ under our geometric assumptions developed in \cite{Laplace-Extrapolation, Darmstadt-KatoMixedBoundary}. This leads to Theorem~\ref{Thm: Kato} and the statements of Theorem~\ref{Thm: Holomorphy} when $p=2$. Throughout we assume \ref{D}, \ref{N}, and \ref{Omega}.

\subsection{The first order formalism}

We write the coefficients of $L$ in matrix form
\begin{align*}
 \begin{bmatrix}
  a_{00} & {\begin{bmatrix}\, a_{10} & \ldots & a_{d0} \, \end{bmatrix}} \\ {\begin{bmatrix} a_{10} \\ \vdots \\ a_{d0} \end{bmatrix}} & {\begin{bmatrix} a_{11} & \ldots & a_{1d} \\ \vdots & & \vdots \\ a_{d1} & \ldots & a_{dd} \end{bmatrix}}
 \end{bmatrix}
 = \begin{bmatrix} A_{\pe \pe} & A_{\pe \pa} \\ A_{\pa \pe} & A_{\pa \pa} \end{bmatrix}
 =A
\end{align*}
and define a closed operator $\nabla_D: \IW^{1,2}(\Omega) \subseteq \L^2(\Omega)^m \to \L^2(\Omega)^{dm}$ through $\nabla_D u = \nabla u$. See Section~\ref{Subsec: L} for the gradient of $\IC^m$-valued functions. An equivalent way of putting the definition of $L$ through the form method in Section~\ref{Subsec: L} is $L = \begin{bmatrix} 1 & \nabla_D^* \end{bmatrix} A \begin{bmatrix} 1 & \nabla_D \end{bmatrix}^\top$.

On $\cH = \L^2(\Omega)^m \times \L^2(\Omega)^m \times \L^2(\Omega)^{dm}$ we define operator matrices on their natural domains,
\begin{align*}
 \Gamma:= 
 \begin{bmatrix} 0 & & \\ 1 & 0 & \\ \nabla_D & & 0 \end{bmatrix}, \qquad
 B_1 = \begin{bmatrix} 1 & & \\ & 0 & \\ && 0 \end{bmatrix}, \qquad
 B_2 = \begin{bmatrix} 0 & & \\ & A_{\pe \pe} & A_{\pe \pa} \\ & A_{\pa \pe} & A_{\pa \pa} \end{bmatrix},
\end{align*}
and consider the \emph{perturbed Dirac operator} $\Pi_B := \Gamma + B_1 \Gamma^* B_2$. Indeed $\Pi_B$ is a Dirac operator in that its square contains $L$, namely
\begin{align*}
 \Pi_B = \begin{bmatrix}  & \begin{bmatrix} 1 & (\nabla_D)^* \end{bmatrix}A \,  \\ \begin{matrix} 1 \\ \nabla_D \end{matrix} & \end{bmatrix}, \qquad 
 \Pi_B^2 = \begin{bmatrix} L & \\ & \begin{bmatrix} 1 & \nabla_D^* \hspace{-7pt} \\ \nabla_D & \nabla_D \nabla_D^* \end{bmatrix} A \end{bmatrix}.
\end{align*}
Within this framework our ellipticity Assumption~\ref{Ass: Ellipticity} can be rephrased as
\begin{align}
\label{Ellipticity Dirac}
 \Re \scal{B_2 u}{u} \geq \lambda \|u\|_\cH^2 \qquad (u \in \Rg(\Gamma)),
\end{align}
that is to say, $B_1$ and $B_2$ are accretive perturbations of $\Gamma^*$ and $\Gamma$, respectively. It follows that $\Pi_B$ is bisectorial of some angle $\omega_B \in (0, \pi/2)$ with resolvent estimates both depending only on $\lambda$, $\Lambda$, see \cite[Prop.~2.5]{AKM-QuadraticEstimates}. The Kato square root problem has now become a question on the functional calculus for $\Pi_B$ -- comparing $\dom(L^{1/2})$ and $\IW^{1,2}(\Omega)$ amounts to comparing (the first components of) $\dom((\Pi_B^2)^{1/2})$ and $\dom(\Pi_B)$. 

On the abstract level, we have the following result from \cite[Prop.~3.3.15]{EigeneDiss} or \cite[p.~461]{Laplace-Extrapolation}. Explicit constants have not been stated there but pop up in the given proofs. 

\begin{lemma}
\label{Lem: Abstract Kato}
Let $T$ be a bisectorial operator in a Hilbert space $X$. Suppose the restriction to the closure of its range $\Rg(T)$ has a bounded $\H^\infty(\S_\psi)$-calculus for some $\psi \in (0,\pi/2)$. Then $\dom((T^2)^{1/2}) = \dom(T)$ and
\begin{align*}
 \frac{1}{C_\psi} \|T u\|_X \leq \|(T^2)^{1/2} u \|_X \leq C_\psi \|Tu\|_X \qquad (u \in \dom(T)),
\end{align*}
where $C_\psi$ is the bound for the functional calculus.
\end{lemma}

On the concrete level, the goal of \cite{Laplace-Extrapolation} was to prove \emph{quadratic estimates} for the particular choice of $\Pi_B$ under a set of hypotheses called (H1) - (H7). We do not need to recall them here and the interested reader can refer to \cite[Sec.~5]{Laplace-Extrapolation}. For the operators above, they have been verified in detail in \cite[Sec.~6.1]{Laplace-Extrapolation} with two exceptions: The accretivity assumption (H2), which is precisely \eqref{Ellipticity Dirac}, and the regularity assumption (H7), whose verification in \cite{Laplace-Extrapolation} was subject to an additional assumption called Assumption~(E) that became a true theorem only later on in \cite[Thm.~4.4]{Darmstadt-KatoMixedBoundary}. 

This being said, \cite[Thm.~3.3]{Laplace-Extrapolation} reads as follows.

\begin{proposition}
\label{Prop: PiB}
For some constant $C \in (0,\infty)$ depending on geometry, dimensions, and ellipticity, there are quadratic estimates
\begin{align*}
 \frac{1}{C} \|u\|_2^2 \leq \int_0^\infty \|t \Pi_B (1+t^2 \Pi_B^2)^{-1}u\|_2^2 \; \frac{\d t}{t} \leq C \|u\|_2^2 \qquad (u \in \cl{\Rg(\Pi_B)}).
\end{align*}
\end{proposition}

By McIntosh's theorem~\cite{McIntosh} \emph{quadratic estimates} as above imply boundedness of the $\H^\infty(\S_\psi)$-calculus of any angle $\psi \in (\omega_\B, \pi/2)$ for the \emph{restriction of $\Pi_B$} to $\cl{\Rg(\Pi_B)}$, which is a one-to-one bisectorial operator. See also \cite[Thm.~3.4.11]{EigeneDiss}. The bound for the functional calculus depends on the angle and the resolvent bounds for $\Pi_B$ as is easily seen from the proofs in \cite{McIntosh} or \cite{EigeneDiss}. Hence, we obtain the

\begin{proof}[Proof of Theorem~\ref{Thm: Kato}]
We have just seen that Lemma~\ref{Lem: Abstract Kato} applies to $\Pi_B$ and yields $\dom((\Pi_B^2)^{1/2}) = \dom(\Pi_B)$ with equivalent graph norms. This implies $\dom(L^{1/2}) = \IW^{1,2}(\Omega)$ with equivalent norms upon restricting to the first coordinate in the $3 \times 3$ operator matrices. Implied constants in this argument depend on geometry and ellipticity.
\end{proof}

\subsection{Holomorphic dependence}

Let us discuss holomorphic dependence in the spirit of Theorem~\ref{Thm: Holomorphy} in the case $p=2$.  We assume some familiarity with vector-valued holomorphic functions and refer to \cite[App.~A]{ABHN} for background information. 

Henceforth, let $O \subseteq \IC$ be an open set and $A(z)$ be coefficient matrices as above that depend holomorphically on $z \in O$. We assume that all of them satisfy the ellipticity assumptions from Section~\ref{Subsec: L} with the same parameters $\lambda, \Lambda$ and we write $L_z$ for the corresponding operators defined through the sesquilinear forms $a_z$. By a slight abuse of notation, $\omega$ denotes the supremum of all accretivity angles of the operators $L_z$ so that $\sigma(L_z) \subseteq \cl{\S_\omega^+}$ for every $z$.

For all $u, v \in \IW^{1,2}(\Omega)$ also the map $z \mapsto a_z(u,v)$ is holomorphic on $O$. This follows for example from Morrera's theorem after changing the order of integration. Hence, we have a holomorphic family of sectorial forms in the sense of~\cite{Kato}. It follows that the associated operators $L_z$ are \emph{resolvent holomorphic}, that is to say, 
\begin{align*}
 O \to \Lop(\L^2(\Omega)^m), \qquad z \mapsto (\mu - L_z)^{-1}
\end{align*}
is holomorphic for every $\mu \in \IC \setminus \cl{\S_\omega^+}$. For a proof see \cite[Thm.~VII.4.2]{Kato} or the elegant argument presented in \cite{Vogt-Voigt}. By superposition, this carries over to objects in the functional calculus for the operators $L_z$. Two important examples are as follows.

\begin{corollary}
\label{Cor: Functional calculus L L2}
In the situation above, let $f \in \H^\infty(\S_\phi^+)$ for some $\phi \in (\omega, \pi)$. Then the map $O \to \Lop(\L^2(\Omega)^m)$, $z \mapsto f(L_z)$ is holomorphic.
\end{corollary}

\begin{proof}
If $f \in \H_0^\infty(\S_\phi^+)$, then the claim follows from Morrera's theorem after changing the order of integration in the integral representation of $f(L_z)$. In the general case we conclude by Vitali's theorem: Indeed, as in Section~\ref{Subsec: Functional calculus} we can take a bounded sequence $\{f_n\}_n$ in $\H_0^\infty(\S_\phi^+)$ that converges pointwise to $f$ such that for every $z \in O$ we have strong convergence $f_n(L_z) \to f(L_z)$ on $\L^2(\Omega)^m$. The missing hypothesis for Vitali's theorem, that is the uniform bound in $n$ and $z$ for the holomorphic functions $z \mapsto f_n(L_z)$, is due to Proposition~\ref{Prop: Hinfty on L2}.
\end{proof}

\begin{corollary}
\label{Cor: Square root holomorphic L2}
In the situation above the map $O \to \Lop(\IW^{1,2}(\Omega), \L^2(\Omega)^m)$, $z \mapsto L_z^{1/2}$ is holomorphic.
\end{corollary}

\begin{proof}
The map under consideration is uniformly bounded on $O$ thanks to Theorem~\ref{Thm: Kato}. Hence, it suffices to check holomorphy of $z \mapsto L_z^{1/2}u$ for every $u \in \IW^{1,2}(\Omega)$, see \cite[Prop.~A.3]{ABHN} for this reduction. Applying $L_z^{1/2}$ on both side of \eqref{Eq: Resolution of the identity} yields
\begin{align*}
 L_z^{1/2}u = \lim_{n \to \infty} \frac{2}{\sqrt{\pi}} \int_{2^{-n}}^{2^n} L_z \e^{-t^2 L_z} u \; \d t =: F_n(L_z) L_z^{1/2}u,
\end{align*}
with convergence in $\L^2(\Omega)^m$. Here, $F_n(\mu) = \frac{2}{\sqrt{\pi}} \int_{2^{-n}}^{2^n} (t^2\mu)^{1/2} \e^{-t^2 \mu} \frac{\d t}{t}$ are bounded holomorphic functions on any sector contained in the right complex halfplane and a substitution reveals a uniform bound in $n$ and $\mu$. By the first inequality above, $L_z^{1/2}u$ is the pointwise limit of a sequence of holomorphic functions on $O$. And taking into account Proposition~\ref{Prop: Hinfty on L2} and Theorem~\ref{Thm: Kato}, the second one means that this sequence is uniformly bounded in $n$ and $z$. As before, Vitali's theorem yields the claim.
\end{proof}

%%%%%%%%%%%%%%%%%%%%%%%%%%%%%%%%%%%%%%%%%%%%%%%%%%%%%%%%%%%%%%%%%%%%%%%%%%%%%%%%%%%%%%%%%%%%%%%%%%%%%%%%%%%%%%%%%%%%%%%%%%%%%%%%%%%%%%%%%%%%%%%%%%%%%%%%%%%%%%%%%%%%
\section{Off-diagonal estimates}
\label{Sec: Properties of the semigroup}

In this section we establish $\L^p \to \L^2$ off-diagonal estimates for the semigroup generated by $-L$ and related families. They are the proper substitute for Gaussian kernel bounds in our context and play a crucial role in all subsequent sections. Here, they shall already lead us to the proof Theorem~\ref{Thm: Range for Lp realization}.

\begin{definition}
\label{Def: OD estimates}
Let $J \subseteq \IC$ and $\cT = \{T(z)\}_{z \in J}$ a family of bounded linear operators $\L^2(\Xi)^{m_1} \to \L^2(\Xi)^{m_2}$, where $m_1, m_2 \in \IN$ and $\Xi \subseteq \R^d$ is (Lebesgue) measurable. Given $1 \leq p \leq q \leq \infty$, we say that $\cT$ satisfies \emph{$\L^p \to \L^q$ off-diagonal estimates} if for some constants $C, c \in (0,\infty)$ the estimate
\begin{align*}
 \|T(z) u\|_{\L^q(F)^{m_2}} \leq C |z|^{\frac{d}{2q} - \frac{d}{2p}} \e^{-c \frac{\dist(E,F)^2}{|z|}} \|u\|_{\L^p(E)^{m_1}}
\end{align*}
holds for all $z \in J$, all measurable sets $E, F \subseteq \Xi$, and all $u \in \L^p(\Xi)^{m_1} \cap \L^2(\Xi)^{m_1}$ that are supported in $E$. We say that $\cT$ is \emph{$\L^p \to \L^q$ bounded} if for all $u \in \L^p(\Xi)^{m_1} \cap \L^2(\Xi)^{m_1}$,
\begin{align*}
 \|T(z) u\|_{\L^q(\Xi)^{m_2}} \leq C |z|^{\frac{d}{2q} - \frac{d}{2p}} \|u\|_{\L^p(\Xi)^{m_1}}.
\end{align*}
In the case $p=q$ we simply speak of \emph{$\L^p$ off-diagonal estimates} and \emph{$\L^p$ boundedness}
\end{definition}

We begin with $\L^2 \to \L^2$ off-diagonal bounds.

\begin{proposition}
\label{Prop: OD estimates for semigroup on L2}
Suppose that $\Omega \subseteq \R^d$ is a bounded domain. Let $\psi \in [0, \pi/2 - \omega)$. Then $\{\e^{-zL}\}_{z \in \S^+_\psi}$, $\{zL\e^{-zL}\}_{z \in \S^+_\psi}$, and $\{\sqrt{z}\nabla \e^{-zL}\}_{z \in \S^+_\psi}$ satisfy $\L^2$ off-diagonal estimates and implicit constants depend on $\psi$, ellipticity, dimensions, and the diameter of $\Omega$.
\end{proposition}

This will follow by Davies' perturbation method \cite{Riesz-Transforms, LSV, Davies} and we shall indicate the major steps in order to help the reader through. To get the method running, we need the following invariance property.

\begin{lemma}
\label{Lem: Form domain invariant under multiplication}
Every Lipschitz continuous function $\varphi: \R^d \to \R$ induces a bounded multiplication operator on $\IW^{1,2}(\Omega)$.
\end{lemma}

\begin{proof}
Boundedness of the multiplication operator with respect to the $\W^{1,2}(\Omega)^m$-norm follows from the product rule. Hence, it suffices to check that the closed subspace $\IW^{1,2}(\Omega)$ is left invariant and by density this will follow from $\varphi u \in \IW^{1,2}(\Omega)|_\Omega$ for $u \in \prod_{k=1}^m \C_{D_k}^\infty(\R^d)$. But in this case $\varphi u \in \W^{1,2}(\R^d)^m$ with compact support and each of its components having support away from the respective Dirichlet part, so that approximants in $\prod_{k=1}^m \C_{D_k}^\infty(\R^d)$ for the $\W^{1,2}(\R^d)^m$-topology can be constructed via mollification with smooth kernels.
\end{proof}

\begin{proof}[Proof of Proposition~\ref{Prop: OD estimates for semigroup on L2}]

We begin with off-diagonal bounds for $z = t >0$. Let $\varphi: \R^d \to \R$ be Lipschitz continuous with $\|\nabla \varphi \|_\infty \leq 1$ and let $\rho >0$; both yet to be specified. Since by the preceding lemma $\IW^{1,2}(\Omega)$ is invariant under multiplication with $\e^{\pm \rho \varphi}$, we can define $L_{\rho, \varphi}:= \e^{\rho \varphi} L \e^{- \rho \varphi}$ by means of the form method using the bounded sesquilinear form
\begin{align*}
 \fa_{\rho, \varphi}: \IW^{1,2}(\Omega) \times \IW^{1,2}(\Omega) \to \IC, \quad (u,v) \mapsto \fa(\e^{-\rho \varphi}u, \e^{\rho \varphi} v).
\end{align*}
In order to see that $\fa_{\rho, \varphi}$ is sectorial, we multiply out the expression for $\fa(\e^{-\rho \varphi}u, \e^{\rho \varphi} u)$ obtained from the definition of $\fa$ in \eqref{a}, use boundedness and ellipticity of $\fa$, and control the error terms $\fa(u,u)-\fa(\e^{-\rho \varphi}u, \e^{\rho \varphi} u)$ by means of Young's inequality with $\eps$. This results in the two estimates
\begin{align*}
|\fa_{\rho, \varphi}(u,u)| \leq 2\Lambda(d+1) (\|u\|_2^2 + \|\nabla u\|_2^2) + c (\rho^2+\rho) \|u\|_2^2
\end{align*}
and
\begin{align}
\label{Eq1: OD estimates for semigroup on L2}
\Re \fa_{\rho, \varphi}(u,u) \geq \frac{\lambda}{2}(\|u\|_2^2 + \|\nabla u\|_2^2) - c (\rho^2 + \rho) \|u\|_2^2,
\end{align}
where $c \in (0,\infty)$ depends upon ellipticity and dimensions. Thus, $L_{\rho, \varphi} + 2c(\rho^2 + \rho)$ is maximal accretive with angle $\arctan(4\Lambda(d+1)/\lambda)$. The universal bound for its $\H^\infty$-calculus yields
\begin{align}
\label{Eq2: OD estimates for semigroup on L2}
 \|\e^{-tL_{\rho, \varphi}} \|_{\L^2 \to \L^2} + \|t (L_{\rho, \varphi} + 2c (\rho^2 + \rho)) \e^{-tL_{\rho, \varphi}} \|_{\L^2 \to \L^2}  \leq  4\e^{2c (\rho^2 + \rho) t} \qquad (t>0),
\end{align}
see Proposition~\ref{Prop: Hinfty on L2}. Moreover, we have by definition 
\begin{align*}
 \fa_{\rho, \varphi}(u,u) + 2c (\rho^2 + \rho) \|u\|_2^2 = \scal{L_{\rho,\varphi}u + 2c(\rho^2 + \rho)u}{u} \qquad (u \in \dom(L_{\rho,\varphi}))
\end{align*}
and as a holomorphic semigroup maps into the domain of its generator, the previous bounds along with the ellipticity estimate \eqref{Eq1: OD estimates for semigroup on L2} imply
\begin{align}
\label{Eq3: OD estimates for semigroup on L2}
 \|\sqrt{t} \nabla \e^{-tL_{\rho, \varphi}} \|_{\L^2 \to \L^2} \leq 4 \Big(\frac{\lambda}{2}\Big)^{-1/2}  \e^{2c (\rho^2 + \rho) t} \qquad (t>0).
\end{align}

Now, let $E, F \subseteq \Omega$ be measurable sets, and let $u \in \L^2(\Omega)^m$ be supported in $E$. We specialize $\varphi(x) = \dist(x,E)$ and obtain 
\begin{align*}
 \e^{-tL} u = \e^{- \rho \varphi} \e^{\rho \varphi} \e^{-tL} \e^{- \rho \varphi} u =  \e^{-\rho \varphi} \e^{-tL_{\rho, \varphi}} u \qquad (t>0),
\end{align*}
where in the last step we used that the similarity of operators $L_{\rho, \varphi}:= \e^{\rho \varphi} L \e^{- \rho \varphi}$ inherits to resolvents and hence to the functional calculi. From \eqref{Eq2: OD estimates for semigroup on L2} we can infer 
\begin{align*}
 \|\e^{-tL} u \|_{\L^2(F)} 
\leq \e^{-\rho \dist(E,F)} \|\e^{-tL_{\rho, \varphi}} u \|_{\L^2(\Omega)}
\leq 4 \e^{2c (\rho^2 + \rho)t - \rho \dist(E,F)} \|u\|_{\L^2(E)},
\end{align*}
which on choosing $\rho := \frac{\dist(E,F)}{4c t}$ and recalling that $\Omega$ is bounded, becomes the off-diagonal bound 
\begin{align*}
 \|\e^{-tL} u \|_{\L^2(F)}  \leq 4 \e^{-\frac{\dist(E,F)^2}{8ct}} \e^{\frac{\dist(E,F)}{2}}\|u\|_{\L^2(E)} \leq 4 \e^{\frac{\diam(\Omega)}{2}} \e^{-\frac{\dist(E,F)^2}{8ct}} \|u\|_{\L^2(E)}.
\end{align*}
The estimates for $tL \e^{-tL}$ and $\sqrt{t} \nabla \e^{-tL}$ follow likewise from either \eqref{Eq2: OD estimates for semigroup on L2} or \eqref{Eq3: OD estimates for semigroup on L2}.

Finally, to treat the general case $z \in \S^+_\psi$, we replace $L$ by $\e^{\i \arg z} L$: Since $|\arg z|< \pi/2 - \omega$, this is an operator in the same class as $L$ and ellipticity constants of the corresponding form depend on $\lambda, \Lambda, \psi$. Hence, the first part of the proof applies with $t = |z|$ and the claim follows on noting $\e^{-zL} = \e^{-t \e^{\i \arg z} L}$.
\end{proof}

The subsequent proposition builds the bridge to $\L^p \to \L^p$ and $\L^p \to \L^2$ estimates. Going through the cycle of all five implication shows that for the semigroup all concepts are more or less equivalent if one allows a small play in the Lebesgue exponents.

\begin{proposition}
\label{Prop: Hypercontractivity}
Assume $\Omega$ satisfies Assumption~\ref{N}. Let $p \in [1,2)$. For $\psi \in [0, \pi/2 - \omega)$ put $\cS = \{\e^{-zL}\}_{z \in \S^+_\psi}$ and $\cN = \{\sqrt{z} \nabla \e^{-zL}\}_{z \in \S^+_\psi}$. Then the following hold.
\begin{enumerate}
 \item \label{i Hypercontractivity} If $\{\e^{-tL}\}_{t>0}$ is $\L^p$ bounded, then $\cS$ is $\L^p \to \L^2$ bounded.
 \item \label{ii Hypercontractivity} If $\cS$ is $\L^p \to \L^2$ bounded, then so is $\cN$.
 \item \label{iii Hypercontractivity} If $\cS$ is $\L^p \to \L^2$ bounded and $q \in (p,2)$, then $\cS$ satisfies $\L^q \to \L^2$ off-diagonal estimates.
 \item \label{iv Hypercontractivity} If $\cS$ satisfies $\L^p \to \L^2$ off-diagonal estimates, then so does $\cN$.
 \item \label{v Hypercontractivity} If $\cS$ and $\cN$ satisfy $\L^p \to \L^2$ off-diagonal estimates, then $\cS$ and $\cN$ are $\L^p$ bounded, respectively.
\end{enumerate}
\end{proposition}

\begin{proof}
We begin with \eqref{i Hypercontractivity}. Let $u \in \L^2(\Omega)^m$ with $\|u\|_p = 1$. First, we establish the $\L^p \to \L^2$ bounds for the semigroup in the case $z = t > 0$. We obtain the interpolation inequality
 \begin{align*}
  \|v\|_{\L^2(\Omega)}^2 \lesssim \|v\|_{\IW^{1,2}(\Omega)}^{2 \theta} \|v\|_{\L^p(\Omega)}^{2-2\theta} \qquad (v \in \IW^{1,2}(\Omega)),
 \end{align*}
 where $1/\theta = 1+ 2p/(2d-pd)$, from the classical Gagliardo-Nirenberg inequality for functions on $\R^d$, see \cite[p.~125]{Nirenberg}, and the boundedness of the extension operators $E_k$, see Section~\ref{Subsec: Sobolev spaces}. Here, Assumption~\ref{N} was used. We apply this with $v = \e^{-tL} u$ and obtain from the assumption and ellipticity
 \begin{align*}
  \|\e^{-tL}u \|_2^2 \lesssim \|\e^{-tL}u \|_{\IW^{1,2}}^{2 \theta} \lesssim \big(\Re \fa (\e^{-tL}u, \e^{-tL}u)\big)^{\theta} = \big(\Re \scal{L\e^{-tL}u}{\e^{-tL}u}\big)^{\theta}.
 \end{align*}
 Hence, $f(t):=\|\e^{-tL}u\|_2^2$ satisfies the differential inequality
 \begin{align*}
  f(t) \leq C (-f'(t))^{\theta} \qquad (t>0),
 \end{align*}
 where $C>0$ depends on geometry and $\Jnorm{p}$, see \eqref{Jnorm}. If $f$ vanishes at some point of the interval $(t/2,t)$, then $f(t) = 0$ by the semigroup property and we are done. Otherwise, we obtain 
 \begin{align*}
  \frac{t}{2} \leq - \int_{t/2}^t \frac{C f'(s)}{f(s)^{1/\theta}} \; \d s \leq \frac{C \theta}{1- \theta} f(t)^{1 - 1/\theta} = \frac{C \theta}{1- \theta} f(t)^{-2p/(2d-pd)},
 \end{align*}
 which, by definition of $f$, is the required $\L^p \to \L^2$ estimate. In order to extend this bound to $z \in \S^+_\psi$, we put $\psi':= (\psi + \frac{\pi}{2} - \omega)/2$ and decompose $z = z' + t$, where $|\arg z'| = \psi'$ and $t>0$, so that $|z| \simeq |z'| \simeq t$ with implicit constants depending on $\psi$ and $\omega$. The claim then follows from the contractivity of the semigroup on $\L^2$ and the first part of the proof:
 \begin{align*}
   \|\e^{-zL}u  \|_2 \leq  \|\e^{-z'L}  \|_{\L^2 \to \L^2}  \|\e^{-tL}u  \|_2 \lesssim |t|^{\frac{d}{4} - \frac{d}{2p}} \simeq |z|^{\frac{d}{4} - \frac{d}{2p}}.
 \end{align*}
 
Next, \eqref{ii Hypercontractivity} follows from the semigroup law and the assertion for $\cS$. Indeed, it suffices to write
 \begin{align*}
  \sqrt{2z} \nabla \e^{-2zL} = \sqrt{2}(\sqrt{z} \nabla \e^{-z L}) \e^{-z L}
 \end{align*}
and concatenate the $\L^2$ bound of the first factor (Proposition~\ref{Prop: OD estimates for semigroup on L2}) with the assumed $\L^p \to \L^2$ bound for the second one. 

As for \eqref{iii Hypercontractivity}, we interpolate by means of the Riesz-Thorin theorem the assumed $\L^p \to \L^2$ bound with the $\L^2$ off-diagonal estimates provided by Proposition~\ref{Prop: OD estimates for semigroup on L2}. (Once we fixed the sets $E$, $F$, in the definition of off-diagonal estimates.)

Assertion \eqref{iv Hypercontractivity} follows by a refinement of the argument for \eqref{ii Hypercontractivity}. We let $E,F \subseteq \Omega$ measurable sets, $u \in \L^2(\Omega)^m$ with support in $E$ and $z \in \S^+_\psi$. We also use a measurable set $G \subseteq \Omega$ to be specified yet. By the semigroup law we have
\begin{align*}
 \|\sqrt{2z} \nabla \e^{-2zL} u\|_{\L^2(F)}
&\leq \sqrt{2} \bigg(\big\|\sqrt{z} \nabla \e^{-z L} \ind_G \e^{-zL} u\big\|_{\L^2(F)} + \big\|\sqrt{z} \nabla \e^{-zL} \ind_{{}^c G} \e^{-zL}u \big\|_{\L^2(F)} \bigg)
\intertext{and hence by assumption and $\L^2$ off-diagonal estimates for the gradient of the semigroup,}
& \leq CC' |z|^{\frac{d}{4}-\frac{d}{2p}} \bigg(\e^{-c\frac{\dist(G,F)^2}{|z|} -c'\frac{\dist(E,G)^2}{|z|}} + \e^{-c\frac{\dist({}^c G,F)^2}{|z|}-c'\frac{\dist(E,{}^c G)^2}{|z|}} \bigg) \|u\|_{\L^p(E)},
\end{align*}
where $C,C',c,c' \in (0,\infty)$. For the choice $G = \{x \in \Omega : \dist(x, F) \geq \dist(E,F)/2 \}$ we have $\dist(G,F)  \geq \dist(E,F)/2$ and $\dist(E, {}^c G) \geq \dist(E,F)/2$, which in turn yields the claim.

Eventually, \eqref{v Hypercontractivity} follows from the subsequent lemma applied to $T = \e^{-zL}$ or $T= \sqrt{z} \nabla \e^{-zL}$ on choosing $g(r) = C |z|^{d/4 - d/(2p)} \e^{-c r^2/|z|}$ and $s=\sqrt{|z|}$.\qedhere
\end{proof}

\begin{lemma}
\label{Lem: OD implies bounded}
Let $1 \leq p \leq q \leq \infty$ and $T$ a bounded linear operator $\L^2(\Xi)^{m_1} \to \L^2(\Xi)^{m_2}$, where $m_1, m_2 \in \IN$ and $\Xi \subseteq \R^d$ is measurable. If $T$ satisfies $\L^p \to \L^q$ off-diagonal estimates in the form 
\begin{align*}
 \|T u\|_{\L^q(F \cap \Xi)} \leq g(\dist(E,F)) \|u\|_{\L^p(E \cap \Xi)},
\end{align*}
whenever $E$, $F$ are closed axis-parallel cubes in $\R^d$ and $u \in \L^p(\Xi)^{m_1} \cap \L^2(\Xi)^{m_1}$ is supported in $E \cap \Xi$ and $g$ is some decreasing function. Then $T$ is $\L^p$ bounded with norm bounded by $s^{d/p - d/q} \sum_{k \in \IZ^d} g(s \max\{|k|/\sqrt{d} -1, 0\})$ for any $s>0$ provided this sum is finite.
\end{lemma}

This is essentially \cite[Lem.~4.3]{Riesz-Transforms} but because of two somewhat confusing misprints, one in the statement and one in proof, we decided to include the argument.

\begin{proof}
Let $u \in \L^p(\Xi)^{m_1} \cap \L^2(\Xi)^{m_1}$. We partition $\R^d$ into closed, axis-parallel cubes $\{Q_k\}_{k \in \IZ^d}$ of sidelength $s$ with center $sk$ and let $u_k := \ind_{Q_k \cap \Xi} u$. From H\"older's inequality and the assumption we obtain
\begin{align*}
 \|T u\|_{\L^p(\Xi)}^p 
 &= \sum_{k \in \IZ^d} \|Tu\|_{\L^p(Q_k \cap \Xi)}^p
 \leq s^{d-dp/q} \sum_{k \in \IZ^d} \|Tu\|_{\L^q(Q_k \cap \Xi)}^p \\
 &\leq s^{d-dp/q} \sum_{k \in \IZ^d} \Big(\sum_{j \in \IZ^d} \|Tu_j\|_{\L^q(Q_k \cap \Xi)}\Big)^p \\
 &\leq s^{d-dp/q} \sum_{k \in \IZ^d} \Big(\sum_{j \in \IZ^d} g(\dist(Q_j \cap \Xi, Q_k \cap \Xi)) \|u_j\|_p\Big)^p.
\end{align*}
Let $|\, \cdot \,|_\infty$ be the $\ell^\infty$ norm on $\R^d$ and $\dist_\infty$ the corresponding distance. We have $\dist_\infty(Q_j,Q_k) =\max\{|sj-sk|_\infty-s,0\}$ and thus $\dist(Q_j \cap \Xi,Q_k \cap \Xi) \geq s \max\{|j-k|/\sqrt{d}-1,0\}$. Since $g$ is decreasing, we can infer
\begin{align*}
 \|T u\|_{\L^p(\Xi)}
 &\leq s^{d/p-d/q} \bigg(\sum_{k \in \IZ^d} \bigg(\sum_{j \in \IZ^d} g \bigg(s \max \bigg\{\frac{|j-k|}{\sqrt{d}}-1,0 \bigg\}\bigg) \|u_j\|_p\bigg)^p\bigg)^{1/p} \\
 &\leq s^{d/p-d/q} \bigg(\sum_{k \in \IZ^d} g \bigg(s \max \bigg\{\frac{|k|}{\sqrt{d}}-1,0 \bigg\}\bigg)\bigg) \bigg(\sum_{j \in \IZ^d} \|u_j\|_p^p \bigg)^{1/p},
\end{align*}
where the second step is an application of Young's inequality for (discrete) convolutions. The sum in $j$ equals $\|u\|_{\L^p(\Xi)}^p$ and the claim follows.
\end{proof}

\begin{remark}
\label{Rem: Hypercontractivity}
The proof of Proposition~\ref{Prop: Hypercontractivity} reveals that in each implication implicit constants in the conclusion depend at most on those appearing in the premise and potentially on $p$, $\psi$, ellipticity, and geometry.
\end{remark}

\begin{remark}
\label{Rem: Hypercontractivity2}
As for exponents $p \in (2,\infty]$, the results analogous to Proposition~\ref{Prop: Hypercontractivity} for the semigroup $\cS$ follows from duality using that $(\e^{-zL})^* = \e^{-\cl{z}L^*}$. For the gradient family $\cN$ there is, however, no such direct argument. A proof on $\Omega = \R^d$ that makes specific use of the invariance of $\Omega$ under linear transformations is presented in \cite[Prop.~3.9]{Riesz-Transforms}. We shall not need this in the following and thus do not attempt to adapt it.
\end{remark}

The following lemma deals with the first part of Theorem~\ref{Thm: Range for Lp realization}.

\begin{lemma}
\label{Lem: JA contains Sobolev conjugates}
Suppose $\Omega$ satisfies Assumption~\ref{N} and let $p \in (2_*, 2^*)$. Then $\{\e^{-tL}\}_{t>0}$ is $\L^p$ bounded with a bound depending only on $p$, ellipticity, dimensions, and geometry.
\end{lemma}

\begin{proof}
By duality we may restrict ourselves to $p \in (2, 2^*)$, compare with Remark~\ref{Rem: Hypercontractivity2}. By the same remark, we can use \eqref{iii Hypercontractivity} and \eqref{v Hypercontractivity} from Proposition~\ref{Prop: Hypercontractivity} for $p>2$. The upshot is that it suffices to check $\L^2 \to \L^p$ boundedness of the semigroup. By ellipticity and the Cauchy-Schwarz inequality we have for $u \in \L^2(\Omega)^m$ and $t>0$,
\begin{align*}
  \lambda \|\e^{-tL} u  \|_{\IW^{1,2}}^2 \leq \Re a(\e^{-tL} u, \e^{-tL} u) = \Re \scal{L \e^{-tL} u}{\e^{-tL} u} \leq \|L\e^{-tL} u\|_2 \|\e^{-tL} u\|_2.
\end{align*}
On the other hand, we obtain for $v \in \IW^{1,2}(\Omega)$ the interpolation inequality
\begin{align*}
 \|v\|_p \lesssim \|v\|_{\IW^{1,2}}^\theta \|v\|_2^{1-\theta},
\end{align*}
where $1/p = (1-\theta)/2 + \theta/2^*$, from the classical Gagliardo-Nirenberg inequality for functions on $\R^d$, see \cite[p.~125]{Nirenberg}, and the boundedness of the extension operators $E_k$, see Section~\ref{Subsec: Sobolev spaces}. We pick $v = \e^{-tL} u$ and obtain with the aid of the previous bound
\begin{align*}
  \|\e^{-tL} u  \|_p
 \lesssim  \|\e^{-tL} u  \|_{\IW^{1,2}}^\theta  \|\e^{-tL} u  \|_2^{1-\theta}
 \lesssim t^{-\theta/2} \|u\|_2,
\end{align*}
where we have also used the semigroup properties $\|\e^{-tL} u \|_2 \leq \|u\|_2$, $\|L \e^{-tL} u \|_2 \lesssim t^{-1} \|u\|_2$. Implicit constants depend on $p$, ellipticity, dimensions, and geometry. Substituting the value of $\theta$, this turns out just to be $\L^2 \to \L^p$ boundedness of $\{\e^{-tL}\}_{t>0}$.  
\end{proof}

We cite the following regularity result for the operator $\Lop: \IW^{1,2}(\Omega) \to \IW^{-1,2}(\Omega)$, whose maximal restriction to $\L^2(\Omega)^m$ is $L$, see Section~\ref{Subsec: L}. Essentially, this follows from {\v{S}}ne{\u\ii}berg's theorem \cite{Sneiberg-Original}, see also \cite[Thm.~1.3.25]{EigeneDiss}, but tracking the interpolation constants in order to deduce the required uniformity of the bounds is a non-trivial task.

\begin{proposition}[{\cite[Thm.~6.2]{HJKR}}]
\label{Prop: Domain is W1p}
Under Assumptions~\ref{D} and \ref{N} there exists $\eps'>0$ such that $\Lop$ extends/restricts to an isomorphism $\IW^{1,p}(\Omega) \to \IW^{-1,p}(\Omega)$ for all $p \in (2-\eps',2+\eps')$. In addition, $\eps'$ and upper and lower bounds for $\Lop$ can be given in terms of ellipticity, dimensions, and geometry. 
\end{proposition}

Now, we are ready to give the proof of Theorem~\ref{Thm: Range for Lp realization}. Let us stress that our argument essentially differs from the whole space case \cite[Sec.~4.2]{Riesz-Transforms} in that it avoids a change of variables for the coefficients $A$. This is necessary since the resulting change of the underlying domain would affect geometric constants in an uncontrollable way.

\begin{proof}[Proof of Theorem~\ref{Thm: Range for Lp realization}]
In view of Lemma~\ref{Lem: JA contains Sobolev conjugates} we only need to prove the extrapolation from the range $(2_*, 2^*)$ in the case $d \geq 3$ under Assumptions~\ref{D} and \ref{N}.

Let $\eps'>0$ be as provided by Proposition~\ref{Prop: Domain is W1p}. We fix $p$, $q$, and $r$ such that
\begin{align*}
 \max\{2-\eps', 2_*, 1^* \} < p < q < r < 2,
\end{align*}
which is possible since $d \geq 3$ implies $1^* < 2$. We will prove $r_* \in \cJ(L)$ with a bound depending on $p$, $q$, $r$, ellipticity, and geometry. This implies the claim: First, $p,q,r$ share the same dependencies as $\eps'$ and therefore we have $r_* = 2_* - \eps$ for some $\eps>0$ depending on ellipticity, dimensions, and geometry. Second, Riesz-Thorin interpolation of the $\L^{r_*}$ bound for the semigroup with the contractivity on $\L^2$ yields $\L^s$ bounds for $s \in (r_*,2)$ without introducing further implicit constants. Third, the same argument with the same choice of parameters applies to $L^*$ and by duality we obtain $\L^s$ boundedness for $s \in (2, (r')^*)$, where $1/r' = 1 - 1/r$.

In order to prove $\L^{r_*}$ boundedness, we let $t>0$ and take $u$ in $\L^{p^*}$, a dense subspace of $\L^2 \cap \L^{r_*}$. By Cauchy's integral formula and since $\Lop$ extends $L$, we can write
\begin{align*}
 \e^{-tL} u = L \e^{-tL} L^{-1} u = - \frac{1}{2 \pi \i} \oint_{|z-t| = R} \frac{1}{(z-t)^2} \e^{-zL}\Lop^{-1}u \; \d z,
\end{align*}
where $R = \dist(t, \partial \S^+_\psi)/2$ and $\psi = \pi/4 - \omega/2$. We have $p^* \in (2,2^*)$, so $p^* \in \cJ(L)$ thanks to Lemma~\ref{Lem: JA contains Sobolev conjugates}. Proposition~\ref{Prop: Hypercontractivity} implies $\L^{q^*}$ boundedness of the semigroup for complex times $z \in \S^+_\psi$ and in particular along the integration contour above. Thus, we have 
\begin{align*}
  \|\e^{-tL}u  \|_{q^*} \lesssim t^{-1} \|\Lop^{-1} u\|_{q^*}.
\end{align*}
Now, $\Lop$ extends to an isomorphism $\IW^{1,q}(\Omega) \to \IW^{-1,q}(\Omega)$ by choice of $q$. Since $1^* < q < 2$, we obtain from Sobolev embeddings,
\begin{align*}
 \|\Lop^{-1} u\|_{q^*} \lesssim  \|\Lop^{-1} u\|_{\IW^{1,q}} \lesssim \|u\|_{\IW^{-1,q}} \lesssim \|u\|_{q_*}.
\end{align*}
Altogether, $\|\e^{-tL}u  \|_{q^*} \lesssim t^{-1} \|u\|_{q_*}$, that is, the semigroup is $\L^{q_*} \to \L^{q^*}$ bounded. Riesz-Thorin interpolation with the $\L^2$ off-diagonal estimates from Proposition~\ref{Prop: OD estimates for semigroup on L2} leads to $\L^{r_*} \to \L^s$ off-diagonal estimates for some $s>r_*$ determined by $q$ and $r$, which in turn implies $\L^{r_*}$ boundedness (Lemma~\ref{Lem: OD implies bounded}). 
\end{proof}
%%%%%%%%%%%%%%%%%%%%%%%%%%%%%%%%%%%%%%%%%%%%%%%%%%%%%%%%%%%%%%%%%%%%%%%%%%%%%%%%%%%%%%%%%%%%%%%%%%%%%%%%%%%%%%%%%%%%%%%%%%%%%%%%%%%%%%%%%%%%%%%%%%%%%%%%%%%%%%%%%%%%
\section{Proof of Theorem~\ref{Thm: Hinfty in Lp}}
\label{Sec: Hinfty}

To obtain $\L^p$ estimates for the functional calculus for $L$ it will be convenient to calculate $f(L)$ in terms of the semigroup instead of the resolvent. This can be seen as some kind of Laplace transform inversion.

\begin{lemma}
\label{Lem: Functional calculus via Laplace transform}
Let $\omega < \theta < \nu < \psi < \pi/2$ and $g \in \H_0^\infty(\S^+_\psi)$. Put $\Gamma_{\pm} = (0,\infty)\e^{\pm \i (\pi/2 - \theta)}$ and $\gamma_{\pm} = (0,\infty)\e^{\pm \i \nu}$. Then $g(L)$ can also be computed as an $\L^2(\Omega)^m$-valued Bochner integral
\begin{align*}
 g(L) = \int_{\Gamma_+} \e^{-zL} \eta_+(z) \; \d z - \int_{\Gamma_{-}} \e^{-zL} \eta_{-}(z) \; \d z,
\end{align*}
where
\begin{align*}
 \eta_{\pm}(z) = \frac{1}{2 \pi \i} \int_{\gamma_{\pm}} \e^{z \xi} g(\xi) \; \d \xi \qquad (z \in \Gamma_\pm).
\end{align*}
\end{lemma}

\begin{proof}
Let $\xi \in \gamma_{\pm}$. For $z \in \Gamma_\pm$ we have $|\arg(z \xi)| = \frac{\pi}{2} - \theta + \nu > \frac{\pi}{2}$. Consequently, $(\xi - L)^{-1} \e^{z \xi} \e^{-zL}$ vanishes as $|z| \to \infty$ along the ray $\Gamma_\pm$ and we may compute, using the fundamental theorem of calculus,
\begin{align*}
 \int_{\Gamma_\pm} \e^{z \xi} \e^{-zL} \; \d z = \int_{\Gamma_\pm} \frac{\d}{\d z}\bigg((\xi - L)^{-1} \e^{z \xi} \e^{-zL}\bigg) \; \d z = -(\xi - L)^{-1}.
\end{align*}
By definition of the functional calculus
\begin{align*}
 g(L) = - \frac{1}{2 \pi \i} \int_{\gamma_+} g(\xi) (\xi - L)^{-1} \; \d \xi + \frac{1}{2 \pi \i} \int_{\gamma_{-}} g(\xi) (\xi - L)^{-1} \; \d \xi.
\end{align*}
From these two identities the claim follows by an application of Fubini's theorem.
\end{proof}

Next, we recall an important weak type $(p,p)$ criterion for bounded operators on $\L^2(\Xi)$ that goes back to \cite{BK}. If $\Xi = \R^d$, Proposition~\ref{Prop: BK Theorem} below is exactly the simplified version presented in \cite[Theorem~1.1]{Riesz-Transforms}, see also the subsequent Remark (7) in \cite{Riesz-Transforms} concerning vector-valued extensions. The result below on general measurable sets $\Xi$ is not mentioned therein but follows easily:  Take $R$ as the canonical restriction $\R^d \to \Xi$ and $E$ as the extension $\Xi \to \R^d$ by zero. Then observe that the $\R^d$-version applies to $T' := ETR$ and $A_r' := EA_r R$ with the same parameters and that $T$ and $T'$ have the same $\L^p$ bound.

\begin{proposition}
\label{Prop: BK Theorem}
Let $q \in [1,2)$. Let $T: \L^2(\Xi)^{m_1} \to \L^2(\Xi)^{m_2}$ be a bounded linear operator, where $m_1, m_2 \in \IN$ and $\Xi \subseteq \R^d$ is measurable. Assume there exists a family $\{A_r\}_{r>0}$ of bounded linear operators on $\L^2(\Xi)^{m_1}$ with the following properties: For $j \geq 2$,
\begin{align}
\label{Eq: BK Condition 1}
\bigg(\int_{C_j(B) \cap \Xi} |T(1-A_{r})u|^2 \bigg)^{1/2} \leq g(j) r^{d/2 - d/q} \bigg(\int_{B \cap \Xi} |u|^{q} \bigg)^{1/q}
\end{align}
and for $j \geq 1$,
\begin{align}
\label{Eq: BK Condition 2}
\bigg(\int_{C_j(B) \cap \Xi} |A_{r}u|^2 \bigg)^{1/2} \leq g(j) r^{d/2 - d/q} \bigg( \int_{B \cap \Xi} |u|^{q} \bigg)^{1/q},
\end{align}
whenever $B \subseteq \R^d$ is an open ball with radius $r$ and $u \in \L^2(\Xi)^{m_1}$ has support in $B \cap \Xi$. If $\Sigma := \sum g(j) 2^{dj/2}$ is finite, then $T$ is of weak type $(q,q)$ and hence $\L^p$ bounded for $p \in (q,2)$ with a bound depending on $q, m_1, m_2, \Sigma$, and an $\L^2$ bound for $T$.
\end{proposition}

As a first application we prove

\begin{lemma}
\label{Lem: OD implies Hinfty}
Suppose $\{\e^{-tL}\}_{t>0}$ satisfies $\L^{q} \to \L^2$ off-diagonal estimates for some $q \in (1,2)$. Then
\begin{align*}
 \|f(L)u\|_p \leq C \|f\|_\infty \|u\|_p \qquad (f \in \H_0^\infty(\S^+_\psi),\, u \in \L^2(\Omega)^m ),
\end{align*}
whenever $\psi \in (\omega, \pi)$ and $p \in (q,2)$. Here, $C$ depends on $\psi, p, q$, ellipticity, dimensions, geometry and constants implicit in the assumption.
\end{lemma}

\begin{proof}
Without loss of generality we may assume $\psi < \pi/2$. Let $f \in \H_0^\infty(\S^+_\psi)$ be normalized such that $\|f\|_\infty = 1$. We appeal to Proposition~\ref{Prop: BK Theorem} with $T = f(L)$. We put $A_r = 1 - (1- \e^{-r^2 L})^n$, where $n \in \IN$ has to be determined yet. Proposition~\ref{Prop: Hinfty on L2} yields $\|T\|_{\L^2 \to \L^2} \leq 4$ and we need to check \eqref{Eq: BK Condition 1} and \eqref{Eq: BK Condition 2}. 

For the argument we put $\gamma := d/q - d/2 > 0$, let $B \subseteq \R^d$ be an open ball with radius $r>0$, and $u \in \L^2(\Omega)^m$ have its support in $B \cap \Omega$. Having expanded
\begin{align*}
A_r = \sum_{k=1}^n \binom{n}{k} (-1)^{k-1} \e^{-kr^2 L}, 
\end{align*}
the assumed $\L^{q} \to \L^2$ off-diagonal estimates (with constants $C,c \in (0,\infty)$) yield for $j \geq 1$,
\begin{align*}
\bigg(\int_{C_j(B) \cap \Omega} |A_r u|^2 \bigg)^{1/2} \leq
C 2^n  \e^{-c(2^j -2)^2/n^2} r^{-\gamma} \bigg(\int_{B \cap \Omega} |u|^{q} \bigg)^{1/q}.
\end{align*}
Hence, \eqref{Eq: BK Condition 2} holds with $g(j)2^{dj/2}$ summable no matter the value of $n$. 

Turning to \eqref{Eq: BK Condition 1}, we apply Lemma~\ref{Lem: Functional calculus via Laplace transform} to the function $g(z) = f(z)(1-\e^{-r^2 z})^n$ and write
\begin{align}
\label{Eq1: OD implies Hinfty}
 T(1-A_r)u =  \int_{\Gamma_+} \eta_+(z) \e^{-zL}u  \; \d z - \int_{\Gamma_{-}} \eta_{-}(z) \e^{-zL}u  \; \d z,
\end{align}
where
\begin{align*}
 \eta_{\pm}(z) = \frac{1}{2 \pi \i} \int_{\gamma_{\pm}} \e^{z \xi} g(\xi) \; \d \xi \qquad (z \in \Gamma_\pm).
\end{align*}
For $j \geq 2$ we take $\L^2(C_j(B) \cap \Omega)$-norms in \eqref{Eq1: OD implies Hinfty} and apply off-diagonal estimates to give
\begin{align}
\label{Eq2: OD implies Hinfty}
\bigg(\int_{C_j(B) \cap \Omega} |T(1-A_r)u|^2 \bigg)^{1/2}
&\lesssim (I_{j,+} + I_{j,-}) \bigg(\int_{B \cap \Omega} |u|^{q} \bigg)^{1/q},
\end{align}
where
\begin{align*}
 I_{j,\pm} = \int_{\Gamma_\pm} C |\eta_\pm(z)|  |z|^{-\gamma/2} \e^{-c 4^{j-1} r^2/|z|} \; \d |z|.
\end{align*}
By the mean value theorem and the normalization of $f$ we have $|g(\xi)| \leq \min\{2^n, r^{2n}|\xi|^n \}$ for $\xi \in \gamma_\pm$. Consequently, 
\begin{align*}
 |\eta_\pm(z)| \leq \alpha |z|^{-1} \min\{1, r^{2n} |z|^{-n} \} \qquad (z \in \Gamma_\pm),
\end{align*}
where $\alpha$ depends on $\psi$, $\omega$, $n$. Setting $|z| = t$, we deduce that
\begin{align*}
 I_{j,\pm}
 &\leq \alpha \e^{-c 4^{j-1}/2} \int_0^{r^2} t^{-\gamma/2}  \e^{-c 4^{j-1} r^2 /(2t)} \; \frac{\d t}{t} \\
 &\quad + \alpha r^{2n} \int_{r^2}^\infty t^{-\gamma/2-n}  \e^{-c 4^{j-1} r^2 /t} \; \frac{\d t}{t} \\
 &\leq\alpha r^{-\gamma} 2^{-\gamma(j-1)} \bigg( \e^{-c 4^{j-1}/2}  \int_0^\infty s^{-\gamma/2}  \e^{-c/(2s)} \; \frac{\d s}{s} \\
 &\qquad \qquad \qquad \qquad + 4^{-(j-1)n} \int_0^\infty s^{-\gamma/2-n} \e^{-c/s} \; \frac{\d s}{s} \bigg).
\end{align*}
The remaining integrals in $s$ are finite. Thus, we have found $I_{j,\pm} \leq g(j) r^{-\gamma}$ with $\{2^{dj/2} g(j)\}_{j \geq 2}$ summable provided $\gamma+ 2n > d/2$. For such choice of $n$, \eqref{Eq: BK Condition 1} follows from \eqref{Eq2: OD implies Hinfty}.
\end{proof}

Now we can complete the

\begin{proof}[Proof of Theorem~\ref{Thm: Hinfty in Lp}]
The necessity part follows simply because we can take $f(z) = \e^{-tz}$ in \eqref{Eq: Hinfty in Lp} for every $t>0$.

By duality it suffices to treat the sufficiency part in the case $p_0 \in \cJ(L) \cap [1,2)$. Let $p \in (p_0,2]$ and $\psi \in (\omega, \pi)$. Proposition~\ref{Prop: Hypercontractivity} provides $\L^q \to \L^2$ off-diagonal estimates for the semigroup, for instance for the choice $q=(p+p_0)/2$, and implied constants depend on $\Jnorm{p_0}$, $p$, and geometry. See Remark~\ref{Rem: Hypercontractivity} for the latter. Lemma~\ref{Lem: OD implies Hinfty} yields
\begin{align*}
 \|f(L)u\|_p \leq C \|f\|_\infty \|u\|_p \qquad (f \in \H_0^\infty(\S^+_\psi), u \in \L^2(\Omega)^m),
\end{align*}
with $C$ depending on  $\Jnorm{p_0}$, $p$, $\psi$, and geometry. This bound extends to $f \in \H^\infty(\S^+_\psi)$ and $u \in \L^2(\Omega)^m$, see Section~\ref{Subsec: Functional calculus}. Indeed, if $f_n \in \H_0^\infty(\S^+_\psi)$ are such that $f_n \to f$ pointwise, $\|f_n\|_\infty \to \|f\|_\infty$, and $f_n(L)u \to f(L)u$ in $\L^2$, then $f_n(L)u \to f(L)u$ also in $\L^p$ since $\Omega$ is bounded.
\end{proof}
%%%%%%%%%%%%%%%%%%%%%%%%%%%%%%%%%%%%%%%%%%%%%%%%%%%%%%%%%%%%%%%%%%%%%%%%%%%%%%%%%%%%%%%%%%%%%%%%%%%%%%%%%%%%%%%%%%%%%%%%%%%%%%%%%%%%%%%%%%%%%%%%%%%%%%%%%%%%%%%%%%%%
\section{\texorpdfstring{$\L^p$}{Lp} bounds for the Riesz transform}
\label{Sec: Riesz}

As a primer to Theorem~\ref{Thm: Square root solution Lp} we study $\L^p$ boundedness of the Riesz transform $\nabla L^{-1/2}$. Due to Theorem~\ref{Thm: Kato} this is an $\L^2$ bounded operator. It follows from \eqref{Eq: Resolution of the identity} that
\begin{align}
\label{Eq: Riesz transform representation}
\nabla L^{-1/2} u = \frac{2}{\sqrt{\pi}} \int_0^\infty \nabla \e^{-t^2 L}u \; \d t \qquad (u \in \L^2(\Omega)^m)
\end{align}
in the sense of an improper Riemann integral.

\begin{lemma}
\label{lem: OD implies Riesz transform}
Suppose $\{\e^{-tL}\}_{t>0}$ satisfies $\L^q \to \L^2$ off-diagonal estimates for some $q \in (1,2)$. Then $\nabla L^{-1/2}$ is $\L^p$ bounded for every $p \in (q,2)$. The bound depends on $p$, $q$, ellipticity, dimensions, geometry and constants implicit in the assumption.
\end{lemma}

\begin{proof}
We appeal to Proposition~\ref{Prop: BK Theorem} with $T = \nabla L^{-1/2}$ and $A_r = 1 - (1- \e^{-r^2 L})^n$, where $n \in \IN$ has to be determined yet. We have seen that $T$ is $\L^2$ bounded. From the proof of Lemma~\ref{Lem: OD implies Hinfty} we also know that  $\L^q \to \L^2$ off-diagonal estimates for the semigroup imply \eqref{Eq: BK Condition 2} for any choice of $n$. So, we only have to check \eqref{Eq: BK Condition 1}. 

For the argument we put $\gamma := d/q - d/2 > 0$, let $B \subseteq \R^d$ be an open ball with radius $r>0$, and $u \in \L^2(\Omega)^m$ have its support in $B \cap \Omega$.  We calculate $T(1-A_r)u$ via \eqref{Eq: Riesz transform representation} and expand $A_r$ by the binomial theorem. This leads to the formula
\begin{align*}
 T(1-A_r) u 
= \frac{2}{\sqrt{\pi}} \int_0^\infty \sum_{k=0}^n (-1)^k \binom{n}{k} \nabla \e^{-(t^2 + kr^2)L} u \; \d t, 
\end{align*}
which, by substituting $t^2/r^2 + k$, can more conveniently be written as
\begin{align*}
T(1-A_r) u = \frac{1}{\sqrt{\pi}} \int_0^\infty g(t) r \nabla \e^{-r^2tL} u \;\d t,
\end{align*}
where
\begin{align}
\label{Eq1: OD implies Riesz transform}
 g(t) = \sum_{k=0}^n \binom{n}{k} (-1)^k \frac{\ind_{(0,\infty)}(t- k)}{\sqrt{t - k}}.
\end{align}
Proposition~\ref{Prop: Hypercontractivity}.\eqref{iv Hypercontractivity} yields $\L^q \to \L^2$ off-diagonal estimates for $\{\sqrt{t} \nabla \e^{-tL}\}_{t>0}$. Let $C,c \in (0,\infty)$ be the implied constants and recall from Remark~\ref{Rem: Hypercontractivity} that they do not bring in further dependencies. Taking $\L^2(C_j(B) \cap \Omega)$-norms in the above formula, we find for $j \geq 2$,
\begin{align}
\label{Eq2: OD implies Riesz transform}
 \bigg(\int_{C_j(B) \cap \Omega} |T(1-A_r)u|^2 \bigg)^{1/2}
\leq
C \pi^{-1/2} r^{-\gamma} (I_{-} + I_{+})  \bigg(\int_{B \cap \Omega} |u|^{q} \bigg)^{q},
\end{align}
where
\begin{align*}
 I_{-} = \int_0^{4n} |g(t)| t^{-\gamma/2-1/2} \e^{- c 4^{j-1}/t} \; \d t, \qquad I_{+} = \int_{4n}^\infty |g(t)| t^{-\gamma/2-1/2} \e^{- c 4^{j-1}/t} \; \d t.
\end{align*}
It remains to bound these integrals. We begin with the crude estimate
\begin{align*}
 I_{-} \leq \e^{-c \frac{4^{j-1}}{8n}} \int_0^{4n} |g(t)| t^{-\frac{\gamma}{2}-\frac{1}{2}} \e^{-\frac{c}{2t}} \; \d t,
\end{align*}
where the remaining integral is finite since $|g|$ is integrable on $(0,4n)$ and the other factor remains bounded as $t \to 0$. As for $I_+$, we first note that for $t> 4n$ all characteristic functions in \eqref{Eq1: OD implies Riesz transform} evaluate to $1$. Hence, the residue theorem yields
\begin{align*}
 g(t) = \frac{(-1)^n}{2 \pi \i} \oint_{|z|= t/2} \frac{n!}{z(z-1)\cdots(z-n)} \frac{1}{\sqrt{t-z}} \; \d z.
\end{align*}
Along the path of integration $t-z, z,z-1,\ldots, z-n$ are of absolute value at least $t/4$ each. Thus $|g(t)| \leq \alpha t^{-n-1/2}$ for some $\alpha \in (0,\infty)$ depending on $n$. In conclusion,
\begin{align*}
 I_+
 \leq \alpha \int_{4n}^\infty t^{-\gamma/2-n} \e^{- c 4^{j-1}/t} \; \frac{\d t}{t}
 \leq \alpha 2^{-\gamma(j-1)} 4^{-(j-1)n} \int_0^\infty s^{-\gamma/2-n} \e^{-c/s} \; \frac{\d s}{s},
\end{align*}
and the integral in $s$ is finite. We have found $I_{-} + I_{+} \leq g(j)$ with $\{2^{dj/2} g(j)\}_{j \geq 2}$ summable provided $\gamma+ 2n > d/2$. For such choice of $n$, \eqref{Eq: BK Condition 1} follows from \eqref{Eq2: OD implies Riesz transform}.
\end{proof}

\begin{corollary}
\label{Cor: Riesz estimate}
If $p_0 \in \cJ(L) \cap [1,2)$, then for every $p \in (p_0,2)$ the lower bound
\begin{align*}
 \|u\|_{\IW^{1,p}} \lesssim \|L^{1/2} u\|_{p} \qquad (u \in \IW^{1,2}(\Omega)).
\end{align*}
The implied constant depends on $\Jnorm{p_0}$, $p$, and geometry.
\end{corollary}

\begin{proof}
Let $p \in (p_0,2)$. Proposition~\ref{Prop: Hypercontractivity} provides $\L^q \to \L^2$ off-diagonal estimates for $\{\e^{-tL}\}_{t>0}$, for instance for the choice $q=(p+p_0)/2$, and implied constants depend on $\Jnorm{p_0}$, $p$, and geometry, see Remark~\ref{Rem: Hypercontractivity}. Hence, Lemma~\ref{lem: OD implies Riesz transform} applies. An equivalent way of stating its conclusion is the estimate 
\begin{align*}
 \|\nabla u\|_p \lesssim \|L^{1/2}u \|_p \qquad (u \in \IW^{1,2}(\Omega)),
\end{align*}
where the implied constant shares the same dependencies. To add $\|u\|_p$ on the left-hand side, we first interpolate the assumed $\L^{p_0}$ bound for the semigroup with the exponential decay on $\L^2$ stated in Lemma~\ref{Lem: Expontial stability}. This yields $\|\e^{-tL}\|_{\L^p \to \L^p} \leq \Jnorm{p_0}^{1-\theta} \e^{-\lambda \theta t/2}$ for $t>0$, where $1/p = (1-\theta)/p_0 + \theta/2$. Now we can use \eqref{Eq: Resolution of the identity} to give
\begin{align*}
 \|u\|_p \leq \frac{2}{\sqrt{\pi}} \int_0^\infty \Jnorm{p_0}^{1-\theta} \e^{- \lambda \theta t^2 /2} \|L^{1/2} u\|_p \; \d t = \frac{\sqrt{2} \Jnorm{p_0}^{1-\theta}}{\sqrt{\lambda \theta}} \|L^{1/2} u\|_p. & \qedhere
\end{align*}
\end{proof}

%%%%%%%%%%%%%%%%%%%%%%%%%%%%%%%%%%%%%%%%%%%%%%%%%%%%%%%%%%%%%%%%%%%%%%%%%%%%%%%%%%%%%%%%%%%%%%%%%%%%%%%%%%%%%%%%%%%%%%%%%%%%%%%%%%%%%%%%%%%%%%%%%%%%%%%%%%%%%%%%%%%%
\section{A Calder\'{o}n--Zygmund decomposition}
\label{Sec: CZ decomposition}

\noindent In this section we shall craft a Calder\'{o}n--Zygmund decomposition within $\IW^{1,p}$. We extend the approach from \cite{ABHR}. The crucial insight in this paper was that the following Hardy inequality can be used to maintain Dirichlet boundary conditions for the `good' and all `bad' functions. For a proof see  \cite[Thm.~6.1]{ABHR} or \cite[Thm.~3.2]{Hardy-Poincare}. We agree on $\dist(x, \emptyset) = \infty$ for $x \in \R^d$ so that the estimate below holds trivially for empty Dirichlet parts.

\begin{proposition}
\label{Prop: Hardy}
Suppose $\Omega$ is a bounded domain that satisfies Assumptions~\ref{D} and \ref{N} and let $p \in (1,\infty)$. For every $k=1,\ldots,m$ there is a Hardy-type inequality
\begin{align*}
 \int_\Omega \bigg| \frac{v}{\dist_{D_k}} \bigg|^p \lesssim \int_\Omega |\nabla v|^p \qquad (v \in \W^{1,p}_{D_k}(\Omega)).
\end{align*}

\end{proposition}

We let $\cQ$ be the system of closed axis-parallel cubes with non-empty interior in $\R^d$. The Hardy-Littlewood maximal operator $M$ is defined for $u \in \Lloc^1(\R^d)$ by
\begin{align*}
 (M u)(x) := \sup_{x \in Q \in \cQ} \barint_Q |u| \qquad (x \in \R^d),
\end{align*}
where $\barint_Q := |Q|^{-1}\int_Q$ denotes the average over $Q$. Then $|u| \leq M u$ holds almost everywhere and for some constant $c_d > 0$ depending only on $d$ there is a weak type estimate
\begin{align*}
 \big|\big\{x \in \R^d : \, |(M u)(x)| > \alpha \big\}\big| \leq \frac{c_d}{\alpha} \|u\|_{\L^1(\R^d)} \qquad (\alpha > 0, \, u \in \L^1(\R^d)),
\end{align*}
see for example \cite[Ch.~2]{Grafakos}. Let us denote the coordinates of a $\IC^m$-valued function $v$ by $v^{(k)}$.

\begin{lemma}
\label{Lem: Adapted CZ decomposition}
Suppose $\Omega$ is a bounded domain with Assumptions~\ref{D} and \ref{N} and let $1 < p < \infty$. For every $u \in \IW^{1,p}(\Omega)$ and every $\alpha > 0$ there exists a countable index set $J$, cubes $Q_j \in \cQ$, $j \in J$, and measurable functions $g, b_j: \Omega \to \IC^m$ such that for some $C \geq 1$, independent of $u$ and $\alpha$, the following hold.
\begin{enumerate} 
 \item \label{i} $\displaystyle u = g + \sum_{j \in J} b_j$ pointwise almost everywhere. \\[3pt]

 \item \label{ii} Each $b_j$ has support in $Q_j$ and each $x \in \R^d$ is contained in at most $12^d$ of the $Q_j$. \\[3pt]
 
 \item \label{iii} $g \in \IW^{1,\infty}(\Omega)$ with $\displaystyle \|g \|_{\IW^{1,\infty}(\Omega)} + \sum_{k=1}^m \|g^{(k)}/\dist_{D_k}\|_{\L^\infty(\Omega)} \leq C \alpha$. \\[3pt]

 \item \label{iv} $b_j \in \IW^{1,p}(\Omega)$ with $\displaystyle \int_\Omega \bigg(\fabs{\nabla b_j}^p + \fabs{b_j}^p + \sum_{k=1}^m \bigg|\frac{b_j^{(k)}}{\dist_{D_k}} \bigg|^p \bigg) \leq C \alpha^p \fabs{Q_j}$ for every $j \in J$. \\[3pt]

 \item \label{v} $\displaystyle \sum_{j \in J} \fabs{Q_j} \leq \frac{C}{\alpha^p} \|u\|_{\IW^{1,p}(\Omega)}^p$. \\[3pt]

 \item \label{vi} $g \in \IW^{1,p}(\Omega)$ with $\|g\|_{\IW^{1,p}(\Omega)} \leq C \|u\|_{\IW^{1,p}(\Omega)}$. \\[3pt]

 \item \label{vii} If $u \in \L^q(\Omega)^m$ for some $q \in [1,\infty)$, then also $b_j \in \L^q(\Omega)^m$ for every $j \in J$. \\[3pt]
 
 \item \label{viii} For every subset $J' \subseteq J$ the sum $\sum_{j \in J'} b_j$ converges unconditionally in $\IW^{1,p}(\Omega)$.
\end{enumerate}
\end{lemma}

\begin{proof}
The proof follows the classical pattern relying on a Whitney decomposition of an exceptional set determined by an adapted maximal function. It is divided into seven steps.

\subsubsection*{Step 1: Adapted maximal function}

Recall the bounded Sobolev extension operators $E_k$ introduced in Section~\ref{Subsec: Sobolev spaces}. Then $\tildu^{(k)}:= E_k u^{(k)}$, $k =1,\ldots,m$, defines an extension $\tildu \in \IW^{1,p}(\R^d)$ of $u$. Since $\Omega$ is bounded, a suitable smooth truncation far away from $\cl{\Omega}$ allows us to modify the extension in such a way that even the Hardy-type terms as in Proposition~\ref{Prop: Hardy} are controlled, that is to say,
\begin{align}
\label{Eq: Def of CE}
\int_{\R^d} \bigg(|\tildu| + |\nabla \tildu| + \sum_{k=1}^m \bigg|\frac{\tildu^{(k)}}{\dist_{D_k}} \bigg|\bigg)^p \lesssim \int_\Omega \bigg(|u| + |\nabla u|\bigg)^p.
\end{align}
The procedure is explained in detail on p.176 of \cite{ABHR}. We define the open set
\begin{align*}
 U:= \bigg \{x \in \R^d : \, M \bigg(\Big(|\nabla \tildu| + |\tildu| + \sum_{k=1}^m \Big|\frac{\tildu^{(k)}}{\dist_{D_k}} \Big|\Big)^p \bigg) > \alpha^p \bigg \}.
\end{align*}
First we treat the case $U = \emptyset$. Then for the choices $J = \emptyset$ and $g = u$ all assertions are immediate except for \eqref{iii}. Referring to this, we use that $\tildu$ is an extension of $u$ and obtain for almost every $x\in \Omega$,
\begin{align*}
 \bigg(|\nabla g(x)| + |g(x)| + \sum_{k=1}^m \bigg|\frac{g^{(k)}(x)}{\dist_{D_k}(x)} \bigg|\bigg)^p
&= \bigg(|\nabla \tildu(x)| + |\tildu(x)| + \sum_{k=1}^m \bigg|\frac{\tildu^{(k)}(x)}{\dist_{D_k}(x)} \bigg| \bigg)^p.
\end{align*}
The right-hand side is dominated almost everywhere by its maximal function, which in turn in globally bounded by $\alpha^p$. We shall discuss at the end of the proof in the general case why this implies $g \in \IW^{1,\infty}(\Omega)$.

So, from now on we can assume that $U$ is a non-empty open subset of $\R^d$. By the weak type estimate for the maximal operator and \eqref{Eq: Def of CE} we obtain
\begin{align}
\label{Eq1: Adapted CZ decomposition}
\begin{split}
 \fabs{U} 
&\lesssim \frac{1}{\alpha^p} \bigg \| |\nabla \tildu| + |\tildu| + \sum_{k=1}^m \bigg|\frac{\tildu^{(k)}}{\dist_{D_k}} \bigg| \bigg\|_p^p
\lesssim \frac{1}{\alpha^p} \|u\|_{\IW^{1,p}}^p < \infty.
\end{split}
\end{align}
In particular, $F:= \R^d \setminus U$ is non-empty. This allows for choosing a Whitney decomposition of $U$, that is, an at most countable index set $J$ and a collection of cubes $Q_j \in \cQ$, $j \in J$, with diameter $d_j$ that satisfy
\begin{alignat*}{2}
 &\mathrm{(A)} \quad U = \bigcup_{j \in J} \frac{8}{9} Q_j, \hspace{50pt}
 &&\mathrm{(B)} \quad {\frac{8}{9}  Q_j^\circ} \cap {\frac{8}{9} Q_k^\circ} = \emptyset \text{ if $j \neq k$}, \\
 &\mathrm{(C)} \quad Q_j \subseteq U \text{ for all $j$}, \hspace{0pt}
 &&\mathrm{(D)} \quad \sum_{j \in J} \ind_{Q_j} \leq 12^d, \\
 &\mathrm{(E)} \quad \frac{5}{6} d_j \leq \dist(Q_j, F) \leq 4 d_j \text{ for all $j$},
\end{alignat*}
see \cite[Lemma 5.5.1/2]{Bennett-Sharpley} for this classical tool but replace the cubes $Q$ by their enlarged counterparts $\frac{9}{8} Q$ therein. Here, $Q^\circ$ denotes the interior of $Q$. Two important consequences can be recorded immediately: Firstly, (E) implies
\begin{align}
\label{Eq2: Adapted CZ decomposition: Enlarged cubes intersect complement}
12 \sqrt{d} Q_j \cap F \neq \emptyset \qquad (j \in J).
\end{align}
Secondly, (D) in combination with \eqref{Eq1: Adapted CZ decomposition} immediately implies \eqref{v} since
\begin{align}
\label{Eq99: Adapted CZ decomposition}
 \sum_{j \in J} \fabs{Q_j} \leq \int_U  \sum_{j \in J} \ind_{Q_j} \leq 12^d \fabs{U} \lesssim \frac{1}{\alpha^p} \|u\|_{\IW^{1,p}(\Omega)}^p.
\end{align}

\subsubsection*{Step 2: Definition of the good and bad functions}

Let $\{\varphi_j\}_{j \in J}$ be a partition of unity on $U$, that is $\sum_{j \in J} \varphi_j = 1$ on $U$, with the properties
\begin{alignat*}{2}
 &\mathrm{(a)} \quad \varphi_j \in \C^\infty(\R^d) \hspace{50pt}
 &&\mathrm{(b)} \quad \supp \varphi_j \subseteq Q_j^\circ \\
 &\mathrm{(c)} \quad \varphi_j = 1 \text{ on $\frac{8}{9} Q_j$} \hspace{50pt}
 &&\mathrm{(d)} \quad \|\varphi_j\|_\infty + d_j \|\nabla \varphi_j\|_\infty \lesssim 1
\end{alignat*}
for all $j \in J$, see \cite[Sec.~5.5]{Bennett-Sharpley} for the construction. Given $1 \leq k \leq m$, we distinguish between three properties a cube $Q_j$ can have:
\begin{alignat*}{1}
 &\bullet \quad \text{$Q_j$ is \emph{$k$-usual} if $d_j < 1 $ and $\dist(Q_j, D_k) \geq d_j$,} \\
 &\bullet \quad \text{$Q_j$ is \emph{$k$-boring} if $\dist(Q_j, D_k) \geq d_j \geq 1$,} \\
 &\bullet \quad  \text{$Q_j$ is \emph{$k$-special} if $\dist(Q_j, D_k) \leq d_j$.}
\end{alignat*}
Then we let $\tildu_{Q_j}^{(k)} := \barint_{Q_j} \tildu^{(k)}$ and define
\begin{align*}
 \tildb_j^{(k)} := \begin{cases}
		\varphi_j (\tildu^{(k)} - \tildu_{Q_j}^{(k)} ) & \text{if $Q_j$ is $k$-usual}\\
		\varphi_j \tildu^{(k)} & \text{if $Q_j$ is $k$-boring or $k$-special}
              \end{cases} \qquad (1 \leq k \leq m, \, j \in J).
\end{align*}
Setting $\tildg := \tildu - \sum_{j \in J} \tildb_j$ as well as $b_j:= \tildb_j|_\Omega$ and $g:= \tildg|_\Omega$, $j \in J$, these functions automatically satisfy \eqref{i}. Due to (D) there is no problem of convergence with this sum and also \eqref{ii} holds true. Moreover, \eqref{vii} holds since the extension operators $E_k$ are bounded $\L^q(\Omega) \to \L^q(\R^d)$ for every $1 \leq q < \infty$, see Section~\ref{Subsec: Sobolev spaces}.

Next, we check that the $b_j$ are contained in $\IW^{1,p}(\Omega)$: For fixed $1 \leq k \leq m$ we have $\tildb_j^{(k)} \in \W^{1,p}(\R^d)$ by construction. If $Q_j$ is either $k$-usual or $k$-boring, then $\dist(Q_j, D_k) \geq d_j > 0$ and via mollification $\tildb_j^{(k)}$ can be approximated by $\C_{D_k}^\infty(\R^d)$-functions in the norm of $\W^{1,p}(\R^d)$. If $Q_j$ is $k$-special, then $\tildu^{(k)} \in \W_{D_k}^{1,p}(\R^d)$ implies $\tildb_j^{(k)} = \varphi_j \tildu^{(k)} \in \W_{D_k}^{1,p}(\R^d)$.

\subsubsection*{Step 3: Proof of \eqref{iv}}

After the considerations above it remains to prove the estimate. To this end, we fix a coordinate $1 \leq k \leq m$.

We start with a $k$-usual cube, in which case $\nabla \tildb_j^{(k)} = \varphi_j \nabla \tildu^{(k)} + (\tildu^{(k)} - \tildu_{Q_j}^{(k)}) \nabla \varphi_j$. Using~(d),
\begin{align}
\label{Eq3: Adapted CZ decomposition}
\begin{split}
 \int_{Q_j} |\nabla \tildb_j^{(k)}|^p 
&\lesssim \int_{Q_j} \bigg(|\varphi_j \nabla \tildu^{(k)}|^p + |(\tildu^{(k)} - \tildu_{Q_j}^{(k)}) \nabla \varphi_j|^p \bigg)\\
&\lesssim \int_{Q_j} |\nabla \tildu^{(k)}|^p + \frac{1}{d_j^p} \int_{Q_j} |\tildu^{(k)} - \tildu_{Q_j}^{(k)}|^p,
\end{split}
\end{align}
where the rightmost integral can be estimated via Poincar\'{e}'s inequality
\begin{align}
\label{Eq4: Adapted CZ decomposition}
 \frac{1}{d_j^p} \int_{Q_j} |\tildu^{(k)} - \tildu_{Q_j}^{(k)}|^p
\lesssim \int_{Q_j} |\nabla \tildu^{(k)}|^p.
\end{align}
Invoking \eqref{Eq2: Adapted CZ decomposition: Enlarged cubes intersect complement}, we pick some $z_j \in Q_j^* \cap F$, where $Q_j^* = 12 \sqrt{d} Q_j$, in order to bring into play the maximal operator:
\begin{align}
\label{Eq5: Adapted CZ decomposition}
\begin{split}
 \int_{Q_j} |\nabla \tildb_j^{(k)}|^p
&\lesssim \int_{Q_j} \bigg(|\varphi_j \nabla \tildu^{(k)}|^p + |(\tildu^{(k)} - \tildu_{Q_j}^{(k)}) \nabla \varphi_j|^p \bigg)
\lesssim \int_{Q_j^*} |\nabla \tildu^{(k)}|^p \\
&\leq |Q_j^*| \barint_{Q_j^*} |\nabla \tildu^{(k)}|^p
\lesssim |Q_j| M(|\nabla \tildu|^p)(z_j).
\end{split}
\end{align}
Now, we capitalize $z_j \in F$ to give
\begin{align}
\label{Eq6: Adapted CZ decomposition}
 \int_\Omega |\nabla b_j^{(k)}|^p 
\leq \int_{Q_j} |\nabla \tildb_j^{(k)}|^p
 \lesssim \alpha^p \fabs{Q_j}.
\end{align}
The corresponding estimate for $|b_j^{(k)}|$ can be derived similarly, starting from
\begin{align}
\label{Eq7: Adapted CZ decomposition}
 \int_{\Omega} |b_j^{(k)}|^p
\leq \int_{Q_j} |\tildb_j^{(k)}|^p
= \int_{Q_j} |\tildu^{(k)} - \tildu_{Q_j}^{(k)}|^p|\varphi_j|^p
\lesssim d_j^p \int_{Q_j} |\nabla \tildu^{(k)}|^p
\leq \int_{Q_j} |\nabla \tildu^{(k)}|^p
\end{align}
and proceeding as in \eqref{Eq5: Adapted CZ decomposition} and \eqref{Eq6: Adapted CZ decomposition}. For the third term $b_j^{(k)}/\dist_{D_k}$ we note that on $k$-usual cubes $\dist_{D_k} \geq d_j$ holds and so by \eqref{Eq4: Adapted CZ decomposition} and the same argument as in \eqref{Eq5: Adapted CZ decomposition} and \eqref{Eq6: Adapted CZ decomposition} we get
\begin{align*}
 \int_{\Omega} \bigg|\frac{b_j^{(k)}}{\dist_{D_k}}\bigg|^p
\leq \int_{Q_j} \bigg| \frac{\tildb_j^{(k)}}{\dist_{D_k}}\bigg|^p
\lesssim \frac{1}{d_j^p} \int_{Q_j} |\tildu^{(k)} - \tildu_{Q_j}^{(k)}|^p
\lesssim \alpha^p \fabs{Q_j}.
\end{align*}

We turn to the $k$-boring cubes. Then $\tildb_j^{(k)} = \tildu^{(k)} \varphi_j$ and $\dist_{D_k} \geq d_j \geq 1$ a.e.\ on $Q_j$. By (d) we have
\begin{align}
\label{Eq11: Adapted CZ decomposition}
\begin{split}
|\tildb_j^{(k)}| + |\nabla \tildb_j^{(k)}| + \bigg| \frac{\tildb_j^{(k)}}{\dist_{D_k}} \bigg| 
&\leq |\tildu^{(k)} \varphi_j| + |\varphi_j \nabla \tildu^{(k)}| + |\tildu^{(k)} \nabla \varphi_j | + \bigg|\frac{\varphi_j \tildu^{(k)}}{\dist_{D_k}}\bigg|\\
&\lesssim |\tildu^{(k)}| + |\nabla \tildu^{(k)}| + \frac{1}{d_j} |\tildu^{(k)}| + \bigg| \frac{\tildu^{(k)}}{\dist_{D_k}} \bigg| \qquad(\text{a.e.\ on $Q_j$})
\end{split}
\end{align}
and the usual start of play for the maximal operator, following \eqref{Eq5: Adapted CZ decomposition} and \eqref{Eq6: Adapted CZ decomposition}, leads to
\begin{align}
\label{Eq10: Adapted CZ decomposition}
\begin{split}
\int_{\Omega} \bigg(|b_j^{(k)}|^p + |\nabla b_j^{(k)}|^p + \bigg|\frac{b_j^{(k)}}{\dist_{D_k}}\bigg|^p \bigg)
&\leq \int_{Q_j} \bigg(|\tildb_j^{(k)}| + |\nabla \tildb_j^{(k)}| + \bigg|\frac{\tildb_j^{(k)}}{\dist_{D_k}}\bigg| \bigg)^p\\
&\lesssim \int_{Q_j} \bigg(|\tildu^{(k)}| + |\nabla \tildu^{(k)}| + \frac{1}{d_j} |\tildu^{(k)}| + \bigg|\frac{\tildu^{(k)}}{\dist_{D_k}}\bigg| \bigg)^p \\
&\lesssim \int_{Q_j} \bigg( |\tildu^{(k)}| + |\nabla \tildu^{(k)}| + \bigg|\frac{\tildu^{(k)}}{\dist_{D_k}}\bigg| \bigg)^p\\
&\lesssim \alpha^p \fabs{Q_j}.
\end{split}
\end{align}
Finally, if $Q_j$ is $k$-special, then again $\tildb_j^{(k)} = \tildu^{(k)} \varphi_j$ and we conclude as in \eqref{Eq10: Adapted CZ decomposition} above with one crucial difference: This time we do not absorb the non-Hardy term $|\tildu^{(k)}| / d_j$ into $|\tildu^{(k)}|$, but rather we use
\begin{align}
\label{Eq12: Adapted CZ decomposition}
 \dist_{D_k}(x) = \dist(x,D_k) \leq \diam(Q_j) + \dist(Q_j, D_k) \leq 2 d_j \qquad (x \in Q_j),
\end{align}
in order to absorb it into the Hardy-term $|\tildu^{(k)}|/\dist_{D_k}$.

\subsubsection*{Step 4: Non-gradient terms of the good function}

Let $1 \leq k \leq m$. In this step we prove for almost every $x \in \R^d$ the estimate
\begin{align*}
 |\tildg^{(k)}(x)|^p + \bigg|\frac{\tildg^{(k)}(x)}{\dist_{D_k}(x)}\bigg|^p \lesssim \alpha^p.
\end{align*}
On $F$ all bad functions $\tildb_j$ vanish. Hence, $\tildg = \tildu$ on this set and the required estimate follows on controlling the left-hand side above by its maximal function. So, we can concentrate on $x \in U$. Denoting by $J_u$, $J_b$, and $J_s$ the sets of those $j \in J$ such that $Q_j$ is $k$-usual, $k$-boring, and $k$-special, respectively, we obtain on $U$ that
\begin{align*}
 \tildg^{(k)}
= \tildu^{(k)} - \sum_{j \in J} \tildb_j^{(k)}
= \tildu^{(k)} - \sum_{j \in J_u} \varphi_j(\tildu^{(k)} - \tildu_{Q_j}^{(k)}) - \sum_{j \in J_b \cup J_s} \varphi_j\tildu^{(k)}
= \sum_{j \in J_u} \tildu_{Q_j}^{(k)}\varphi_j,
\end{align*}
since $\sum_{j \in J} \varphi_j = 1$ on $U$. Now, let $x \in U$ and let $J_{u,x}$ be the set of those $j \in J_u$ for which $x$ is contained in the $k$-usual cube $Q_j$. We recall from (D) that $\# J_{u,x} \leq 12^d$. Due to (d) and H\"older's inequality for sequences we find
\begin{align}
\label{Eq14: Adapted CZ decomposition}
 |\tildg^{(k)}(x)|^p
\leq \bigg(\sum_{j \in J_{u,x}} |\tildu_{Q_j}^{(k)}| \bigg)^p
\lesssim 12^{d(p-1)}\sum_{j \in J_{u,x}} \bigg(\barint_{Q_j} |\tildu^{(k)}|\bigg)^p
\lesssim \sum_{j \in J_{u,x}} \barint_{Q_j} |\tildu^{(k)}|^p.
\end{align}
Picking again elements $z_j \in 12 \sqrt{d} Q_j \cap F$, the same argument we have used several times before, for instance in \eqref{Eq5: Adapted CZ decomposition} and \eqref{Eq6: Adapted CZ decomposition}, provides control on the right-hand side by $\alpha^p$. This is the first required estimate on $U$. For the second one involving $\dist_{D_k}$, we first observe that if $y \in Q_j$ for some $j \in J_{u,x}$, then since $x \in Q_j$ as well,
\begin{align*}
 \dist_{D_k}(y) \leq \diam(Q_j) + \dist_{D_k}(x) = d_j + \dist_{D_k}(x) \leq 2 \dist_{D_k}(x) 
\end{align*}
by the defining property of $k$-usual cubes. Combining this estimate with \eqref{Eq14: Adapted CZ decomposition}, we conclude as usual,
\begin{align*}
 \bigg|\frac{\tildg^{(k)}(x)}{\dist_{D_k}(x)}\bigg|^p
&\lesssim \sum_{j \in J_{u,x}} \barint_{Q_j} \bigg|\frac{\tildu^{(k)}(y)}{\dist_{D_k}(x)}\bigg|^p \; \d y
\leq 2^p \sum_{j \in J_{u,x}} \barint_{Q_j} \bigg|\frac{\tildu^{(k)}(y)}{\dist_{D_k}(y)}\bigg|^p \; \d y
\lesssim \alpha^p.
\end{align*}

\subsubsection*{Step 5: Proofs of \eqref{vi} and \eqref{viii}}

In order to estimate $\nabla \tildg$ we have to make sure that the gradient can be pushed through the sum defining $\tildg$. We shall prove on the way the unconditional convergence stated in \eqref{viii}. To this end, let $1 \leq k \leq m$. Also let $J_0 \subseteq J$ be a finite set. Adopting the notation from Step~4 and arguing similarly to \eqref{Eq14: Adapted CZ decomposition}, we obtain
\begin{align*}
 \bigg\| \sum_{j \in J_0} \tildb_j^{(k)} \bigg \|_{\W^{1,p}}^p
&\lesssim \sum_{j \in J_u \cap J_0} \int_{Q_j} \bigg( |\varphi_j (\tildu^{(k)} - \tildu_{Q_j}^{(k)} )|^p + |\varphi_j \nabla \tildu^{(k)}|^p + |(\tildu^{(k)} - \tildu_{Q_j}^{(k)}) \nabla \varphi_j|^p \bigg) \\
&\quad + \sum_{j \in (J_b \cup J_s) \cap J_0} \int_{Q_j} \bigg( |\varphi_j \tildu^{(k)}|^p + |\varphi_j \nabla \tildu^{(k)}|^p + |\tildu^{(k)} \nabla \varphi_j|^p \bigg). 
\intertext{Investing the estimates \eqref{Eq3: Adapted CZ decomposition}, \eqref{Eq4: Adapted CZ decomposition}, and \eqref{Eq7: Adapted CZ decomposition} on $k$-usual cubes, \eqref{Eq11: Adapted CZ decomposition} on $k$-boring cubes and in addition \eqref{Eq12: Adapted CZ decomposition} on $k$-special cubes, we find}
&\lesssim \sum_{j \in J_0} \int_{Q_j} \bigg(|\tildu^{(k)}| + |\nabla \tildu^{(k)}| + \bigg| \frac{\tildu^{(k)}}{\dist_{D_k}}\bigg| \bigg)^p \\
&= \int_{\R^d} \sum_{j \in J_0} \ind_{Q_j}(x) \bigg(|\tildu^{(k)}(x)| + |\nabla \tildu^{(k)}(x)| + \bigg| \frac{\tildu^{(k)}(x)}{\dist_{D_k}(x)}\bigg| \bigg)^p \; \d x.
\end{align*}
As a consequence of (D) the series $\sum_{j \in J_0} \ind_{Q_j}$ converges pointwise to a function bounded everywhere by $12^d$. Therefore Lebesgue's theorem implies that the partial sums of $\sum_{j \in J} \tildb_j^{(k)}$ form Cauchy sequences in $\W^{1,p}(\R^d)$. The limit is independent of the order of summation again by (D). The same applies to $\sum_{j \in J'} \tildb_j^{(k)}$ for any $J' \subseteq J$ and hence we obtain \eqref{viii}.

Revisiting the calculation above for $J_0 = J$ and recalling \eqref{Eq: Def of CE}, we find
\begin{align}
\label{Eq: CZ g estimate}
\begin{split}
 \bigg\| \sum_{j \in J} \tildb_j^{(k)} \bigg \|_{\W^{1,p}(\R^d)}^p
\lesssim \int_{\R^d} \bigg(|\tildu^{(k)}(x)| + |\nabla \tildu^{(k)}(x)| + \bigg| \frac{\tildu^{(k)}(x)}{\dist_{D_k}(x)}\bigg| \bigg)^p \; \d x
\lesssim \|u\|_{\IW^{1,p}(\Omega)}^p.
\end{split}
\end{align}
We recall from Step~2 that all $\tildb_j^{(k)}$ are contained in $\W_{D_k}^{1,p}(\R^d)$. Since the latter is a closed subspace of $\W^{1,p}(\R^d)$, it also contains $\sum_{j \in J}  \tildb_j^{(k)}$ and $\tildg^{(k)} = \tildu^{(k)} - \sum_{j \in J}  \tildb_j^{(k)}$. Restricting to $\Omega$ gives $g^{(k)} \in \W_{D_k}^{1,p}(\Omega)$, that is, $g \in \IW^{1,p}(\Omega)$. Finally, the estimate in \eqref{vi} follows directly from \eqref{Eq: CZ g estimate}

\subsubsection*{Step 6: Gradient estimate of the good function}

Let $1 \leq k \leq m$. The objective of this step is to prove $|\nabla \tildg^{(k)}(x)| \lesssim \alpha$ for almost every $x \in \R^d$. Thanks to \eqref{viii}  we can compute
\begin{align*}
 \nabla \tildg^{(k)}
= \nabla \tildu^{(k)} - \sum_{j \in J_u} (\varphi_j \nabla \tildu^{(k)} + (\tildu^{(k)} - \tildu_{Q_j}^{(k)})\nabla \varphi_j) - \sum_{j \in J_b \cup J_s} (\varphi_j \nabla \tildu^{(k)} + \tildu^{(k)} \nabla \varphi_j )
\end{align*}
and all sums converge in $\L^p(\R^d)$. As in the previous step we write
\begin{align}
\label{Eq15: Adapted CZ decomposition}
\nabla \tildg^{(k)} = \nabla \tildu^{(k)} - \nabla \tildu^{(k)} \sum_{j \in J} \varphi_j - \tildu^{(k)} \sum_{j \in J} \nabla \varphi_j  + \sum_{j \in J_u} \tildu_{Q_j}^{(k)}\nabla \varphi_j.
\end{align}
Now, on $F = \R^d \setminus U$ all terms on the right-hand side vanish but the first one and we get
\begin{align*}
|\nabla \tildg^{(k)}(x)|^p = |\nabla \tildu^{(k)}(x)|^p \leq M(|\nabla \tildu|^p)(x) \leq \alpha^p \qquad (\text{a.e.\ $x \in F$}).
\end{align*}
So, we can concentrate on the similar estimate on $U$. Due to (D) the sum $\sum_{j \in J} \varphi_j$ converges in $\Lloc^1(\R^d)$ and by construction the limit is identically $1$ on $U$. Thus, $\sum_{j \in J} \nabla \varphi_j = 0 $ on $U$ in the sense of distributions and \eqref{Eq15: Adapted CZ decomposition} collapses to
\begin{align*}
 \nabla \tildg^{(k)}(x) =  \sum_{j \in J_u} \tildu_{Q_j}^{(k)} \nabla \varphi_j(x) \qquad (x \in U).
\end{align*}
We will not estimate this sum directly. Instead, we define
\begin{align*}
 h_u(x):= \sum_{j \in J_u} \tildu_{Q_j}^{(k)} \nabla \varphi_j(x), \qquad h_{s,b}(x):= \sum_{j \in J_s \cup J_b} \tildu_{Q_j}^{(k)} \nabla \varphi_j(x) \qquad (x \in U)
\end{align*}
and aim at proving $\fabs{h_{s,b}(x)}^p \lesssim \alpha^p$ and $\fabs{h_u(x) + h_{s,b}(x)}^p \lesssim \alpha^p$ for almost every $x \in U$. This of course implies the same bound for $h_u = \nabla \tildg^{(k)}$ and the proof will be complete. 

As for $h_{s,b}(x)$, we recall from \eqref{Eq12: Adapted CZ decomposition} that $\d_{D_k}(y) \leq 2d_j$ holds for all $y$ in a $k$-special cube $Q_j$ and that by definition the diameter of a $k$-boring cube is at least $1$. With $J_{b,x}$ and $J_{s,x}$ the sets of those $j \in J$ for which $x$ is contained in the $k$-boring and $k$-special cube $Q_j$, respectively, we obtain in analogy with \eqref{Eq14: Adapted CZ decomposition} the bound
\begin{align*}
 |h_{s,b}(x)|^p
\lesssim \sum_{j \in J_{b,x}}  \frac{1}{d_j^p} |\tildu_{Q_j}^{(k)}|^p + \sum_{j \in J_{s,x}} \frac{1}{d_j^p} |\tildu_{Q_j}^{(k)}|^p
\leq \sum_{j \in J_{b,x}} \barint_{Q_j} |\tildu^{(k)}|^p + \sum_{j \in J_{s,x}} \barint_{Q_j} \bigg|\frac{\tildu^{(k)}(y)}{\dist_{D_k}(y)}\bigg|^p \; \d y.
\end{align*} 
The usual maximal operator argument together with (D) provides control by $\alpha^p$. 

Preliminary to the estimate of $h_u(x) + h_{s,b}(x)$ fix an index $j_0 \in J$ such that $x \in Q_{j_0}$. For any cube $Q_j$ that contains $x$ as well, we obtain from (E) that
\begin{align}
\label{Eq-1: Adapted CZ decomposition}
 \frac{5}{6} d_j \leq \dist(Q_j, F) \leq \dist(x,F) \leq \dist(Q_{j_0}, F) + d_{j_0} \leq 5 d_{j_0}.
\end{align}
The same estimate is true with the roles of $j$ and $j_0$ interchanged. So, with $Q_{j_0}^* := 14 \sqrt{d} Q_{j_0}$ every such cube satisfies $Q_j \subseteq Q_{j_0}^*$. Again denote by $J_x$ the set of all $j \in J$ such that $Q_j$ contains $x$. Due to $\sum_{j \in J} \nabla \varphi_j = 0 $ almost everywhere on $U$ we find
\begin{align*}
 h_u(x) + h_{s,b}(x) = \sum_{j \in J_x} \tildu_{Q_j}^{(k)}\nabla \varphi_j(x) = \sum_{j \in J_x} (\tildu_{Q_j}^{(k)} - \tildu_{Q_{j_0}^*}^{(k)}) \nabla \varphi_j(x)
\end{align*}
and thus by (d) and H\"older's inequality for sequences
\begin{align*}
 |h_u(x) + h_{s,b}(x)|^p \lesssim  12^{d(p-1)} \sum_{j \in J_x}  \frac{1}{d_j^p} |\tildu_{Q_j}^{(k)} - \tildu_{Q_{j_0}^*}^{(k)}|^p.
\end{align*}
Now, for $j \in J_x$ we have
\begin{align*}
 |\tildu_{Q_j}^{(k)} - \tildu_{Q_{j_0}^*}^{(k)}|^p 
&= \bigg| \barint_{Q_j} \tildu^{(k)}(y)  -\tildu_{Q_{j_0}^*}^{(k)} \; \d y \bigg|^p
\lesssim \barint_{Q_{j_0}^*} |\tildu^{(k)}(y) - \tildu^{(k)}_{Q_{j_0}^*}|^p \; \d y
\end{align*}
since $Q_j \subseteq Q_{j_0}^*$ and $d_{j_0} \leq 6 d_{j}$ by \eqref{Eq-1: Adapted CZ decomposition} with the roles of $j$ and $j_0$ interchanged. By means of Poincar\'{e}'s inequality \eqref{Eq4: Adapted CZ decomposition} on the cube $Q_{j_0}^*$, we finally find
\begin{align*}
|\tildu_{Q_j}^{(k)} - \tildu_{Q_{j_0}^*}^{(k)}|^p 
\lesssim \diam(Q_{j_0}^*)^p \barint_{Q_{j_0}^*} |\nabla \tildu^{(k)}|^p \lesssim d_j^p  \barint_{Q_{j_0}^*} |\nabla \tildu^{(k)}|^p,
\end{align*}
leading to
\begin{align*}
 |h_u(x) + h_{s,b}(x)|^p 
\leq  \sum_{j \in J_x}  \barint_{Q_{j_0}^*} |\nabla \tildu^{(k)}|^p.
\end{align*}
As from now, the estimate by $\alpha^p$ can be completed in the usual manner.

\subsubsection*{Step 7: Proof of \eqref{iii}}

After all it remains to check $g^{(k)} \in \W_{D_k}^{1, \infty}(\Omega)$ for $1 \leq k \leq m$ with appropriate bound. The statement of Steps~4 and 6 is
\begin{align}
\label{Eq: Good function gradient}
 \|\tildg^{(k)}\|_{\L^\infty(\R^d)} + \|\nabla \tildg^{(k)}\|_{\L^\infty(\R^d)} +  \bigg \|\frac{\tildg^{(k)}}{\dist_{D_k}} \bigg\|_{\L^\infty(\R^d)} \lesssim \alpha.
\end{align}
So, $\tildg^{(k)}: \R^d \to \IC$ is essentially bounded with essentially bounded distributional gradient. Thus, it has a Lipschitz continuous representative $\widetilde{\mathfrak{g}}^{(k)}$ with Lipschitz norm bounded by generic multiple of $\alpha$, see for example \cite[Sec.~5.8.2]{Evans} for this classical result. Finally, $\widetilde{\mathfrak{g}}^{(k)}$ vanishes everywhere on $D_k$ since every $x \in D_k$ can be approximated by a sequence along which $\widetilde{\mathfrak{g}}^{(k)}/\dist_{D_k}$ remains uniformly bounded.
\end{proof}

\begin{remark}
\label{Rem: CZ}
Assumption~\ref{D} implies that all Dirichlet parts $D_k$ have vanishing Lebesgue measure, which in turn entails $\IW^{1,\infty}(\Omega) \subseteq \IW^{1,p}(\Omega)$ for every $p \in [1,\infty)$, see \cite[Lem.~3.1]{ABHR}. Hence, the good function $g$ belongs to all Sobolev spaces $\IW^{1,p}(\Omega)$.
\end{remark}

%%%%%%%%%%%%%%%%%%%%%%%%%%%%%%%%%%%%%%%%%%%%%%%%%%%%%%%%%%%%%%%%%%%%%%%%%%%%%%%%%%%%%%%%%%%%%%%%%%%%%%%%%%%%%%%%%%%%%%%%%%%%%%%%%%%%%%%%%%%%%%%%%%%%%%%%%%%%%%%%%%%%
\section{Proof of Theorem~\ref{Thm: Square root solution Lp}}
\label{Sec: Reverse estimates for the Riesz transform}

\noindent Throughout this section we assume \ref{D}, \ref{N}, and \ref{Omega}. Necessity is the easy part so let us begin with that.

\subsection{Proof of \eqref{ii Square root solution Lp}}

We borrow an idea from \cite[p.~26]{Riesz-Transforms}.
First of all, according to Theorem~\ref{Thm: Range for Lp realization} we have $(1,2) \subseteq \cJ(L)$ if $d=2$ and $[2_*,2) \subseteq \cJ(L)$ if $d \geq 3$. Henceforth, we only need to treat the case where $d \geq 3$ and $L^{1/2}$ extends to an isomorphism $\IW^{1,p} \to \L^p(\Omega)^m$ for some $p \in [1,2_*)$. We need to prove $(p,2) \subseteq \cJ(L)$.

We claim first that $L^{-1/2}$ extends to a bounded operator $\L^q(\Omega)^m \to \L^{q^*}(\Omega)$ for every $q \in [p,2]$. Indeed, by Riesz-Thorin interpolation it suffices to check the endpoints and -- keeping in mind the Sobolev embedding $\IW^{1,q}(\Omega) \subseteq \L^{q^*}(\Omega)^m$ -- we obtain the case $q=p$ from the assumption and the case $p=2$ from Theorem~\ref{Thm: Kato}.

This being said, we put $p_0 := p$, $p_j = p_{j-1}^*$ for $j = 1,\ldots$ and stop at the first number $j$ with $p_j \in (2_*,2]$.
By construction this happens for some $j\geq 1$ and the above applies to $q=p_0,\ldots,p_j$. We find for every $t>0$ the chain of bounded mappings
\begin{align*}
 \L^{p_0}(\Omega)^m 
 \xrightarrow{L^{-1/2}}\L^{p_1}(\Omega)^m
 \cdots \xrightarrow{L^{-1/2}} \cdots \L^{p_j}(\Omega)^m 
 \xrightarrow{\e^{-t/2 L}} \L^2(\Omega)^m
 \xrightarrow{L^{j/2} \e^{-t/2 L}} \L^2(\Omega)^m,
\end{align*}
where the second to last arrow with operator norm controlled by $t^{d/4 -d/(2p_j)}$ is due to Proposition~\ref{Prop: Hypercontractivity}.\eqref{i Hypercontractivity} on noting that $p_j$ is an interior point of $\cJ(L)$. The final arrow with operator norm controlled by $t^{-j/2}$ follows from the bounded $\H^\infty$-calculus on $\L^2(\Omega)^m$. But the chain above induces $\e^{-tL}$ and since $d/p = d/p_j + j$ holds by construction, we have shown its $\L^p \to \L^2$ boundedness. Proposition~\ref{Prop: Hypercontractivity} yields $(p,2) \subseteq \cJ(L)$. \hfill $\square$ \\

\vspace{-12pt}
\subsection{Proof of \eqref{i Square root solution Lp}}
We turn to \eqref{i Square root solution Lp} and claim that it suffices to prove the following key proposition. \note{In dimension $d \geq 3$, the range of admissible exponents in this proposition can exceed $\cJ(L)$ by up to one Sobolev conjugate.}

\note{
\begin{proposition}
\label{Prop: Weak type estimate square root}
Let $p_0 \in \cJ(L) \cap (1,2)$. In dimension $d \geq 3$, assume that $p \in (1,2)$ is such that $p^* \in (p_0,2]$ and in dimension $d=2$, let $p \in (p_0, 2]$. The weak-type bound
\begin{align*}
 \Big| \Big\{ x \in \Omega: |L^{1/2}u (x)| > \alpha \Big\} \Big| \lesssim \frac{1}{\alpha^p} \|u\|_{\IW^{1,p}}^p \qquad (\alpha>0, \, u \in \IW^{1,2}(\Omega))
\end{align*}
holds with an implicit constant depending on $\Jnorm{p_0}$ and geometry.
\end{proposition}}

%\begin{proposition}
%\label{Prop: Weak type estimate square root}
%For $p \in \cJ(L) \cap (1,2_*)$ the weak-type bound
%\begin{align*}
% \Big| \Big\{ x \in \Omega: |L^{1/2}u (x)| > \alpha \Big\} \Big| \lesssim \frac{1}{\alpha^p} \|u\|_{\IW^{1,p}}^p \qquad (\alpha>0, \, u \in \IW^{1,2}(\Omega)),
%\end{align*}
%with an implicit constant depending on $\Jnorm{p}$ and geometry.
%\end{proposition}

Let us first see how the proposition leads to the proof of the first part of Theorem~\ref{Thm: Square root solution Lp}. 

\note{In dimension $d=2$, let $p_0 \in \cJ(L) \cap (1,2)$ be arbitrary. In dimension $d \geq 3$, thanks to Theorem~\ref{Thm: Range for Lp realization}, we can guarantee that $\cJ(L) \cap (1,2_*)$ is non-empty. Let it contain $p_0$.}
	
Due to Proposition~\ref{Prop: Weak type estimate square root} and Theorem~\ref{Thm: Kato} we have at hand (extensions to) bounded operators
\begin{align*}
 L^{1/2}: \IW^{1,p_0}(\Omega) \to \L^{p_0,\infty}(\Omega)^m, \qquad L^{1/2}: \IW^{1,2}(\Omega) \to \L^2(\Omega)^m,
\end{align*}
where $\L^{p_0,\infty}$ denotes the usual weak $\L^{p_0}$-space. We refer to \cite[Sec.~1.18.6]{Triebel} for background. Now, let $p \in (p_0,2)$ and $1/p = (1-\theta)/p_0 + \theta/2$. By real interpolation this entails boundedness of
\begin{align*}
 L^{1/2}: (\IW^{1,p_0}(\Omega), \IW^{1,2}(\Omega))_{\theta, p} \to (\L^{p_0,\infty}(\Omega)^m, \L^2(\Omega)^m)_{\theta, p},
\end{align*}
see e.g.\ \cite[Sec.~1.3.3]{Triebel} for background on these notions. Up to equivalent norms the left-hand space is $\IW^{1,p}(\Omega)$, see \cite[Thm.~8.1]{ABHR}. The right-hand space is $\L^p(\Omega)^m$ up to equivalent norms, see \cite[Thm.~2, Sec.~1.18.6]{Triebel}. Thus,
\begin{align}
\label{Eq: Reverse estimate}
 \|L^{1/2} u\|_{\L^p(\Omega)^m} \lesssim \|u\|_{\IW^{1,p}(\Omega)} \qquad (u \in \IW^{1,2}(\Omega)).
\end{align}
(In \cite{ABHR} only the case $m=1$ was considered, but real interpolation interchanges with finite products of spaces by an abstract principle, see \cite[Sec.~1.17.1]{Triebel} or \cite[Cor.~1.3.8]{EigeneDiss}.) Since $p$ is an interior point of $\cJ(L)$, Corollary~\ref{Cor: Riesz estimate} provides the estimate reverse to \eqref{Eq: Reverse estimate}. This means that $L^{1/2}$ extends to a one-to-one operator $\IW^{1,p}(\Omega) \to \L^p(\Omega)^m$ with closed range. Furthermore, from Theorem~\ref{Thm: Kato} we know $\Rg(L^{1/2}) = \L^2(\Omega)^m$. Hence, this extension has dense range and therefore is an isomorphism. We have picked up implicit constants depending on $\Jnorm{p_0}$, geometry, and on $p$. The latter comes in particular from real interpolation. 

This completes the proof of Theorem~\ref{Thm: Square root solution Lp} modulo the

\begin{proof}[{Proof of Proposition~\ref{Prop: Weak type estimate square root} \note{($d \geq 3$)}}]
The argument follows the lines of \cite[Prop.~10.1]{ABHR} with two essential differences: A different Calder\'{o}n--Zygmund decomposition and the presence of the technical condition \note{$p^* \in (p_0,2]$}. The raison d'\^etre for the latter is to have at our disposal
\begin{itemize}
 \item $\L^{p^*} \to \L^2$ off-diagonal estimates for $\{\e^{-tL}\}_{t>0}$ and 
 \item $\L^{p^*}$ boundedness of the $\H^\infty(\S^+_\psi)$-calculus for $L$ of some fixed angle $\psi \in (0, \pi/2)$.
\end{itemize}
Indeed, since \note{$p^* \in (p_0,2]$} the first property is a consequence of Proposition~\ref{Prop: Hypercontractivity}.\eqref{iii Hypercontractivity} and implicit constants depend only on \note{$\Jnorm{p_0}$} and geometry, see Remark~\ref{Rem: Hypercontractivity}. Similarly, the second property is due to Theorem~\ref{Thm: Hinfty in Lp} with the same dependence of implicit constants.

To get started, let $\alpha > 0$ and $u \in \IW^{1,2}(\Omega)$. In particular we have $u \in \IW^{1,p}(\Omega)$ and within this space we decompose 
\begin{align*}
 u = g +b, \quad b = \sum_{j \in J} b_j
\end{align*}
according to Lemma~\ref{Lem: Adapted CZ decomposition}. Since $g$ is contained in $\IW^{1,2}(\Omega)$, see Remark~\ref{Rem: CZ}, the above is also a decomposition in that space. In the further course of the proof \eqref{i} - \eqref{viii} will refer to the respective features of the Calder\'{o}n--Zygmund decomposition. We then split
\begin{align}
\label{Eq: Weak type g b splitting}
\begin{split}
 \Big| \Big\{ x \in \Omega: |L^{1/2}u (x)| > \alpha \Big\} \Big|
&\leq \Big| \Big\{ x \in \Omega: |L^{1/2}g (x)| > \frac{\alpha}{2} \Big\} \Big| \\
&\quad + \Big| \Big\{ x \in \Omega: |L^{1/2}b (x)| > \frac{\alpha}{2} \Big\} \Big|.
\end{split}
\end{align}

\subsubsection*{Step 1: Estimate of the good part}

The good function produces an easy-to-handle term. By H\"older's inequality, \eqref{iii}, and \eqref{vi} (or \eqref{Eq: Good function gradient} as it were), we have
\begin{align*}
 \|g\|_{\IW^{1,2}}^2 \lesssim \alpha^{2-p} \|g\|_{\IW^{1,p}}^p \lesssim \alpha^{2-p} \|u\|_{\IW^{1,p}}^p
\end{align*}
and the desired bound follows from Tchebychev's inequality and Theorem~\ref{Thm: Kato}:
\begin{align*}
 \Big| \Big\{ x \in \Omega: |L^{1/2}g (x)| > \frac{\alpha}{2} \Big\} \Big|
\leq \frac{4}{\alpha^2} \|L^{1/2} g\|_2^2 \lesssim \frac{1}{\alpha^2} \|g\|_{\IW^{1,2}}^2 \lesssim \frac{1}{\alpha^p}\|u\|_{\IW^{1,p}}^p. 
\end{align*}

\subsubsection*{Step 2: Further decomposition of the bad part}

We turn to the second term on the right-hand side of \eqref{Eq: Weak type g b splitting}. 
%\note{We have $b = \sum_{j \in J} b_j$. Since the sum converges in $\IW^{1,2}(\Omega)$ and $L^{1/2}: \IW^{1,2}(\Omega) \to \L^2(\Omega)^m$ is bounded by Theorem~\ref{Thm: Kato}, it suffices to obtain the weak-type bound for all finite partial sums of $\sum_{j \in J} b_j$ with an implicit constant that is independent of the particular truncation in $j$. By a slight abuse of notation we call such a sum still $b$ and continue to write $J$ for the {\emph{finite}} index set.} 
We start out with the formula,
\begin{align*}
 L^{1/2} b = \frac{2}{\sqrt{\pi}} \int_0^\infty L \e^{-t^2 L}b \; \d t,
\end{align*}
which is a direct consequence of \eqref{Eq: Resolution of the identity}. Hence,
\begin{align*}
\Big| \Big\{ x \in \Omega: |L^{1/2}b (x)| > \frac{\alpha}{2} \Big\} \Big|
&= \bigg| \bigg\{ x \in \Omega: \liminf_{n \to \infty} \bigg|\frac{2}{\sqrt{\pi}} \int_{2^{-n}}^\infty L \e^{-t^2 L}b(x) \; \d t\bigg| > \frac{\alpha}{2} \bigg\} \bigg| \\
&\leq \liminf_{n \to \infty} \bigg| \bigg\{ x \in \Omega: \bigg|\frac{2}{\sqrt{\pi}} \int_{2^{-n}}^\infty L \e^{-t^2 L}b(x) \; \d t\bigg| > \frac{\alpha}{2} \bigg\} \bigg|.
\end{align*}
Denote the sidelength of $Q_j$ by $\ell_j$ and write $r_j= 2^\ell$ for the unique value $\ell \in \IZ$ that satisfies $2^\ell \leq \ell_j < 2^{\ell+1}$. Then, we split the integral for every $n \in \IN$ as
\begin{align}
\label{Eq: Weak type m splitting}
\begin{split}
\bigg| \bigg\{ &x \in \Omega: \bigg|\frac{2}{\sqrt{\pi}} \int_{2^{-n}}^\infty L \e^{-t^2 L}b(x) \; \d t\bigg| > \frac{\alpha}{2} \bigg\} \bigg| \\
&\leq  \bigg| \bigg\{x \in \Omega: \bigg| \sum_{j \in J} \int_{2^{-n}}^{r_j \vee 2^{-n}} L \e^{-t^2 L}b_j(x) \; \d t\bigg| > \frac{\sqrt{\pi} \alpha}{8} \bigg\} \bigg|\\
&+ \bigg| \bigg\{x \in \Omega: \bigg| \sum_{j \in J} \int_{r_j \vee 2^{-n}}^\infty L \e^{-t^2 L}b_j(x) \; \d t\bigg| > \frac{\sqrt{\pi} \alpha}{8} \bigg\} \bigg|,
\end{split}
\end{align}
where sum and integral could be interchanged since on the one hand the sum over the $b_j$ converges in $\L^{p^*}$ due to \eqref{viii} and Sobolev embeddings (making use of the extension operators as usual) and on the other hand $L \e^{-t^2 L}$ is bounded on $\L^{p^*}$ with norm under control by $t^{-2}$ by the bounded $\H^\infty$-calculus. 
We also note that the $b_j$ are in $\L^2$, see \eqref{vii}.

\subsubsection*{Step 3: Estimate of the first term on the right of \eqref{Eq: Weak type m splitting}}

Of course we may assume $r_j > 2^{-n}$. From Tchebychev's inequality we can infer
\begin{align}
\label{Eq: Weak type start of step 3}
\begin{split}
 \bigg| \bigg\{ &x \in \Omega :  \bigg| \sum_{j \in J} \int_{2^{-n}}^{r_j} L \e^{-t^2 L}b_j(x) \; \d t\bigg| > \frac{\sqrt{\pi} \alpha}{8} \bigg\} \bigg| \\
&\leq \bigg|\bigcup_{j \in J} 4 Q_j\bigg| + \frac{64}{\pi \alpha^2} \bigg\| \ind_{\Omega \setminus \cup_{j \in J} 4 Q_j} \sum_{j \in J} \int_{2^{-n}}^{r_j} L \e^{-t^2 L}b_j \; \d t \bigg \|_2^2.
\end{split}
\end{align}
The union of the cubes $4Q_j$ does not cause any problems since its measure can be controlled by $ \|u\|_{\IW^{1,p}}^p/\alpha^p$, see \eqref{v}. We start to estimate the leftover $\L^2$ norm by testing against $v \in \L^2(\Omega)^m$ with $\|v\|_2 = 1$:
\begin{align}
\label{Eq: Weak type duality}
\begin{split}
 \bigg|\int_\Omega & \cl{v} \ind_{\Omega \setminus \cup_{j \in J} 4 Q_j} \bigg(\sum_{j \in J} \int_{2^{-n}}^{r_j} L \e^{-t^2 L}b_j \; \d t\bigg) \; \d x \bigg| \\
&\leq \sum_{j \in J} \int_{\Omega \setminus 4Q_j} |v| \bigg| \int_{2^{-n}}^{r_j} L \e^{-t^2 L}b_j \; \d t\bigg| \; \d x.
\end{split}
\end{align}
For fixed $j$ we split $\Omega \setminus 4Q_j$ into annuli $C_k(Q_j) \cap \Omega$, where $C_k(Q_j) = 2^{k+1}Q_j \setminus 2^k Q_j$, and apply the Cauchy-Schwarz inequality to give
\begin{align}
\label{Eq: Weak type duality one cube}
\begin{split}
 \int_{\Omega \setminus 4Q_j} |v| \bigg| &\int_{2^{-n}}^{r_j} L \e^{-t^2 L}b_j \; \d t\bigg| \; \d x \\
 &\leq \sum_{k = 2}^\infty \|v\|_{\L^2(C_k(Q_j) \cap \Omega)} \bigg\|\int_{2^{-n}}^{r_j} L \e^{-t^2 L}b_j \; \d t\bigg\|_{\L^2(C_k(Q_j) \cap \Omega)}.
\end{split}
\end{align}
Identifying $v$ with its extension by zero to $\R^d$, we obtain for every $y \in Q_j$ that
\begin{align*}
 \|v\|_{\L^2(C_k(Q_j) \cap \Omega)}^2 
\lesssim 2^{dk} \ell_j^d  M(|v|^2)(y).
\end{align*}
To control the other norm on the right-hand side of \eqref{Eq: Weak type duality one cube} we recall that as far as off-diagonal estimates are concerned, we have $\L^{p^*} \to \L^2$ for $\{\e^{-t L}\}_{t>0}$.  Since we also have $\L^2 \to \L^2$ for $\{tL \e^{-t L}\}_{t>0}$ from Proposition~\ref{Prop: OD estimates for semigroup on L2}, we obtain $\L^{p^*} \to \L^2$ for $\{tL \e^{-t L}\}_{t>0}$ by composition as in the proof of Proposition~\ref{Prop: Hypercontractivity}.\eqref{iv Hypercontractivity}. All implied results contain a statement about implicit constants and so we may note without any pain that
\begin{align*}
 \|t^2 L \e^{-t^2 L}b_j \|_{\L^2(C_k(Q_j) \cap \Omega)} \leq C t^{d/2 - d/p^*} \e^{-c 4^{k-1} r_j^2 / t^2}\|b_j\|_{\L^{p^*}(\Omega)},
\end{align*}
where $C,c \in (0,\infty)$ depend on \note{$\Jnorm{p_0}$} and geometry, and we have used that $b_j$ is supported in $Q_j$, see \eqref{ii}. From Sobolev embeddings and \eqref{iv} we can infer $\|b_j\|_{p^*} \lesssim \alpha \ell_j^{d/p}$ so that altogether
\begin{align*}
  \bigg\|\int_{2^{-n}}^{r_j} L \e^{-t^2 L}b_j \; \d t\bigg\|_{\L^2(C_k(Q_j) \cap \Omega)}
&\leq \int_{2^{-n}}^{r_j} \|L \e^{-t^2 L}b_j \|_{\L^2(C_k(Q_j)\cap \Omega)} \; \d t\\
&\lesssim  \alpha \ell_j^{d/p} \int_0^{r_j} t^{d/2 - d/p^* - 2} \e^{-c4^{k-1} r_j^2/t^2} \; \d t\\
& = \frac{1}{2} \alpha \ell_j^{\frac{d}{p}} (4^k r_j^2 )^{\frac{d}{4} - \frac{d}{2p^*} - \frac{1}{2}} \int_{4^k}^\infty s^{-\frac{d}{4} + \frac{d}{2p^*} - \frac{1}{2}} \e^{-\frac{cs}{4}} \; \d s,
\intertext{the last step being due to a change of variables $s= 4^{k} r_j^2 / t^2$. Abbreviating $\gamma = -d/2 + d/{p^*} + 1 > 0$ and using $2r_j \geq \ell_j$, we obtain}
&\leq \frac{1}{2} \alpha \ell_j^{d/2} 2^{-(k-1)\gamma} \int_{4^k}^\infty s^{\gamma/2 -1} \e^{-cs/4} \; \d s \\
&\leq \frac{1}{2} \alpha \ell_j^{d/2} 2^{-(k-1)\gamma} \e^{-c4^{k-3}} \int_0^\infty s^{\gamma/2 -1} \e^{-cs/8} \; \d s.
\end{align*}
The integral in $s$ is finite. Coming back to \eqref{Eq: Weak type duality one cube}, so far we have for every $y \in Q_j$ that
\begin{align*}
\int_{\Omega \setminus 4Q_j} |v| \bigg| \int_{2^{-n}}^{r_j} L \e^{-t^2 L}b_j \; \d t\bigg| \; \d x
\lesssim \alpha \ell_j^d \Big(M(|v|^2)(y)\Big)^{1/2} \sum_{k=2}^\infty 2^{(d/2-\gamma)k} \e^{-c4^{k-3}},
\end{align*}
where the sum over $k$ is convergent. So, we can average with respect to $y$ to give
\begin{align*}
\int_{\Omega \setminus 4Q_j} |v| \bigg| \int_{2^{-n}}^{r_j} L \e^{-t^2 L}b_j \; \d t\bigg| \; \d x
\lesssim \alpha \int_{Q_j} \Big(M(|v|^2)(y)\Big)^{1/2} \; \d y.
\end{align*}
Now, we re-insert this estimate on the right-hand of our starting point \eqref{Eq: Weak type duality}. Invoking the finite overlap property \eqref{ii} of the $Q_j$, we obtain
\begin{align*}
 \bigg|\int_\Omega v \ind_{\cup_{j \in J} 4 Q_j} \bigg(\sum_{j \in J} \int_{2^{-n}}^{r_j} L \e^{-t^2 L}b_j \; \d t\bigg) \; \d x \bigg|
&\lesssim \alpha \int_{\bigcup_{j \in J} Q_j} \Big(M(|v|^2)(y)\Big)^{1/2} \; \d y.
\end{align*}
From Kolmogorov's inequality \cite[Lem.~5.16]{Duo}, \eqref{v}, and the normalization of $v$ we can infer 
\begin{align*}
\int_{\bigcup_{j \in J} Q_j} \Big(M(|v|^2)(y)\Big)^{1/2} \; \d y 
\lesssim \bigg|\bigcup_{j \in J} Q_j\bigg|^{1/2} \||v|^2\|_1^{1/2}
\leq \bigg( \sum_{j \in J} |Q_j| \bigg)^{1/2} \|v\|_2
\lesssim \frac{1}{\alpha^{p/2}} \|u\|_{\IW^{1,p}}^{p/2}.
\end{align*}
Going all the way back to the start, we have also bound the $\L^2$ norm occurring in \eqref{Eq: Weak type start of step 3} by $\alpha^{2-p} \|u\|_{\IW^{1,p}}^p$ and therefore completed Step~3.

\subsubsection*{Step 4: Estimate of the second term on the right of \eqref{Eq: Weak type m splitting}}

In preparation of the estimate, we define 
\begin{align*}
 f(z) = \int_1^\infty z \e^{-t^2 z} \; \d t \qquad (\Re z > 0)
\end{align*}
and note $f \in \H_0^\infty(\S^+_\psi)$ for any angle $\psi \in (0, \frac{\pi}{2})$, see also \cite[p.~198]{ABHR}. We have bounded operators $f(r^2L)$ on $\L^2$ for $r>0$, which extend boundedly to $\L^{p^*}$ as we have noted right at the start. By the very definition of the functional calculus and Fubini's theorem we can write more conveniently
\begin{align*}
 f(r^2 L) = \int_1^\infty r^2 L \e^{-t^2 r^2 L} \; \d t = r \int_r^\infty L \e^{-t^2 L} \; \d t.
\end{align*}
Introducing $J_k:= \{j \in J: r_j \vee 2^{-n} = 2^k\}$ for $k \in \IZ$, we therefore have
\begin{align*}
 \sum_{j \in J} \int_{r_j \vee 2^{-n}}^\infty L \e^{-t^2 L}b_j \; \d t
= \sum_{k \in \IZ} \sum_{j \in J_k} \int_{2^k}^\infty L \e^{-t^2 L}b_j \; \d t
= \sum_{k \in \IZ} \sum_{j \in J_k} \frac{1}{2^k} f(4^k L) b_j.
\end{align*}
%TODO Justification of changing order of summation. Probably it is best to truncate J right at the start. Second possibility: Change the order of summation for J in (8.3) to the desired one and mention why this is ok.
We start the actual estimate with Tchebychev's inequality 
\begin{align*}
 \bigg| \bigg\{&x \in \Omega: \bigg| \sum_{j \in J} \int_{r_j \vee 2^{-n}}^\infty L \e^{-t^2 L}b_j(x) \; \d t\bigg| > \frac{\sqrt{\pi} \alpha}{8} \bigg\} \bigg| \\
&\leq \frac{8^{p^*}}{\pi^{p^*/2} \alpha^{p^*}} \bigg \|\sum_{k \in \IZ} \sum_{j \in J_k} \frac{1}{2^k} f(4^k L) b_j \bigg\|_{p^*}^{p^*}.
\intertext{Since $\sum_{j \in J_k} b_j$ converges in $\IW^{1,p}(\Omega)$ by \eqref{viii} and hence in $\L^{p^*}$, we may write}
&= \frac{8^{p^*}}{\pi^{p^*/2} \alpha^{p^*}} \bigg \|\sum_{k \in \IZ} f(4^k L) \bigg(\sum_{j \in J_k} 2^{-k} b_j \bigg) \bigg\|_{p^*}^{p^*}
\intertext{and obtain from Lemma~\ref{Lem: Discrete QE} below the bound}
&\lesssim \frac{1}{\alpha^{p^*}} \bigg\|\bigg(\sum_{k \in \IZ} \Big|\sum_{j \in J_k} 2^{-k} b_j\Big|^2 \bigg)^{1/2} \bigg\|_{p^*}^{p^*}
= \frac{1}{\alpha^{p^*}} \int_\Omega \bigg(\sum_{k \in \IZ} \Big|\sum_{j \in J_k} 2^{-k} b_j(x) \Big|^2 \bigg)^{p^*/2} \; \d x.
\intertext{Here, implicit constants depend only on $p^*$ and a bound for the $\H^\infty$-calculus for $L$ on $\L^{p^*}$ (of some angle $\psi \in (0, \pi/2)$). We have seen that the latter can be given in terms of \note{$\Jnorm{p_0}$} and geometry. Due to $p^*/2 \leq 1$ we can continue by}
& \leq \frac{1}{\alpha^{p^*}} \int_\Omega \sum_{k \in \IZ} \Big(\sum_{j \in J_k} |2^{-k} b_j(x)| \Big)^{p^*} \; \d x.
\intertext{As a consequence of \eqref{ii} the sum in $j$ contains for fixed $x$ at most $12^d$ non-zero terms. Thus, we can replace the inner $\ell^1$-norm by an $\ell^{p^*}$-norm at the expense of a constant depending on $p$ and $d$ in order to give}
&\lesssim \frac{1}{\alpha^{p^*}} \int_\Omega \sum_{k \in \IZ} \sum_{j \in J_k} 2^{-k p^*} |b_j(x)|^{p^*} \; \d x
\leq \frac{2^{p^*}}{\alpha^{p^*}} \sum_{j \in J} \ell_j^{-p^*} \int_\Omega  |b_j(x)|^{p^*} \; \d x,
\intertext{where we have used $\frac{1}{2} \ell_j \leq r_j \leq 2^k$ for $j \in J_k$. Finally, by Sobolev embeddings, \eqref{iv}, and \eqref{v} we deduce}
&\lesssim \frac{1}{\alpha^{p^*}} \sum_{j \in J} \ell_j^{-p^*} \|b_j\|_{\IW^{1,p}}^{p^*}
\lesssim \sum_{j \in J} \ell_j^{-p^*} |Q_j|^{p^*/p} = \sum_{j \in J} |Q_j|
\lesssim \frac{1}{\alpha^p} \|u\|_{\IW^{1,p}}^p.
\end{align*}
This completes the proof of the proposition modulo Lemma~\ref{Lem: Discrete QE}, which we prove below.
\end{proof}

\note{
\begin{proof}[{Proof of Proposition~\ref{Prop: Weak type estimate square root} \note{($d = 2$)}}]
In this case the range $\cJ(L) \cap (1,2) = (1,2)$ is already the largest possible (Theorem~\ref{Thm: Range for Lp realization}) and hence there is no Sobolev conjugate to gain. Consequently, we repeat the proof above word by word, using the exponent $p$ instead of $p^*$. This means that we work with the functional calculus in $\L^p$ and instead of $\|b_j\|_{p^*} \lesssim \alpha \ell_j^{d/p}$ we rely on the bound
\begin{align*}
	\|b_j\|_{p} \leq \|\tildb_j\|_{p} \lesssim \ell_j \|\nabla \tildb_j\|_{p} \lesssim \alpha \ell_j^{1+d/p}.
\end{align*}
We have used the construction of $b_j$, then $\supp(\tildb_j) \subseteq Q_j^\circ$ and Poincaré's inequality, and finally \eqref{Eq6: Adapted CZ decomposition} and \eqref{Eq10: Adapted CZ decomposition}.
\end{proof}
}

Let us remark that for $\Omega = \R^d$ the following lemma was obtained in \cite[Lem.~4.14]{Riesz-Transforms} for $p<2$ by duality and a weak type criterion for $p>2$, which we do not have at our disposal. Later, in \cite{ABHR} it was proved for $p \in (1,\infty)$ through profound $\mathcal{R}$-boundedness techniques for the functional calculi that make it hard to track implicit constants. Here, we present a more elementary approach.

\begin{lemma}
\label{Lem: Discrete QE}
Let $p \in (1,\infty)$, $\Xi \subseteq \R^d$ be a measurable set, and $n \in \IN$. Let $T$ be a one-to-one sectorial operator on $\L^2(\Xi)^n$ such that for some $\psi \in (0,\pi)$ it holds
\begin{align*}
 \|f(T)u\|_p \leq C_\psi \|f\|_\infty \|u\|_p \qquad (f \in \H^\infty(\S^+_\psi), \, u \in \L^2(\Xi)^n \cap \L^p(\Xi)^n).
\end{align*}
Let $f \in \H_0^\infty(\S^+_\psi)$. Then there is a constant $C \in (0,\infty)$ that depends on $f, \psi$, such that 
\begin{align*}
 \bigg\|\sum_{k \in \IZ} f(4^k T) u_k \bigg\|_p \leq C C_\psi \bigg\| \bigg(\sum_{k \in \IZ} |u_k|^2 \bigg)^{1/2} \bigg\|_p
\end{align*}
for every sequence $\{u_k \}_{k \in \IZ} \subseteq \L^2(\Xi)^n \cap \L^p(\Xi)^n$ for which the right-hand side is finite.
\end{lemma} 

\begin{proof}
The adjoint $T^*$ has the same properties as $T$ with $p$ replaced by its H\"older conjugate $q$. This follows from the identity $g(T)^* = g^*(T^*)$, where $g^*(z) = \cl{g(\cl{z})}$ and $g \in \H^\infty(\S^+_\psi)$. Arguing by duality, it suffices to show 
\begin{align}
\label{Eq1: Discrete QE}
 \bigg\|\bigg(\sum_{k \in \IZ} |f^*(4^k T^*) v|^2 \bigg)^{\frac{1}{2}} \bigg\|_q \leq C C_\psi \|v\|_q \qquad (v \in \L^2(\Xi)^n \cap \L^q(\Xi)^n).
\end{align}
Let $\{r_k\}_{k \in \IZ}$ be a sequence of symmetric independent $\{-1,1\}$-valued random variables on the unit interval and let $N \in \IN$. By orthogonality of the $r_k$ in $\L^2(0,1)$ we have
\begin{align*}
  \bigg\|\bigg(\sum_{k = -N}^N |f^*(4^k T^*) v|^2 \bigg)^{\frac{1}{2}} \bigg\|_q^q
&= \int_\Xi \bigg(\int_0^1 \Big|\sum_{k = -N}^N r_k(s) f^*(4^k T^*) v(x) \Big|^2 \; \d s \bigg)^{\frac{q}{2}} \; \d x.
\intertext{Kahane's inequality \cite[Thm.~11.1]{DJT} allows us to replace the $\L^2$ norm in $s$ by an $\L^1$ norm at the expense of an absolute constant $C \in [1,\infty)$. Then we can apply Jensen's inequality to give}
&\leq  C^q \int_\Xi \int_0^1 \Big|\sum_{k = -N}^N r_k(s) f^*(4^k T^*) v(x) \Big|^q \; \d s \; \d x \\
&= C^q \int_0^1 \bigg\|\sum_{k = -N}^N r_k(s) f^*(4^k T^*) v \bigg\|_q^q \; \d s \\
&\leq C^q \sup_{|a_k| \leq 1} \bigg\|\sum_{k = -N}^N a_k f^*(4^k T^*) v \bigg\|_q^q,
\end{align*}
where the $a_k$ are complex numbers. The bounded $\H^\infty(\S^+_\psi)$-calculus for $T^*$ yields
\begin{align*}
\bigg\|\bigg(\sum_{k = -N}^N |f^*(4^k T^*) v|^2 \bigg)^{\frac{1}{2}} \bigg\|_q
&\leq C C_\psi \|v\|_q \bigg(\sup_{z \in \S^+_\psi} \sum_{k=-\infty}^\infty |f^*(4^k z)|\bigg).
\end{align*}
In order to see that this estimate implies \eqref{Eq1: Discrete QE}, let us recall that by assumption there exist $C', s > 0$ such that $|f^*(z)| \leq C' \min\{|z|^s, |z|^{-s}\}$ for all $z \in \S^+_\psi$. Hence, given $z$, we choose $k_0 \in \IZ$ such that $4^{k_0} \leq |z| < 4^{k_0 + 1}$ and obtain
\begin{align*}
 \sum_{k=-\infty}^\infty |f^*(4^k z)|
&\leq C' \sum_{k=-\infty}^\infty \min\{|4^k z|^s, |4^k z|^{-s} \} \\
&\leq C' 4^s \sum_{k=-\infty}^\infty \min\{4^{(k+k_0)s}, 4^{-(k+k_0)s} \}
\leq \frac{2 C' 4^{2s} }{4^s-1}
\end{align*}
by a computation of the geometric series.
\end{proof}
%%%%%%%%%%%%%%%%%%%%%%%%%%%%%%%%%%%%%%%%%%%%%%%%%%%%%%%%%%%%%%%%%%%%%%%%%%%%%%%%%%%%%%%%%%%%%%%%%%%%%%%%%%%%%%%%%%%%%%%%%%%%%%%%%%%%%%%%%%%%%%%%%%%%%%%%%%%%%%%%%%%%
\section{Proof of Theorem~\ref{Thm: Extended square root solution Lp}}
\label{Sec: Extended square root solution Lp}

To begin with, let us recall from Section~\ref{Subsec: L} that $L$ is the maximal restriction to $\L^2(\Omega)^m$ of the isomorphism $\Lop: \IW^{1,2}(\Omega) \to \IW^{-1,2}(\Omega)$.

\begin{lemma}
\label{Lem: Duality formula}
For $u \in \dom(L^{1/2})$ and $v \in \L^2(\Omega)^m$ the duality formula
\begin{align*}
 \scal{L^{1/2}u}{v} = \dscal{\Lop u}{(L^*)^{-1/2}v}.
\end{align*}
\end{lemma}

\begin{proof}
It suffices to take $u$ in $\dom(L)$ since the latter is a core for $\dom(L^{1/2}) = \IW^{1,2}(\Omega)$ and $\Lop: \dom(L^{1/2}) \to \IW^{-1,2}(\Omega)$ is bounded, see Section~\ref{Subsec: square root L}. In this case $\Lop u = Lu \in \L^2(\Omega)^m$ and the claim follows from the duality formula $(L^{1/2})^* = (L^*)^{1/2}$ for the functional calculi, see again Section~\ref{Subsec: Functional calculus}.
\end{proof}

The subsequent proof was inspired by that of \cite[Thm.~6.5]{Disser-terElst-Rehberg}.

\begin{proof}[Proof of Theorem~\ref{Thm: Extended square root solution Lp}]
We pick $\eps(L) > 0$ depending on ellipticity, dimensions, and geometry, such that for $p \in (2, 2 + \eps(L))$,
\begin{enumerate}
\def\theenumi{\roman{enumi}}
 \item \label{Extended square root i}$\Lop$ restricts to an isomorphism $\IW^{1,p}(\Omega) \to \IW^{-1,p}(\Omega)$ and
 \item \label{Extended square root ii} $(L^*)^{1/2}$ extends to an isomorphism $\IW^{1,p'}(\Omega) \to \L^{p'}(\Omega)^m$.
\end{enumerate}
The former is rendered possible by Proposition~\ref{Prop: Domain is W1p}, the latter by Theorem~\ref{Thm: Square root solution Lp} and Theorem~\ref{Thm: Range for Lp realization} for $L^*$. Upper and lower bounds of these isomorphisms import at most an additional dependence on $p$.

Now, let $u \in \IW^{1,p}(\Omega)$ and $v \in \L^2(\Omega)^m$. By Lemma~\ref{Lem: Duality formula}, \eqref{Extended square root i}, and \eqref{Extended square root ii} we have
\begin{align*}
 |\scal{L^{1/2}u}{v}| \leq \|\Lop u\|_{\IW^{-1,p}} \|(L^*)^{-1/2}v\|_{\IW^{1,p'}} \lesssim \|u\|_{\IW^{1,p}} \|v\|_{p'}.
\end{align*}
Thus, $L^{1/2}u \in \L^p(\Omega)^m$ and $L^{1/2}$ restricts to a bounded operator $\IW^{1,p}(\Omega) \to \L^p(\Omega)^m$. As a restriction of an invertible operator it is already one-to-one. To show that it is also onto, let $f \in \L^p(\Omega)$. There exists $u \in \IW^{1,2}(\Omega)$ such that $L^{1/2}u=f$. Given an arbitrary $w \in \IW^{1,2}(\Omega)$, we apply Lemma~\ref{Lem: Duality formula} with $v = (L^*)^{1/2}w$ and use \eqref{Extended square root i}, \eqref{Extended square root ii} to give
\begin{align*}
 |\dscal{\Lop u}{w}| \leq \|L^{1/2}u\|_p \|(L^*)^{1/2}w\|_{p'} \lesssim \|L^{1/2}u\|_p \|w\|_{\IW^{1,p'}} = \|f\|_p \|w\|_{\IW^{1,p'}}.
\end{align*}
Thus $\Lop u \in \IW^{-1,p}(\Omega)$, which implies $u\in \IW^{1,p}(\Omega)$ thanks to \eqref{Extended square root i}.
\end{proof}
%%%%%%%%%%%%%%%%%%%%%%%%%%%%%%%%%%%%%%%%%%%%%%%%%%%%%%%%%%%%%%%%%%%%%%%%%%%%%%%%%%%%%%%%%%%%%%%%%%%%%%%%%%%%%%%%%%%%%%%%%%%%%%%%%%%%%%%%%%%%%%%%%%%%%%%%%%%%%%%%%%%%
\section{Holomorphic dependence: Proof of Theorem~\ref{Thm: Holomorphy}}
\label{Sec: Holomorphy Lp}

We remind the reader that for $p=2$ the holomorphic dependence of the $\H^\infty$-calculus and the square root as stated in parts \eqref{Holo1} and \eqref{Holo2} of Theorem~\ref{Thm: Holomorphy} have already been obtained in Section~\ref{Sec: L2}. Compare with Corollaries~\ref{Cor: Functional calculus L L2} and \ref{Cor: Square root holomorphic L2}. For $p$ as specified in \eqref{Holo1} and \eqref{Holo2} of Theorem~\ref{Thm: Holomorphy} they are at least locally bounded on $O$ as we have proved in Theorems~\ref{Thm: Hinfty in Lp}, \ref{Thm: Square root solution Lp}, and \ref{Thm: Extended square root solution Lp}, respectively.

For $p \in (1,\infty)$ let now $X_p$ denote either of the spaces $\L^p(\Omega)^m$ and $\IW^{1,p}(\Omega)$. Then $X_p$ is reflexive and $X_p \cap X_2$ is dense in both $X_2$ and $X_p$. Hence, all statements of Theorem~\ref{Thm: Holomorphy} are instances of the subsequent abstract result on vector-valued holomorphic functions. We refer to \cite[App.~A]{ABHN} for general background. We say that two Banach spaces are \emph{compatible} if they are included in the same linear Hausdorff space.

\begin{lemma}
\label{Lem: Holomorphic extension}
Let $(X_1, X_2)$ and $(Y_1, Y_2)$ be two pairs of compatible complex Banach spaces. Suppose that $X_1 \cap X_2$ is dense in both $X_1$ and $X_2$, $Y_ 1\cap Y_2$ is dense in both $Y_1$ and $Y_2$, and that $Y_2$ is reflexive. Let $O \subseteq \IC$ be an open set and $f: O \to \Lop(X_1,Y_1)$ holomorphic. If there is a finite constant $C$ such that
\begin{align*}
 \|f(z)x\|_{Y_2} \leq C \|x\|_{X_2} \qquad (z \in O, \, x \in X_1 \cap X_2),
\end{align*}
then each $f(z)$ extends from $X_1 \cap X_2$ to a bounded operator $X_2 \to Y_2$, denoted also $f(z)$, and $f: O \to \Lop(X_2, Y_2)$ is holomorphic.
\end{lemma}

\begin{proof}
The extension $f: O \to \Lop(X_2, Y_2)$ comes from the assumption and since $X_1 \cap X_2$ is dense in $X_2$. For clarity let us call it $g$ just in this proof. Our assumption guarantees $\|g(z)\|_{\Lop(X_2,Y_2)} \leq C$ for all $z \in O$, that is, $g$ is bounded. 

Next we shall prove that the intersection $Y_1^* \cap Y_2^*$ has a meaning and is a dense subspace of $Y_2^*$. Since $Y_1$ and $Y_2$ are compatible, we can consider their sum
\begin{align*}
 Y_1 + Y_2 = \{y_1 + y_2 : y_1 \in Y_1, \, y_2 \in Y_2\}, \qquad \|y\|_{Y_1 + Y_2} = \inf_{\substack{y_1 \in Y_1, \, y_2 \in Y_2 \\ y = y_1 + y_2}} \|y_1\|_{Y_1} + \|y_2\|_{Y_2}.
\end{align*}
This is again a Banach space \cite[Sec.~1.2.1]{Triebel}. Since $Y_1 \cap Y_2$ is dense in both $Y_1$ and $Y_2$, it is also dense in $Y_1 + Y_2$. This justifies to interpret, by restriction of functionals, the inclusions 
\begin{align*}
  (Y_1 + Y_2)^* \subseteq Y_2^* \qquad \text{and} \qquad (Y_1 + Y_2)^* \subseteq Y_1^* \cap Y_2^*
\end{align*}
within the ambient space $(Y_1 \cap Y_2)^*$. The first inclusion is dense since $Y_2$ is reflexive: Any functional on $Y_2^*$ that annihilates $(Y_1 + Y_2)^*$ is given by evaluation at some element of $Y_2$, which therefore has to vanish in $Y_1 + Y_2$. The second one is even an equality: Every $y^* \in Y_1^* \cap Y_2^*$ satisfies $|y^*(y)| \leq \max\{\|y^*\|_{Y_1^*}, \|y^*\|_{Y_2^*} \} \|y\|_{Y_1 + Y_2}$ for all $y \in Y_1 \cap Y_2$ and hence extends by density to a functional on $Y_1 + Y_2$. The intermediate claim follows from these two observations.

By the dense inclusions just alluded to, we can compute the norm of any $T \in \Lop(X_2,Y_2)$ as
\begin{align*}
 \|T\|_{\Lop(X_2,Y_2)} = \sup_{\substack{x \in X_1 \cap X_2 \\ \|x\|_{X_2} = 1}} \sup_{\substack{y^* \in Y_1^* \cap Y_2^* \\ \|y^*\|_{Y_2^*} = 1}} \left|\dscal{y^*}{Tx}_{Y_2^*,Y_2} \right|.
\end{align*}
Since $g$ is bounded, the weak-strong principle for holomorphic functions entails that holomorphy of $g$ is equivalent to holomorphy of all functions $z \mapsto \dscal{y^*}{g(z) x}_{Y_2^*,Y_2}$, where $x\in X_1 \cap X_2$ and $y \in Y_1^* \cap Y_2^*$, see \cite[Prop.~A.3]{ABHN}. But the latter follows from the holomorphy of $f$ since we have 
\begin{align*}
 \dscal{y^*}{g(z)x}_{Y_2^*, Y_2} = \dscal{y^*}{f(z)x}_{Y_1^*, Y_1}
\end{align*}
by construction.
\end{proof}

%%%%%%%%%%%%%%%%%%%%%%%%%%%%%%%%%%%%%%%%%%%%%%%%%%%%%%%%%%%%%%%%%%%%%%%%%%%%%%%%%%%%%%%%%%%%%%%%%%%%%%%%%%%%%%%%%%%%%%%%%%%%%%%%%%%%%%%%%%%%%%%%%%%%%%%%%%%%%%%%%%%%


\begin{thebibliography}{11}

\bibitem{Achache}
\textsc{M.~Achache} and \textsc{E.-M.~Ouhabaz}.
\newblock {\em Lions' maximal regularity problem with $H^{1/2}$-regularity in time\/}.
\newblock J. Differential Equations \textbf{266} (2019), no.~6, 3654--3678.

\bibitem{ABHN}
\textsc{W.~Arendt}, \textsc{C.J.K. Batty}, \textsc{M.~Hieber}, and \textsc{F.~Neubrander}.
\newblock Vector-valued {L}aplace {T}ransforms and {C}auchy {P}roblems. Monographs in Mathematics, vol.~96,
\newblock Birkh{\"a}user, Basel-Boston-Berlin, 2001.

\bibitem{ADLO13}
\textsc{W.~Arendt}, \textsc{D.~Dier}, \textsc{H.~Laasri}, and \textsc{E.-M.~Ouhabaz}.
\newblock {\em Maximal regularity for evolution equations governed by non-autonomous forms\/}.
\newblock Adv. Differential Equations \textbf{19} (2014), no.~11-12, 1043--1066.

\bibitem{Arendt-terElst}
\textsc{W.~Arendt} and \textsc{A.F.M.~ter Elst}.
\newblock{\em Gaussian estimates for second order elliptic operators with boundary conditions\/}.
J. Operator Theory \textbf{38} (1997), no.~1, 87--130.

\bibitem{Riesz-Transforms}
\textsc{P.~Auscher}.
\newblock {\em On necessary and sufficient conditions for {$L^p$}-estimates of
  {R}iesz transforms associated to elliptic operators on {$\mathbb{R}^n$} and
  related estimates\/}.
\newblock Mem. Amer. Math. Soc. \textbf{186} (2007), no.~871.

\bibitem{ABHR}
\textsc{P.~Auscher}, \textsc{N.~Badr}, \textsc{R.~{Haller-Dintelmann}}, and \textsc{J.~Rehberg}.
\newblock {\em The square root problem for second-order divergence form
  operators with mixed boundary conditions on ${L}^p$\/}. J. Evol. Eq. \textbf{15} (2015), no.~1, 165--208.

\bibitem{Kato-Square-Root-Proof}
\textsc{P.~Auscher}, \textsc{S.~Hofmann}, \textsc{M.~Lacey}, \textsc{A.~M\textsuperscript{c}Intosh}, and \textsc{
  P.~Tchamitchian}.
\newblock {\em The solution of the {K}ato square root problem for second order
  elliptic operators on {$\mathbb{R}^n$}\/}.
\newblock Ann. of Math. (2) \textbf{156} (2002), no.~2, 633--654.

\bibitem{Auscher-McIntosh-Nahmod}
\textsc{P.~Auscher}, \textsc{A.~M\textsuperscript{c}Intosh}, and \textsc{A.~Nahmod}.
\newblock {\em The square root problem of {K}ato in one dimension, and first order elliptic systems\/}.
\newblock Indiana Univ. Math. J. \textbf{46} (1997), no.~3, 659--695.

\bibitem{Asterisque}
\textsc{P.~Auscher} and \textsc{P.~Tchamitchian}.
\newblock {\em Square root problem for divergence operators and related
  topics\/}.
\newblock Ast{\'e}risque  (1998), no.~249.

\bibitem{ATp}
\textsc{P.~Auscher} and \textsc{P.~Tchamitchian}.
\newblock {\em Square roots of elliptic second order divergence operators on strongly {L}ipschitz domains: {$L^p$} theory\/}.
\newblock Math. Ann. \textbf{320} (2001), no.~3, 577--623.

\bibitem{AT2}
\textsc{P.~Auscher} and \textsc{P.~Tchamitchian}.
\newblock {\em Square roots of elliptic second order divergence operators on
  strongly {L}ipschitz domains: {$L^2$} theory\/}.
\newblock J. Anal. Math. \textbf{90} (2003), 1--12.

\bibitem{AKM-QuadraticEstimates}
\textsc{A.~Axelsson}, \textsc{S.~Keith}, and \textsc{A.~M\textsuperscript{c}Intosh}.
\newblock {\em Quadratic estimates and functional calculi of perturbed {D}irac
  operators\/}.
\newblock Invent. Math. \textbf{163} (2006), no.~3, 455--497.

\bibitem{AKM}
\textsc{A.~Axelsson}, \textsc{S.~Keith}, and \textsc{A.~M\textsuperscript{c}Intosh}.
\newblock {\em The {K}ato square root problem for mixed boundary value
  problems\/}.
\newblock J. London Math. Soc. (2) \textbf{74} (2006), no.~1, 113--130.

\bibitem{Bennett-Sharpley}
\textsc{C.~Bennett} and \textsc{R.~Sharpley}.
\newblock Interpolation of {O}perators. Pure and Applied
  Mathematics, vol.~129,
\newblock Academic Press, Boston MA, 1988.

\bibitem{BK}
\textsc{S.~Blunck} and \textsc{P.C.~Kunstmann}.
\newblock{\em Calder\'{o}n-Zygmund theory for non-integral operators and the $H^\infty$ functional calculus\/}.
Rev. Mat. Iberoamericana \textbf{19} (2003), no. 3, 919--942.

\bibitem{BK2}
\textsc{S.~Blunck} and \textsc{P.C.~Kunstmann}.
\newblock {\em Weak type $(p,p)$ estimates for Riesz transforms\/}. Math. Z. \textbf{247} (2004), no.~1, 137--148.

\bibitem{Bonifacius-Neitzel}
\textsc{L.~Bonifacius} and \textsc{I.~Neitzel}.
\newblock {\em Second Order Optimality Conditions for Optimal Control of Quasilinear Parabolic Equations\/}.
\newblock{Math. Control Relat. Fields \textbf{8}
(2018), no.~1, 1--34.}


\bibitem{Brewster}
\textsc{K.~Brewster}, \textsc{D.~Mitrea}, \textsc{I.~Mitrea}, and \textsc{M.~Mitrea}.
\newblock {\em Extending {S}obolev functions with partially vanishing traces
  from locally {$(\varepsilon,\delta)$}-domains and applications to mixed
  boundary problems\/}. J. Funct. Anal. \textbf{266} (2014), no.~7, 4314--4421.

\bibitem{Crouzeix-Delyon}
\textsc{M.~Crouzeix} and \textsc{B.~Delyon}.
\newblock {\em Some estimates for analytic functions of strip or sectorial operators\/}.
Arch. Math. (Basel) \textbf{81} (2003), no. 5, 559--566. 

\bibitem{Davies}
\textsc{E.B.~Davies}.
\newblock{\em Uniformly elliptic operators with measurable coefficients\/}.
J. Funct. Anal. \textbf{132} (1995), no.~1, 141--169.

\bibitem{Davies-CE}
\textsc{E.B.~Davies}.
\newblock{\em Limits on {$L^p$} regularity of self-adjoint elliptic operators\/}.
J. Differential Equations \textbf{135} (1997), no.~1, 83--102.

\bibitem{Dier-Zacher}
\textsc{D.~Dier} and \textsc{R.~Zacher}. 
\emph{Non-autonomous maximal regularity in Hilbert spaces}. 
J. Evol. Equ. \textbf{17} (2017), no.~3, 883--907.

\bibitem{DJT}
\textsc{J.~Diestel}, \textsc{H.~Jarchow}, and \textsc{A.~Tonge}.
\newblock Absolutely {S}umming {O}perators. Cambridge Studies in Advanced Mathematics, vol.~43,
\newblock Cambridge University Press, Cambridge, 1995.

\bibitem{Disser-terElst-Rehberg}
\textsc{K.~Disser}, \textsc{A.F.M.~ter Elst}, and \textsc{J.~Rehberg}.
\newblock {\em On maximal parabolic regularity for non-autonomous parabolic operators\/}.
J. Differential Equations \textbf{262} (2017), no.~3, 2039--2072.

\bibitem{Disser-terElst-Rehberg2}
\textsc{K.~Disser}, \textsc{A.F.M.~ter Elst}, and \textsc{J.~Rehberg}.
\newblock {\em H\"older estimates for parabolic operators on domains with rough boundary\/}.
Ann. Sc. Norm. Super. Pisa Cl. Sci. (5) \textbf{17} (2017), no.~1, 65--79.

\bibitem{Dore-Venni}
\textsc{G.~Dore} and \textsc{A.~Venni}.
\newblock{\em On the closedness of the sum of two closed operators\/}.
Math. Z. \textbf{196} (1987), no.~2, 189--201.

\bibitem{Duo}
\textsc{J.~Duoandikoetxea}.
\newblock Fourier {A}nalysis. Graduate Studies in Mathematics, vol.~29,
\newblock American Mathematical Society, Providence RI, 2001.

\bibitem{Duong-Robinson}
\textsc{X.T.~Duong} and \textsc{D.~Robinson}.
\newblock{\em Semigroup kernels, Poisson bounds, and holomorphic functional calculus\/}.
J. Funct. Anal. \textbf{142} (1996), no.~1, 89--128.

\bibitem{EigeneDiss}
\textsc{M.~Egert}.
\newblock{On {K}ato's conjecture and mixed boundary conditions\/}.
PhD Thesis, Sierke Verlag, G\"ottingen, 2015.

\bibitem{Hardy-Poincare}
\textsc{M.~Egert}, \textsc{R.~Haller-Dintelmann}, and \textsc{
  J.~Rehberg}.
\newblock {\em Hardy's inequality for functions vanishing on a part of the
  boundary\/}. Potential Anal. \textbf{43} (2015), no.~1, 49--78.

\bibitem{Laplace-Extrapolation}
\textsc{M.~Egert}, \textsc{R.~{Haller-Dintelmann}}, and \textsc{P.~Tolksdorf}.
\newblock {\em The {K}ato {S}quare {R}oot {P}roblem follows from an extrapolation property of the {L}aplacian\/}.
\newblock Publ. Math. \textbf{60} (2016), no.~2, 451--483.

\bibitem{Darmstadt-KatoMixedBoundary}
\textsc{M.~Egert}, \textsc{R.~{Haller-Dintelmann}}, and \textsc{P.~Tolksdorf}.
\newblock {\em The {K}ato {S}quare {R}oot {P}roblem for mixed boundary
  conditions\/}. J. Funct. Anal. \textbf{267} (2014), no.~5, 1419--1461.
  
\bibitem{ET}
\textsc{M.~Egert} and \textsc{P.~Tolksdorf}.
\newblock {\em Characterizations of Sobolev functions that vanish on a part of the boundary\/}.
Discrete Contin. Dyn. Syst. Ser. S \textbf{10} (2017), no.~4, 729--743.

\bibitem{Evans}
\textsc{L.C.~Evans}.
\newblock Partial {D}ifferential {E}quations. Graduate Studies in Mathematics, vol.~19,
\newblock American Mathematical Society, Providence RI, 1998.

\bibitem{Fackler-LpFrac}
\textsc{S.~Fackler}
\newblock {\em Non-Autonomous Maximal $L^p$-Regularity under Fractional Sobolev Regularity in Time\/}~.
\newblock Anal. PDE \textbf{11} (2018), no.~5, 1143--1169.

\bibitem{Fackler-LpRough}
\textsc{S.~Fackler}.
\newblock {\em Non-Autonomous Maximal $L^p$-Regularity for Rough Divergence Form Elliptic Operators\/}~.
\newblock Available at \url{https://arxiv.org/abs/1511.06207}.

\bibitem{Grafakos}
\textsc{L.~Grafakos}.
\newblock Classical {F}ourier {A}nalysis.
Graduate Texts in Mathematics, vol.~249,
\newblock Springer, New York, 2008.

\bibitem{Haase}
\textsc{M.~Haase}.
\newblock The {F}unctional {C}alculus for {S}ectorial {O}perators. Operator Theory: Advances and Applications, vol.~169,
\newblock Birkh{\"a}user, Basel, 2006.

\bibitem{HJKR}
\textsc{R.~{Haller-Dintelmann}}, \textsc{A.~Jonsson}, \textsc{D.~Knees}, and \textsc{J.~Rehberg}.
\newblock {\em Elliptic and parabolic regularity for second-order divergence operators with mixed boundary conditions\/}.
Math. Meth. Appl. Sci. \textbf{39} (2016), no.~17, 5007--5026.

\bibitem{RobertJDE}
\textsc{R.~Haller-Dintelmann} and \textsc{J.~Rehberg}.
\newblock{\em Maximal parabolic regularity for divergence operators including mixed boundary conditions\/}.
\newblock J. Differential Equations \textbf{247} (2009), no.~5, 1354--1396.

\bibitem{HMMc}
\textsc{S.~Hofmann}, \textsc{S.~Mayboroda}, and \textsc{A.~McIntosh}.
\newblock{\em Second order elliptic operators with complex bounded measurable coefficients in $L^p$, Sobolev and Hardy spaces\/}.
\newblock Ann. Sci. \'Ec. Norm. Sup\'er. (4) \textbf{44} (2011), no. 5, 723--800.
 

\bibitem{Hornung-Weis}
\textsc{L.~Hornung} and \textsc{L.~Weis}
\newblock {\em Quasilinear parabolic stochastic evolution equations via maximal $L^p$-regularity\/}.
Potential Anal. \textbf{50} (2019), no.~2, 279--326.

\bibitem{Kato}
\textsc{T.~Kato}.
\newblock Perturbation {T}heory for {L}inear {O}perators. Classics in Mathematics, 
\newblock Springer, Berlin, 1995.

\bibitem{LSV}
\textsc{V.~Liskevich}, \textsc{Z.~Sobol}, and \textsc{H.~Vogt}.
\newblock {\em On the $L_p$-theory of $C_0$-semigroups associated with second-order elliptic operators. {II}\/}.
J. Funct. Anal. \textbf{193} (2002), no.~1, 55--76.

\bibitem{McIntosh}
\textsc{A.~M\textsuperscript{c}Intosh}.
\newblock {\em Operators which have an {$H^\infty$} functional calculus\/}.
\newblock In Miniconference on operator theory and partial differential
  equations, Proc. Centre Math. Anal. Austral. Nat. Univ., vol.~14,
  Austral. Nat. Univ., Canberra, 1986, 210--231.

\bibitem{MM-Lame}
\textsc{M.~Mitrea} and \textsc{S.~Monniaux}.
\newblock{\em Maximal regularity for the {L}am\'e system in certain classes of non-smooth domains\/}.
J. Evol. Equ. \textbf{10} (2010), no.~4, 811--833.

\bibitem{Nirenberg}
\textsc{L.~Nirenberg}.
\newblock{\em On elliptic partial differential equations\/}. Ann. Scuola Norm. Sup. Pisa (3) \textbf{13} (1959), 115--162.

\bibitem{Ouhabaz}
\textsc{E.-M. Ouhabaz}.
\newblock Analysis of {H}eat {E}quations on {D}omains\/.
\newblock London Mathematical Society Monographs Series, vol.~31, 
Princeton University Press, Princeton NJ, 2005.

\bibitem{Shamir-Counterexample}
\textsc{E.~Shamir}.
\newblock {\em Regularization of mixed second-order elliptic problems\/}.
\newblock Israel J. Math. \textbf{6} (1968), 150--168.

\bibitem{Shen-Resolvent}
\textsc{Z.~Shen}.
\newblock{\em Resolvent estimates in {$L^p$} for elliptic systems in {L}ipschitz domains\/}.
J. Funct. Anal. \textbf{133} (1995), no.~1, 224--251.

\bibitem{Shen-Riesz}
\textsc{Z.~Shen}.
\newblock{\em Bounds of {R}iesz transforms on {$\L^p$} spaces for second order elliptic operators\/}.
Ann. Inst. Fourier (Grenoble) \textbf{55} (2005), no.~1, 173--197.

\bibitem{Sneiberg-Original}
\textsc{I.~{\u{S}}ne{\u{\ii}}berg}.
\newblock {\em Spectral properties of linear operators in interpolation
  families of {B}anach spaces\/}.
\newblock Mat. Issled. \textbf{9} (1974), no.~2, 214--229, 254--255.

\bibitem{Patrick}
\textsc{P.~Tolksdorf}.
\newblock{\em $\mathcal{R}$-sectoriality of higher-order elliptic systems on general bounded domains\/}.
\newblock J. Evol. Equ. \textbf{18} (2018), no.~2, 323--349.

\bibitem{Triebel}
\textsc{H.~Triebel}.
\newblock Interpolation {T}heory, {F}unction {S}paces, {D}ifferential
  {O}perators.  North-Holland Mathematical Library, vol.~18,
\newblock North-Holland Publishing, Amsterdam, 1978.

\bibitem{Vogt-Voigt}
\textsc{H.~Vogt} and \textsc{J.~Voigt}.
\newblock {\em Holomorphic families of forms, operators and  $C_0$-semigroups\/}.
\newblock Monatsh. Math. \textbf{187} (2018), no.~2, 375--380.

\bibitem{Wei-Zhang}
\textsc{W.~Wei} and \textsc{Z.~Zhang}.
\newblock{\em {$L^p$} resolvent estimates for constant coefficient elliptic systems on {L}ipschitz domains\/}.
J. Funct. Anal. \textbf{267} (2014), no.~9, 3262--3293.
\end{thebibliography}
\end{document}